\documentclass[times,final]{elsarticle}

\usepackage{jcomp}

\usepackage{framed,multirow}
\usepackage{hyperref} 
\usepackage{algpseudocode}
\usepackage{amssymb,amsthm,amsmath,bm,amsfonts}
\DeclareMathAlphabet{\mathbbm}{U}{bbm}{m}{n}

\usepackage{dsfont}
\usepackage{latexsym}
\usepackage{subcaption}

\usepackage{url}
\usepackage{xcolor}
\usepackage{mathtools}
\usepackage{cleveref}
\usepackage{showkeys}
\usepackage{multirow}
\usepackage{booktabs}
\usepackage{enumitem}  
\usepackage{array}
\usepackage{cancel}
\usepackage{algorithm}
\usepackage{algpseudocode}
\usepackage{mathabx} 

\journal{Paper Submitted to J. Comput. Phys.}

\usepackage{tikz}
\usetikzlibrary{positioning}
\usetikzlibrary{shapes, arrows}
\usetikzlibrary{arrows.meta}
\tikzstyle{terminator} = [rectangle, draw, text centered, rounded corners]
\tikzstyle{process} = [rectangle, draw, text centered]
\tikzstyle{decision} = [diamond, draw, text centered]
\tikzstyle{data}=[trapezium, draw, text centered, trapezium left angle=60, trapezium right angle=120]
\tikzstyle{connector} = [draw, -{latex[length=2mm]}, thick, blue]
\tikzstyle{connector2} = [draw, -{latex[length=2mm]}, thick, green]
\pgfdeclarelayer{background}
\pgfsetlayers{background,main}

\newtheorem{theorem}{Theorem}[section]
\newtheorem{lemma}[theorem]{Lemma}
\newtheorem{proposition}[theorem]{Proposition}

\newtheorem{remark}{Remark}[section]

\numberwithin{equation}{section}

\theoremstyle{definition}
\newtheorem{exa}[remark]{Example}

\theoremstyle{remark}

\newcommand{\BU}{\mathbf{U}}
\newcommand{\hBU}{\widehat{\BU}}
\newcommand{\tBU}{\widetilde{\BU}}
\newcommand{\cBU}{\widecheck{\BU}}
\newcommand{\BF}{\mathbf{F}}
\newcommand{\hBF}{\widehat{\BF}}

\newcommand{\cBF}{\widecheck{\BF}}
\newcommand{\Bn}{\mathbf{n}}
\newcommand{\oBU}{\overline{\BU}}

\newcommand{\rz}{z_\star}

\newcommand{\toBU}{
	\mathrlap{
		\raisebox{1.2 pt}{$\overline{\phantom{\mathbf{U}}}$}
	}
	\mathrlap{
		\raisebox{-0.9 pt}{$\widetilde{\phantom{\mathbf{U}}}$}
	}
	\mathbf{U}
}

\newcommand{\hoBU}{
	\mathrlap{
		\raisebox{1.2 pt}{$\overline{\phantom{\mathbf{U}}}$}
	}
	\mathrlap{
		\raisebox{-0.9 pt}{$\widehat{\phantom{\mathbf{U}}}$}
	}
	\mathbf{U}
}

\newcommand{\hcBU}{
	\mathrlap{
		\raisebox{0.6 pt}{$\widecheck{\phantom{\mathbf{U}}}$}
	}
	\mathrlap{
		\raisebox{-0.6 pt}{$\widehat{\phantom{\mathbf{U}}}$}
	}
	\mathbf{U}
}

\newcommand{\hcr}{
	\mathrlap{
		\raisebox{0.5 pt}{$\check{\phantom{\rho}}$}
	}
	\mathrlap{
		\raisebox{-0.5 pt}{$\hat{\phantom{\rho}}$}
	}
	\rho
}

\newcommand{\hcy}{
	\mathrlap{
		\raisebox{0.5 pt}{$\check{\phantom{y}}$}
	}
	\mathrlap{
		\raisebox{-0.5 pt}{$\hat{\phantom{y}}$}
	}
	y
}

\newcommand{\hcz}{
	\mathrlap{
		\raisebox{0.5 pt}{$\check{\phantom{z}}$}
	}
	\mathrlap{
		\raisebox{-0.5 pt}{$\hat{\phantom{z}}$}
	}
	z
}

\newcommand{\hcv}{
	\mathrlap{
		\raisebox{0.5 pt}{$\check{\phantom{v}}$}
	}
	\mathrlap{
		\raisebox{-0.5 pt}{$\hat{\phantom{v}}$}
	}
	v
}

\newcommand{\hcs}{
	\mathrlap{
		\raisebox{0.5 pt}{$\check{\phantom{s}}$}
	}
	\mathrlap{
		\raisebox{-0.5 pt}{$\hat{\phantom{s}}$}
	}
	s
}

\newcommand{\hcBn}{
	\mathrlap{
		\raisebox{0.5 pt}{$\check{\phantom{\Bn}}$}
	}
	\mathrlap{
		\raisebox{-0.5 pt}{$\hat{\phantom{\Bn}}$}
	}
	\Bn
}

\newcommand{\hbr}{
	\mathrlap{
		\raisebox{1.4 pt}{$\bar{\phantom{\rho}}$}
	}
	\mathrlap{
		\raisebox{-0.4 pt}{$\hat{\phantom{\rho}}$}
	}
	\rho
}

\newcommand{\hbz}{
	\mathrlap{
		\raisebox{1.4 pt}{$\bar{\phantom{z}}$}
	}
	\mathrlap{
		\raisebox{-0.4 pt}{$\hat{\phantom{z}}$}
	}
	z
}
\begin{document}

\verso{W. Chen, S. Cui, K. Wu, T. Xiong}

\begin{frontmatter}
	
	\title {
		Bound-preserving OEDG schemes for Aw--Rascle--Zhang traffic models \\ on networks \tnoteref{tnote1}
	}
	
	\tnotetext[tnote1]{
		The work of Wei Chen and Tao Xiong was partially supported by National Key R\&D Program of China No. 2022YFA1004500, NSFC grant No. 11971025, and NSF of Fujian Province grant No. 2023J02003.
		The work of Shumo Cui and Kailiang Wu was partially supported by Shenzhen Science and Technology Program grant No. RCJC20221008092757098.
		The work of Kailiang Wu was also partially supported by NSFC grant No. 12171227.
	}
	
	\author[XU]{Wei \snm{Chen}}
	\ead{weichenmath@stu.xmu.edu.cn}
	
	\author[ICM,SUSTech]{Shumo \snm{Cui}\corref{cor1}}
	\ead{cuism@sustech.edu.cn}
	
	\author[SUSTech,ICM,NCAMS]{Kailiang \snm{Wu}}
	\ead{wukl@sustech.edu.cn}
	
	\author[USTC]{Tao \snm{Xiong}\corref{cor1}}
	\ead{txiong@xmu.edu.cn}
	
	\address[XU]{School of Mathematical Sciences, Xiamen University, Xiamen, Fujian 361005, China.}
	\address[SUSTech]{Department of Mathematics, Southern University of Science and Technology, Shenzhen 518055, China.}
	\address[ICM]{Shenzhen International Center for Mathematics, Southern University of Science and Technology, Shenzhen 518055, China.}
	\address[NCAMS]{National Center for Applied Mathematics Shenzhen (NCAMS), Shenzhen 518055, China.}
	\address[USTC]{School of Mathematical Sciences, University of Science and Technology of China, Hefei, Anhui 230026, China.}
	
	\cortext[cor1]{Corresponding author.}
	

\begin{abstract}
Physical solutions to the widely used Aw--Rascle--Zhang (ARZ) traffic model and the adapted pressure (AP) ARZ model should satisfy the positivity of density, the minimum and maximum principles with respect to the velocity $v$ and other Riemann invariants. 
Many numerical schemes suffer from instabilities caused by violating these bounds, and the only existing bound-preserving (BP) numerical scheme (for ARZ model) is random, only first-order accurate, and not strictly conservative.
This paper introduces arbitrarily high-order provably BP DG schemes for these two models, preserving all the aforementioned bounds except the maximum principle of $v$, which has been rigorously proven to conflict with the consistency and conservation of numerical schemes.
Although the maximum principle of $v$ is not directly enforced, we find that the strictly preserved maximum principle of another Riemann invariant $w$ actually enforces an alternative upper bound on $v$.
At the core of this work, analyzing and rigorously proving the BP property is a particularly nontrivial task: the Lax--Friedrichs (LF) splitting property, usually expected for hyperbolic conservation laws and employed to construct BP schemes, does not hold for these two models. To overcome this challenge, we formulate a generalized version of the LF splitting property, and prove it via the geometric quasilinearization (GQL) approach [Kailiang Wu and Chi-Wang Shu, \textit{SIAM Review}, 65: 1031--1073, 2023]. To suppress spurious oscillations in the DG solutions, we incorporate the oscillation-eliminating (OE) technique, recently proposed in [Manting Peng, Zheng Sun, and Kailiang Wu, \textit{Mathematics of Computation}, in press, https://doi.org/10.1090/mcom/3998], which is based on the solution operator of a novel damping equation.
Several numerical examples are included to demonstrate the effectiveness, accuracy, and BP properties of our schemes, with applications to traffic simulations on road networks. 
\end{abstract}
	
	
\end{frontmatter}

\section{Introduction}

Macroscopic traffic models play a crucial role in transportation management, prediction, and planning.
Rather than tracking individual vehicle movements, these models focus on describing the average properties of traffic flow, such as density and mean velocity.
Several macroscopic traffic models have been proposed, for example, the Lighthill--Whitham--Richards (LWR) \cite{richards1956shock,lighthill1955kinematic}, the Payne--Whitham (PW) \cite{payne1971model}, and the Aw--Rascle--Zhang (ARZ) \cite{zhang2002non,aw2000resurrection} models.
The LWR model is constituted by a single conservation law of vehicle density. This model assumes the traffic velocity solely depends on traffic density via a closure relationship.
The solution to the LWR model is always bounded by the minimum and maximum of initial condition. Such a characteristic makes the LWR model unsuitable for describing traffic instabilities, such as the phenomena of small traffic fluctuations developing into waves of larger magnitude without external disturbance.
The PW model \cite{payne1971model} introduces an additional equation of momentum, and the minimum and maximum principles no longer hold, allowing this model to represent the aforementioned traffic instabilities \cite{FlynnKasimovNaveRosalesSeibold2009}.
However, one of the eigenspeeds of the PW model is nonphysically large, potentially leading to unrealistic traffic behavior.
As an improvement over the PW model, Aw and Rascle \cite{aw2000resurrection}, and independently Zhang \cite{zhang2002non}, proposed the ARZ model, which resolves the issue of unrealistic eigenspeeds. This ARZ model was later extended and generalized in \cite{Greenberg2002,BerthelinDegondDelitalaRascle2008,gottlich2021second}. 

While these models provide valuable insights into traffic flows on single road segments, real-world traffic systems are composed of road networks, making the design and utilization of appropriate coupling conditions of great importance.
Some recent works in this respect include \cite{BressanCanicSunicaGaravelloHertyPiccoli2014,GaravelloHanPiccoli2016,GaravelloPiccoli2006}.
Considering the application of ARZ model on networks, various junction conditions have been proposed, e.g., \cite{herty2006optimization,garavello2006traffic,haut2007second,siebel2009balanced,gottlich2021second,kolb2018pareto,kolb2017capacity}. However, as outlined in \cite{herty2006coupling}, most abovementioned junction conditions are not compatible with the Lagrangian/microscopic form of the ARZ model. 
This incompatibility arises because these junction conditions often ignore the dependency of the pressure function on the Lagrangian marker and use the same pressure function for both the upstream and downstream sides of the junction.

\subsection{Adapted pressure ARZ model}\label{sec:88}

Recently, a new adapted pressure (AP) ARZ model and a new junction condition were proposed in \cite{gottlich2021second}, allowing for the assignment of different pressure functions for traffic before and after passing through a junction. 
The new model and the corresponding junction condition satisfy the conservation of mass and generalized momentum, and are consistent with the Lagrangian/microscopic forms. 
The AP ARZ model can be written as
\begin{equation}\label{system}
	\partial_t \BU + \partial_x \BF(\BU) = {\bf 0},
\end{equation}
where the conservative variable $\BU$ and the flux function $\BF$ are given by $\BU = (\rho, y, z)^\top$, $\BF(\BU) = v \BU$ with the density $\rho$ and the generalized momentum $y$. The ``driver behavior marker'' $w$ and velocity $v$ are respectively defined by
$w = w(\BU) := y/\rho$ and $v = v(\BU) := w - p(\rho, z)$.
Here, the pressure $p(\rho, z)$ describes the response of drivers to changes in the density of the vehicle ahead. In this paper, we will consider the following pressure function:
\begin{equation}\label{P}
	p(\rho, z) = \left\{
	\begin{aligned}
		& \frac{v_{\rm ref}}{\gamma} z^\kappa \rho^{\gamma-\kappa}
		=
		\frac{v_{\rm ref}}{\gamma} c^\kappa \rho^\gamma, \; && {\rm if} \; \gamma > 0, \\
		& v_{\rm ref} \log{\rho}, \; && {\rm if} \; \gamma = 0
	\end{aligned}
	\right.
\end{equation}
with a given reference velocity $v_{\rm ref}$ and parameter $\kappa \ge \gamma+1$.
Here, the variable $c = c(\BU) := z/\rho > 0$ is used as a label to distinguish distinct ``driver behaviors'', and such a label is transported passively with traffic flow. The eigenvalues of the Jacobian $\partial \BF/\partial \BU$ are $\lambda_1 = v-\rho p_\rho - z p_z$, $\lambda_2 = \lambda_3 = v$. 
If one sets $c \equiv 1$, then the AP ARZ model \eqref{system} degenerates into the original ARZ model \cite{aw2000resurrection}, and the pressure function \eqref{P} degenerates into the ones adopted in \cite{aw2000resurrection,BagneriniRascle2003,garavello2006traffic,herty2006optimization,haut2007second,chalons2007transport,FlynnKasimovNaveRosalesSeibold2009,betancourt2018random,buli2020discontinuous}.

The Riemann invariants of \eqref{system} are $v$, $w$, and $c$. It holds that the solution to the system \eqref{system} satisfies minimum and maximum principles with respect to these Riemann invariants {\em globally} (see \cite{BagneriniRascle2003}), namely, 
\begin{equation}\label{eq:328}
	\BU(x,t) \in \mathcal{I}^{\rm G},\qquad t > 0,
\end{equation}
where the global invariant domain $\mathcal{I}^{\rm G}$ is defined as
\begin{equation}\label{eq:332}
	\mathcal{I}^{\rm G}
	:= 
	\left\{ 
	\BU
	\in \mathbb{R}^3 
	: 
	v \in [v^{\rm G}_{\min}, v^{\rm G}_{\max}], 
	w \in [w^{\rm G}_{\min}, w^{\rm G}_{\max}],
	c \in [c^{\rm G}_{\min}, c^{\rm G}_{\max}] 
	\right\}
\end{equation}
with the global upper and lower bounds of Riemann invariants defined as
\begin{gather*}
	v^{\rm G}_{\min} := \min_x v(x,0), \quad w^{\rm G}_{\min} := \min_x w(x,0), \quad c^{\rm G}_{\min} := \min_x c(x,0), \\
	v^{\rm G}_{\max} := \max_x v(x,0), \quad w^{\rm G}_{\max} := \max_x w(x,0), \quad c^{\rm G}_{\max} := \max_x c(x,0).
\end{gather*}

The hyperbolic system \eqref{system} exhibits finite propagation speed, therefore, one may extend the global bound \eqref{eq:328} to a local version,
$\BU(x,t) \in \mathcal{I}^{\rm L}(x,t,t_0)$, $t > t_0 \geq 0$,
where the local invariant domain $\mathcal{I}^{\rm L}(x,t,t_0)$ is defined as
\begin{equation}\label{eq:154}
	\mathcal{I}^{\rm L}(x,t,t_0)
	:=
	\left\{ 
	\BU= \left( \rho, y, z \right)^\top \in \mathbb{R}^3 
	: 
	\renewcommand\arraystretch{1.2}
	\begin{array}{r}
		v \in [v^{\rm L}_{\min}(x,t,t_0), v^{\rm L}_{\max}(x,t,t_0)] \\  
		w \in [w^{\rm L}_{\min}(x,t,t_0), w^{\rm L}_{\max}(x,t,t_0)] \\
		c \in [c^{\rm L}_{\min}(x,t,t_0), c^{\rm L}_{\max}(x,t,t_0)] 
	\end{array}
	\right\},
\end{equation}
where the local bounds of Riemann invariants are defined as
\begin{align*}
	v^{\rm L}_{\min}(x,t,t_0) &:= \min_{\xi \in \mathcal{X}} v(\xi,t_0),
	& w^{\rm L}_{\min}(x,t,t_0) &:= \min_{\xi \in  \mathcal{X}} w(\xi,t_0),
	& c^{\rm L}_{\min}(x,t,t_0) &:= \min_{\xi \in \mathcal{X}} c(\xi,t_0), 
	\\
	v^{\rm L}_{\max}(x,t,t_0) &:= \max_{\xi \in  \mathcal{X}} v(\xi,t_0),
	& w^{\rm L}_{\max}(x,t,t_0) &:= \max_{\xi \in  \mathcal{X}} w(\xi,t_0),
	& c^{\rm L}_{\max}(x,t,t_0) &:= \max_{\xi \in \mathcal{X}} c(\xi,t_0).
\end{align*}
Here, the local domain of determination is defined as $\mathcal{X}:=[x-\alpha(t-t_0),x+\alpha(t-t_0)]$, and $\alpha$ denotes the maximum propagation speed. Please note that $\mathcal{I}^{\rm L} \subseteq \mathcal{I}^{\rm G}$.

\subsection{DG methods}

Solutions to the system \eqref{system}, as a nonlinear hyperbolic conservation law, may develop discontinuities in finite time, even from smooth initial conditions. The discontinuous nature of the solutions can cause spurious oscillations in numerical results, particularly in high-order schemes, leading to nonphysical wave structures or numerical instability. Therefore, when solving the system \eqref{system} numerically, it is crucial to choose appropriate numerical methods that can resolve discontinuous structures with high resolution and are free of spurious oscillations.

The discontinuous Galerkin (DG) method is a class of finite element methods for solving hyperbolic conservation laws. These methods employ discontinuous piecewise polynomial basis functions and offer several advantages, including local conservation, stencil compactness, simple boundary condition implementation, and $hp$-adaptivity. 
The DG methods have gained widespread popularity in high-performance computing applications, especially for large-scale problems. For parallel implementations, the DG methods have demonstrated parallel efficiencies of over 99\% for fixed meshes \cite{BiswasDevineFlaherty1994}. These features make the DG methods particularly promising candidates for numerically solving traffic flow simulations on complex road networks. 
To mitigate oscillations in DG methods, several strategies have been proposed. The first strategy employs limiters, such as total variation diminishing, total variation bounded, and weighted essentially non-oscillatory limiters; see, e.g., \cite{zhong2013simple,shu2009discontinuous,qiu2005runge}. The second strategy adds artificial viscosity terms with second or higher-order spatial derivatives to diffuse oscillations  \cite{hiltebrand2014entropy,huang2020adaptive,yu2020study,zingan2013implementation}.  
Recently, \cite{liu2022essentially,lu2021oscillation} introduced the essentially oscillation-free DG methods, which incorporate local damping terms to suppress spurious oscillations. 
Inspired by these works, Peng, Sun, and Wu \cite{peng2023oedg} developed a novel damping mechanism and the oscillation-eliminating DG (OEDG) methods, which are notable for their non-intrusive, scale- and evolution-invariant characteristics.
In addition to these advantages, the OEDG methods remain stable with normal time step sizes and do not require a characteristic decomposition procedure, thereby exhibiting remarkable simplicity and efficiency. 

\subsection{Bound-preserving schemes}\label{sec:195}
Besides the common characteristic of containing discontinuous structures, the solutions to hyperbolic conservation laws typically satisfy certain bounds or constraints, such as the constraints \eqref{eq:328} for system \eqref{system}. It is highly desirable, or even crucial, for numerical schemes to preserve these bounds, resulting in bound-preserving (BP) numerical methods. 
In recent decades, research on BP schemes for hyperbolic equations has gained considerable attention. This includes schemes that preserve minimum/maximum principles \cite{zhang2010maximum,xu2014parametrized,Lv2014EntropyboundedDG,AndersoDobrevKolevKuzminQuezadaRiebenTomov2017}, positivity \cite{zhang2011maximum,wu2018positivity,WuShu2019}, and other types of bounds \cite{wu2021minimum,WuTang2015,wu2017design,mabuza2018local,Dmitri2021}. Zhang and Shu developed a general framework for constructing high-order BP DG and finite-volume schemes in  \cite{zhang2011maximum,zhang2010maximum}.
Designing BP numerical schemes becomes significantly more demanding when the bounds to be preserved exhibit very complicated explicit formulas involving high nonlinearity, not to mention cases where the bounds can only be implicitly defined. To address this challenge, Wu and Shu \cite{wu2023geometric} introduced the geometric quasilinearization (GQL) framework, offering an effective new approach for BP analysis and design with nonlinear and complicated constraints.

For the original ARZ model, designing BP numerical schemes is of great importance but rather challenging. The main difficulty lies in the fact that its invariant domain is not convex, making most aforementioned BP techniques inapplicable. By combining the Godunov method with a random sampling strategy, Chalons and Goatin \cite{chalons2007transport} developed a first-order quasi-conservative random method for the original ARZ model that satisfies the \( v \) maximum principle, namely, \( v \le v_{\max} \). Thanks to the satisfaction of the \( v \) maximum principle, this method resolves the contact waves sharply and is free of nonphysical velocity overshoots, which are often observed in numerical results obtained using classic schemes, such as Godunov and Lax–Friedrichs (LF) schemes. 
While this random BP scheme is not strictly conservative, a natural question arises: Is it possible to design a strictly conservative BP method for the ARZ model that preserves \( v \le v_{\max} \)? By constructing a counter-example, Proposition 4.1 of \cite{betancourt2018random} provided a clear negative answer: no strictly conservative first-order scheme with a consistent two-point numerical flux can satisfy the \( v \) maximum principle. 
Considering the AP ARZ model as a generalization of the ARZ model, this conflict between satisfying \( v \le v_{\max} \) and maintaining conservation/consistency also applies. This leads to the new question: Is it possible to design BP schemes for the original and the AP ARZ models that satisfy the constraints in \eqref{eq:332} except \( v \le v_{\max} \), but adhere to another suitable upper bound for \( v \)?

In this paper, we develop BP-OEDG schemes for the original and the AP ARZ type models based on LF numerical flux, which preserve the maximum principles of $w$ and $c$, the minimum principle of $v$, and the positivity of $\rho$.
When designing BP numerical schemes with the LF flux preserving invariant domain $\mathcal{I}$ for hyperbolic conservation laws, the following property \cite{zhang2010b,WuTang2015} is often desired and employed as the core of BP analysis:
\begin{equation}\label{eq:309}
	\forall \, \BU \in \mathcal{I}
	\; \Longrightarrow \;
	\mathcal{S}^\pm_{\tt LF}(\BU) := \BU \pm \frac{\BF({\BU})}{\alpha} \in \mathcal{I},
\end{equation}
where $\alpha$ denotes an estimate of maximum propagation speed, and the above statement \eqref{eq:309} is typically called the LF splitting property \cite{WuTang2015} for $\mathcal{I}$.
Unfortunately, the statement \eqref{eq:309} does not generally hold for the system \eqref{system}--\eqref{P}, making the task of designing BP numerical schemes rather nontrivial.
To overcome this difficulty, we investigate and prove the following generalized statement:
\begin{equation}\label{eq:318}
	\forall \, \BU_1, \BU_2 \in \mathcal{I}
	\; \Longrightarrow \; 
	\mathcal{S}_{\tt GLF}(\BU_1,\BU_2) := 
	\frac12
	\left[
	\BU_1 + \frac{\BF({\BU_1})}{\alpha}
	+
	\BU_2 - \frac{\BF({\BU_2})}{\alpha}
	\right]
	\in \mathcal{I},
\end{equation}
which is referred to as the generalized  LF splitting property in the literature \cite{wu2018positivity}. 
By establishing this generalized property \eqref{eq:318}, we then design a provably BP-OEDG method for the original and the AP ARZ models. 
It is also worth mentioning that, due to the satisfaction of $w$ maximum principle, even though the $v$ maximum principle is not directly addressed in this study, our schemes effectively enforce an alternative upper bound on $v$ values, substantially mitigating the $v$ overshoots, especially in the vicinity of near-vacuum states.
The proposed BP-OEDG schemes are also applicable to the original ARZ model, which can be considered as a degenerate case of the AP ARZ model.

\subsection{Contributions}

The key contributions of this paper include: 
\begin{itemize}[leftmargin=3mm]
	
	\item Through systematical analysis, we rigorously prove several key properties of invariant domains induced by the above constraints, including convexity and the (generalized) LF splitting properties. These theoretical efforts are novel and nontrivial, involving both technical estimates and the GQL approach.
	
	\item Based on these efforts and findings, we develop the first-of-their-kind BP DG schemes for the original and the AP ARZ models, which preserve the maximum principles of all Riemann invariants, the minimum principle of all Riemann invariants except $v$, and the positivity of $\rho$. 
	
	\item Although the $v$ maximum principle is not directly addressed in this work, the strictly enforced maximum principle of $w$ gives an alternative upper bound of $v$ values, which substantially mitigates the velocity overshoots in the numerical results, especially in the vicinity of near-vacuum states. 
	
	\item In addition to the global invariant domains, we also provide a recipe to estimate local invariant domains. As demonstrated in the numerical examples, the locally BP-OEDG schemes are more robust than the globally BP ones, especially when the solution consists of waves with various magnitudes.
	
	\item The non-intrusive scale-invariant OE procedure \cite{peng2023oedg} is incorporated into the proposed DG schemes to suppress spurious oscillations, while maintaining the high-order accuracy of the DG schemes. 
	
\end{itemize}

This paper is structured as follows: 
Section \ref{sec:inv} presents various invariant domains and derive their key properties. 
\Cref{sec:OEDG} introduces the OEDG method for system \eqref{system}. 
Section \ref{sec:BP-OEDG} presents the high-order BP-OEDG schemes and the rigorous proof of their BP property. 
Section \ref{sec:num} presents a series of numerical examples, including applications on single road segment and networks, to demonstrate the accuracy, BP property,  and robustness of our schemes. 
As a graphic summary, a flowchart of our theoretical investigations is depicted in the following \Cref{fig:314}.

\begin{figure}[hb!]
	\centering
	\begin{tikzpicture}[scale = 0.8]
		\node [process, fill=green!20, inner xsep=2pt, inner ysep=2pt] at (-0.5 ,-0.4) (L22) {\footnotesize \textbf{\Cref{lem:466}}: invariant domains $\Omega_0, \Omega_1, \Omega_2$, and $\Omega$ are all convex};
		\node [process, fill=blue!20, inner xsep=2pt, inner ysep=2pt, text width=2.5cm, align=left] at (-6.4  ,-1.8) (L24) {\footnotesize \textbf{\Cref{L1}}: LF splitting property holds for $\Omega_1$};
		\node [process, fill=blue!20, inner xsep=2pt, inner ysep=2pt, text width=2.5cm, align=left] at (-6.4  ,-3.7) (L25) {\footnotesize \textbf{\Cref{lem:647}}: GQL linear equivalence of $\Omega_2$};
		\node [process, fill=blue!20, inner xsep=2pt, inner ysep=2pt, text width=2.9cm, align=left] at (-2.4  ,-1.8) (L27) {\footnotesize \textbf{\Cref{thm:1471}}: generalized LF splitting property holds for $\Omega$};
		\node [process, fill=blue!20, inner xsep=2pt, inner ysep=2pt, text width=2.9cm, align=left] at (-2.4  ,-3.7) (L26) {\footnotesize \textbf{\Cref{L2}}: generalized LF splitting property holds for $\Omega_2$};
		\node [process, fill=blue!20, inner xsep=2pt, inner xsep=2pt, inner ysep=2pt, text width=1.95cm, align=left] at (1.3  ,-2.6) (L410) {\footnotesize \textbf{\Cref{lem:2154}}: BP property of forward-Euler steps};
		\node [process, fill=red!20, inner xsep=2pt, inner ysep=2pt, text width=3.6cm, align=left] at (5.5  ,-1.8) (T411) {\footnotesize \textbf{\Cref{lem:2350}}: the proposed scheme is BP w.r.t. global invariant domain};
		\node [process, fill=red!20, inner xsep=2pt, inner ysep=2pt, text width=3.6cm, align=left] at (5.5  ,-3.7) (T412) {\footnotesize \textbf{\Cref{lem:2394}}: the proposed scheme is BP w.r.t. local invariant domain};
		\begin{pgfonlayer}{background}
			\path [connector] (L24) -- (L27);
			\path [connector] (L25) -- (L26);
			\path [connector] (L26) -- (L27);
			\path [connector] (L27) -- (L410);
			\path [connector] (L410) -- (T411);
			\path [connector] (L410) -- (T412);
			\path[connector2] (L22.west) -| (L24.north);
			\path[connector2] (L22) -| (L410.north);
			\path[connector2] (L22) -| (L27.north);
			\path[connector2] (L22.west) -| ++(-3.3,0) |- (L25.west);
			\path[connector2] (L22.east) -| (T411.north);
			\path[connector2] (L22.east) -| ++( 4.1,0) |- (T412.east);
			\path[connector2] (L22.west) -| ++(-3.3,-4.5) -| (L26.south);
		\end{pgfonlayer}
	\end{tikzpicture}
	\caption{\sf Flowchart of key theoretical findings towards the BP property of proposed schemes.}
	\label{fig:314}
\end{figure}
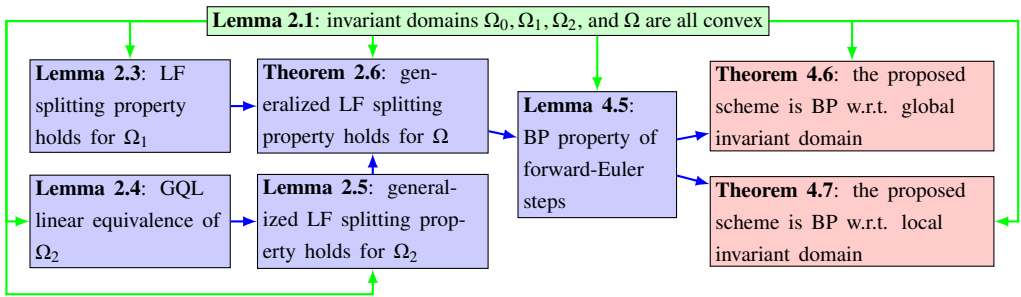

\section{Invariant domains and their properties}\label{sec:inv}
In this paper, we consider the following invariant domains of the system \eqref{system}:
\begin{align}
	\Omega_0 &:= \left\{ \BU  \in \mathbb{R}^3 : \rho > 0, \; z > 0 \right\}, \nonumber \\
	\Omega_1 &:= \left\{ \BU  \in \mathbb{R}^3 : \rho > 0, \; w \in [w_{\min}, w_{\max}], \; c \in [c_{\min}, c_{\max}] \right\}, 
	\nonumber \\
	\Omega_2 &:= \left\{ \BU  \in \mathbb{R}^3 : \rho > 0, \; z > 0, \; v \geq v_{\min} \right\}, \label{eq:O2} \\
	\Omega   &:= \Omega_1 \cap \Omega_2
	= 
	\left\{
	\BU  \in \mathbb{R}^3
	:
	\begin{array}{r}
		\rho > 0, z > 0, w \in \left[ w_{\min}, w_{\max} \right] \\
		v \ge v_{\min},	c \in \left[ c_{\min}, c_{\max} \right]
	\end{array}
	\right\}.  \label{eq:270}
\end{align}
Towards designing BP-OEDG schemes for the system \eqref{system}, we will derive several key properties of these invariant domains, as outlined in \Cref{fig:314}.

\begin{remark}\label{rmk:435}
	Although the constraint $v \le v_{\max}$ is not directly addressed, the considered constraint $w \le w_{\max}$ actually implies another upper bound for $v$:
	\[
	v = w-p(\rho,z) \le w_{\max} - p(\rho,z)
	\le
	\begin{dcases}
		w_{\max} 
		& 
		{\rm if }~ \gamma > 0,
		\\
		w_{\max} - v_{\rm ref} \log (\min_x \rho)
		& {\rm if }~ \gamma = 0.
	\end{dcases}
	\]
	As one will see in the numerical experiments, this upper bound on $v$ can also help mitigate overshoots in $v$ in the numerical results. 
\end{remark}

\begin{lemma}\label{lem:466}
	The invariant domains $\Omega_0$, $\Omega_1$, $\Omega_2$, and $\Omega$ are all convex sets.
\end{lemma}

\begin{proof}
	{\sf Part I}: 
	The set $\Omega_0$ is clearly convex.
	
	\noindent{\sf Part II}: 
	For any two states $\BU_0 = (\rho_0,y_0,z_0)^\top$ and $\BU_1 = (\rho_1,y_1,z_1)^\top \in \Omega_1$, and any $\theta \in [0,1]$, let us denote $ (\rho_\theta, y_\theta, z_\theta)^\top = (1 - \theta) \BU_0 + \theta \BU_1$. In the following, we will prove that $\BU_\theta \in \Omega_1$. 
	First, the positivity of $\rho_\theta$ can be obtained by
	$\rho_\theta = (1 - \theta) \rho_0 + \theta \rho_1 \ge \min\{\rho_1,\rho_0\} > 0$.
	Since
	\begin{align*}
		w_\theta
		=
		\frac{y_\theta}{\rho_\theta}
		=
		\frac
		{(1 - \theta) y_0 + \theta y_1}
		{(1 - \theta) \rho_0 + \theta \rho_1}
		&
		=
		\left(
		\frac{(1 - \theta) \rho_0}{(1 - \theta) \rho_0 + \theta \rho_1}
		\right)
		w_0
		+
		\left(
		\frac
		{\theta \rho_1}
		{(1 - \theta) \rho_0 + \theta \rho_1}
		\right)
		w_1,
		\\
		c_\theta
		=
		\frac{z_\theta}{\rho_\theta}
		=
		\frac
		{(1 - \theta) z_0 + \theta z_1}
		{(1 - \theta) \rho_0 + \theta \rho_1}
		& 
		=
		\left(
		\frac{(1 - \theta) \rho_0}{(1 - \theta) \rho_0 + \theta \rho_1}
		\right)
		c_0
		+
		\left(
		\frac
		{\theta \rho_1}
		{(1 - \theta) \rho_0 + \theta \rho_1}
		\right)
		c_1,
	\end{align*}
	namely, $w_\theta$ (resp. $c_\theta$) is a convex combination of $w_0$ and $w_1$ (resp. $c_0$ and $c_1$), we have $w_\theta \in [w_{\min},w_{\max}]$ (resp. $c_\theta \in [c_{\min},c_{\max}]$). Thus, $\BU_\theta \in \Omega_1$ and $\Omega_1$ is convex. \\
	{\sf Part III}: Now, let us assume $\BU_0$ and $\BU_1 \in \Omega_2$. The positivity of $z_\theta$ can be obtained by 
	$z_\theta = (1 - \theta) z_0 + \theta z_1 \ge \min\{z_1,z_0\} > 0$.
	Additionally, we need to show $v_\theta \ge v_{\min}$.
	Consider a function of $\rho$ and $z$ defined as
	\begin{equation}\label{eq:406}
		\eta(\rho,z)
		:=
		\rho (v_{\min} + p(\rho, z))
		=
		y+\rho(v_{\min}-v)
		,
	\end{equation}
	such that for state $\BU = (\rho,y,z)^\top$, the condition $v(\BU) \ge v_{\min}$ is equivalent to $y \ge \eta(\rho,z)$. Since $y_\theta \ge \eta(\rho_\theta,z_\theta)$ for $\theta = 0$ and $1$, in order to prove $v_\theta \ge v_{\min}$, or equivalently, $y_\theta \ge \eta(\rho_\theta,z_\theta)$, we only need to show that the function $\eta(\rho,z)$ is convex for $(\rho,z) \in \mathbb{R}_+ \times \mathbb{R}_+$. To this end, we will show that its Hessian matrix $\mathbf{H}$ is positive semi-definite.
	In the case of $\gamma = 0$, the Hessian matrix is given by
	\(
	\mathbf{H}
	= {\rm diag}\{ v_{\rm ref}/\rho, 0  \}, 
	\)
	which is clearly positive semi-definite.
	In the case of $\gamma > 0$,
	\[
	\mathbf{H}
	=
	\frac{v_{\rm ref}}{\gamma}z^{\kappa}\rho^{\gamma-\kappa-1}
	\left(\begin{array}{cc} 
		(\kappa-\gamma)(\kappa-\gamma-1) z^2
		& -\kappa(\kappa-\gamma-1)z\rho \\ 
		-\kappa(\kappa-\gamma-1)z\rho
		& \kappa (\kappa-1) \rho^2
	\end{array}\right),
	\]
	and the determinants of its principle minors are given by
	\begin{gather*}
		\mathbf{H}_{11}
		= 
		\frac{v_{\rm ref}}{\gamma}z^{\kappa+2}\rho^{\gamma-\kappa-1}
		(\kappa-\gamma)(\kappa-\gamma-1)
		\ge 0,
		\quad
		\mathbf{H}_{22}
		= 
		\frac{v_{\rm ref}}{\gamma}z^{\kappa}\rho^{\gamma-\kappa+1}
		\kappa (\kappa-1) 
		\ge 0,
		\\
		|\mathbf{H}|
		=
		\frac{v_{\rm ref}}{\gamma}z^{\kappa+2}\rho^{\gamma-\kappa-1}
		\left[
		(\gamma-\kappa)(\gamma-\kappa+1) \kappa(\kappa-1)-\kappa^2(\gamma-\kappa+1)^2
		\right]
		\ge 0.
	\end{gather*}
	Thus, we conclude that $\BU_\theta \in \Omega_2$ and $\Omega_2$ is convex. \\
	{\sf Part IV}: Since both $\Omega_1$ and $\Omega_2$ are convex, their intersection $\Omega$ is also convex.
\end{proof}

\begin{lemma}\label{lem:791}
	Consider a function $h(\rho_*,z_*;\rho,z)$ for positive $\rho_*$, $z_*$, $\rho$, and $z$:
	\[
	h(\rho_*,z_*;\rho,z)
	:=
	\rho \, \Big[ p(\rho, z) - p(\rho_*, z_*) \Big]
	+
	\rho_*\left[
	\left(\rho_*-\rho\right) p_{\rho} (\rho_*, z_*)
	+
	\left(z_*-z\right) p_{z} (\rho_*, z_*)
	\right].
	\]
	It holds that 
	$h(\rho_*,z_*;\rho,z) \ge h(\rho,z;\rho,z) = 0$.
\end{lemma}

\begin{proof}
	Note that
	\[
	h_{z_*}(\rho_*, z_*;\rho,z)
	= 
	\rho_*\left[
	\left(\rho_*-\rho\right) p_{\rho z}(\rho_*, z_*)
	+
	(z_*-z) p_{zz}(\rho_*, z_*)
	+
	p_z(\rho_*, z_*)
	\right]
	-\rho p_z(\rho_*, z_*) 
	\]
	We first consider the case of $\gamma > 0$. In this case,
	\[
	h_{z_*}(\rho_*, z_*;\rho,z)
	= 
	\frac{\kappa v_{\rm ref}}{\gamma} z_*^{\kappa-2} \rho_*^{\gamma-\kappa}
	\Big[
	\gamma \rho_*z_*
	+
	(\kappa-\gamma-1) \rho z_*
	-
	(\kappa-1) \rho_* z
	\Big],
	\]
	which, due to $\kappa \ge \gamma+1$ and the positivities of $\rho_*$, $z_*$, $\rho$, and $z$, implies that
	\[
	\begin{dcases}
		h_{z_*}(\rho_*, z_*;\rho,z) \ge 0
		&
		{\rm if} ~  z_* \ge \frac{(\kappa-1) \rho_* z}{\gamma \rho_* + (\kappa-\gamma-1) \rho} =: \rz,
		\\
		h_{z_*}(\rho_*, z_*;\rho,z) < 0
		&
		{\rm otherwise}.
	\end{dcases}
	\]
	Therefore,
	\begin{align*}
		h(\rho_*, z_*;\rho,z)
		\ge 
		h\left(\rho_*, \rz;\rho,z\right)
		= \rho p (\rho, z) - \frac{v_{\rm ref}}{\gamma} z \, \rho_*^{\gamma - \kappa + 1} \rz^{\kappa - 1}
		=:	\tilde{h}(\rho_*;\rho,z).
	\end{align*}
	Taking partial derivative of $\tilde{h}(\rho_*;\rho,z)$ with respect to $\rho_*$ gives
	\[
	\tilde h_{\rho_*}(\rho_*;\rho,z)
	=
	v_{\rm ref} \frac{\kappa - \gamma - 1}{\kappa - 1} 
	(\rho_* -\rho) \, \rho_* ^{\gamma -\kappa -1} \rz^\kappa,
	\]
	which implies $\tilde h(\rho_*;\rho,z)$ is increasing with respect to $\rho_*$ if $\rho_* \ge \rho$, and decreasing if $\rho_* < \rho$.
	Therefore, $\tilde h(\rho_*;\rho,z)	\ge \tilde h(\rho;\rho,z)$, and we finally have
	\[
	h(\rho_*,z_*;\rho,z) 
	\ge \tilde h(\rho_*;\rho,z)
	\ge \tilde h(\rho;\rho,z)
	=
	h(\rho,z;\rho,z) = 0.
	\]
	In the remaining case of $\gamma = 0$, $p(\rho,z) = v_{\rm ref} \log \rho$, and
	\[
	h(\rho_*,z_*;\rho,z)
	=
	v_{\rm ref} \left(
	\rho \log \rho - \rho \log \rho_* + \rho_* - \rho
	\right)
	=: \hat{h}(\rho_*;\rho).
	\]
	Taking derivative with respect to $\rho_*$ yields
	$\hat{h}_{\rho_*}(\rho_*;\rho)=v_{\rm ref}(\rho_*-\rho)/\rho_*$, indicating $\hat{h}(\rho_*;\rho)$ is increasing with respect to $\rho_*$ if $\rho_* \ge \rho$, and decreasing if $\rho_* < \rho$.
	Therefore, in the case of $\gamma = 0$, we conclude the proof by
	\[
	h(\rho_*,z_*;\rho,z) \ge \hat{h}(\rho_*;\rho) \ge \hat{h}(\rho;\rho) = h(\rho,z;\rho,z) = 0.
	\]
\end{proof}

\begin{lemma}
	\label{L1}
	Assume that $\BU = (\rho,y,z)^\top \in \Omega_1$, then the LF splitting form $\hBU_\pm := \mathcal{S}^\pm_{\tt LF}(\BU) = \BU \pm \alpha^{-1} \BF (\BU) \in \Omega_1$ provided that $\alpha > v(\BU) > 0$.
\end{lemma}

\begin{proof}
	Let $\hBU_\pm = (\hat{\rho}_\pm,\hat{y}_\pm,\hat{z}_\pm)^\top$, $\hat{v}_\pm = v(\hBU_\pm)$, $\hat{w}_\pm = w(\hBU_\pm)$, and $\hat{c}_\pm = c(\hBU_\pm)$.
	Because $\alpha > v(\BU) > 0$, we have $\hat{\rho}_\pm = \left( 1 \pm v(\BU)/\alpha \right) \rho > 0$, and
	\begin{equation*}
		\begin{aligned}
			\hat{w}_\pm - w_{\min} 
			& = {(\hat{y}_\pm - \hat{\rho}_\pm \, w_{\min})}/{\hat{\rho}_\pm} 
			= \left( 1 \pm {v(\BU)}/{\alpha} \right) \left( {\rho}/{\hat{\rho}_\pm}\right) (w - w_{\min}) \geq 0, 
			\\
			w_{\max} - \hat{w}_\pm 
			& = {(\hat{\rho}_\pm \, w_{\max} - \hat{y}_\pm)}/{\hat{\rho}_\pm} 
			= \left( 1 \pm {v(\BU)}/{\alpha} \right) \left( {\rho}/{\hat{\rho}_\pm} \right) (w_{\max} - w) \geq 0, 
			\\
			\hat{c}_\pm - c_{\min}  
			& = {(\hat{z}_\pm - \hat{\rho}_\pm \, c_{\min})}{\hat{\rho}_\pm} 
			= \left( 1 \pm {v(\BU)}/{\alpha} \right) \left( {\rho}/{\hat{\rho}_\pm} \right) (c - c_{\min}) \geq 0, 
			\\
			c_{\max} - \hat{c}_\pm 
			& = {(\hat{\rho}_\pm \, c_{\max} - \hat{z}_\pm)}/{\hat{\rho}_\pm} 
			= \left( 1 \pm {v(\BU)}/{\alpha} \right) \left( {\rho}/{\hat{\rho}_\pm} \right) (c_{\max} - c) \geq 0. 
		\end{aligned}
	\end{equation*}
	In summary, $\hBU_\pm \in \Omega_1$. The proof is completed.
\end{proof}

The above \Cref{L1} proves the LF splitting property for $\Omega_1$. However, the constraint $v \ge v_{\min}$ in \eqref{eq:O2} is highly nonlinear, which makes it rather challenging to conduct a similar BP analysis with respect to $\Omega_2$. To overcome this challenge, we adopt the GQL framework from \cite{wu2023geometric} to transform the nonlinear constraint $v \ge v_{\min}$ into equivalent linear constraints by introducing some auxiliary variables, leading to the following GQL representation of $\Omega_2$.

\begin{lemma}[GQL representation of $\Omega_2$]\label{lem:647}
	The set $\Omega_2$ is equivalent to
	\begin{equation*}\label{eq:559}
		\Omega_2^* := \left\{
		\BU \in \mathbb{R}^3 
		: 
		\rho > 0, \; z > 0, \; \BU \cdot \mathbf{n}_* + s_* \geq 0 \quad \forall \; \rho_* > 0, \; z_* > 0 
		\right\},
	\end{equation*}
	where $\mathbf{n}_* := \mathbf{n}(\rho_*,z_*)$, $s_* := s(\rho_*,z_*)$, 
	\begin{equation*}\label{eq:638}
		\begin{aligned}
			\mathbf{n}(\rho,z) &:= - \left(v_{\min} + p (\rho, z) + \rho p_\rho(\rho, z), -1, \rho p_{z} (\rho, z) \right)^\top,
			\\
			s(\rho,z) &:= \rho^2 p_\rho (\rho, z) + \rho z p_z (\rho, z).
		\end{aligned}
	\end{equation*}
\end{lemma}

\begin{proof}
	On one hand, if $\BU = (\rho,y,z)^\top \in \Omega_2$, then, $\rho > 0$, $z > 0$. Notice that, for any $\rho_* > 0$ and $z_* > 0$, 
	\begin{equation}\label{eq:632}
		\begin{aligned}
			\BU \cdot \Bn_* + s_* 
			&= 
			y - \rho \big( v_{\min} + p(\rho,z)\big) + h(\rho_*,z_*;\rho,z)
			\\
			&\stackrel{\rm \Cref{lem:791}}{\ge}
			y - \rho \big( v_{\min} + p(\rho,z)\big)
			=
			\rho \big( v(\BU) - v_{\min}\big)
			\ge 0
			,
		\end{aligned}
	\end{equation}
	which implies that $\BU \in \Omega_2^*$.
	On the other hand, if $\BU \in \Omega_2^*$, then $\rho > 0$, $z > 0$.
	Notice that $\BU \cdot \Bn_* + s_* \ge 0$ for any $\rho_* > 0$ and $z_* > 0$. Taking $\rho_* = \rho$ and $z_* = z$ implies
	\[
	\BU \cdot \Bn(\rho,z) + s(\rho,z)
	=
	y - \rho \big( v_{\min} + p(\rho,z)\big)
	=
	\rho \big( v(\BU) - v_{\min}\big)
	\ge 0.
	\]
	Therefore, it holds that $\BU \in \Omega_2$. The proof is completed.
\end{proof}

\begin{lemma}\label{L2}
	If  $\hBU = (\hat{\rho},\hat{y},\hat{z})^\top \in \Omega_2^*$ 
	and $\cBU = (\check{\rho},\check{y},\check{z})^\top \in \Omega_2^*$, then 
	$\hcBU = (\hcr,\hcy,\hcz)^\top$ $:= \mathcal{S}_{\tt GLF}(\hBU,\widecheck\BU)	\in \Omega^*_2	$,
	or equivalently, the inequality
	\begin{equation}\label{eq:656}
		\hcBU \cdot \mathbf{n}_* + s_* \geq 0
	\end{equation}
	holds for any $\rho_* > 0, z_* > 0$, under the condition
	\begin{equation}\label{eq:674}
		\alpha \geq 
		\alpha_{\max}(\hBU,\widecheck{\BU})
		:=
		\max
		\Big\{ 
		\alpha_{\rm std}
		,
		|v(\hBU)|+\overline{s}(\hBU,\widecheck{\BU})/\,\hat{\rho}
		,
		|v(\widecheck{\BU})|+\overline{s}(\hBU,\widecheck{\BU})/\,\check{\rho}\,
		\Big\}.
	\end{equation}
	Here, 
	$\alpha_{\rm std} := \max \{ \overline\alpha(\hBU), \overline\alpha(\widecheck{\BU})\}$, $\overline \alpha(\BU)$ $:= v + |\rho p_{\rho}+z p_z|$ is an estimate of the maximum propagation speed,
	\begin{equation}\label{eq:686}
		\overline{s}(\hBU,\widecheck{\BU})
		:=
		\max_
		{\substack{
				\rho \in [\underline{\rho},\overline{\rho}]
				\\
				z \in [\underline{z},\overline{z}]
			}
		}
		s(\rho,z)
		=
		\begin{dcases}
			v_{\rm ref}
			\overline{z}^{\,\kappa}\,
			\underline{\rho}^{\gamma-\kappa}
			&
			{\rm if}~ \gamma > 0,
			\\
			v_{\rm ref} \, \overline\rho 
			&
			{\rm if}~ \gamma = 0,
		\end{dcases}
	\end{equation}
	\begin{align*}
		\underline{\rho} &:=\min \left\{{\rho}_1,{\rho}_2\right\},&
		\;
		\overline{\rho}  &:= \max	\left\{{\rho}_1,{\rho}_2\right\},&
		\;
		{\rho}_1 &:= \frac{\hat{\rho}+\check{\rho}}{2},&
		\;
		{\rho}_2 &:= {\rho}_1 + \frac{\hat{\rho}\;\hat{v} - \check{\rho}\;\check{v}}{2 \alpha_{\rm std}},
		\\
		\underline{z} &:= \min \left\{z_1,{z}_2\right\},&
		\;
		\overline{z} &:= \max \left\{{z}_1,{z}_2\right\},&
		\;
		{z}_1 &:= \frac{\hat{z}+\check{z} }{2},&
		\;
		{z}_2 &:= z_1 + \frac{\hat{z}\;\hat{v} - \check{z}\;\check{v}}{2 \alpha_{\rm std}}.
	\end{align*}
\end{lemma}

\begin{proof}
	For any $\BU = (\rho,y,z)^\top \in \Omega_2^*$, we denote $v = v(\BU) \ge v_{\min}$, then
	\begin{equation}\label{eq:669}
		\begin{aligned}
			& ~~~ \pm \BF(\BU) \cdot \Bn_*
			\pm v_{\min} s_*
			\\
			& = 
			\pm v \, (\BU \cdot \Bn_* + s_*)
			\pm (v_{\min}-v) s_*
			\le 
			|v|\,(\BU \cdot \Bn_* + s_*)
			+\frac{s_*}{\rho}\Big[ \rho (v-v_{\min}) \Big]
			\\
			& \stackrel{\eqref{eq:632}}{\le} 
			\, |v|\,(\BU \cdot \Bn_* + s_*)
			+
			\frac{s_*}{\rho}(\BU \cdot \Bn_* + s_*)
			=
			\left( |v|+\frac{s_*}{\rho} \right)
			(\BU \cdot \Bn_* + s_*)
			,
		\end{aligned}
	\end{equation}
	where the last inequality is due to
	\begin{equation}\label{eq:2966}
		\BU \cdot \Bn_* + s_*
		=
		\rho (v-v_{\min}) + h(\rho_*,z_*;\rho,z)
		\stackrel{\rm \Cref{lem:791}}{\ge}
		\rho (v-v_{\min}).
	\end{equation}
	Let us denote $\hcv = v(\hcBU)$, $\hBF = \BF(\hBU)$, and $\cBF = \BF(\cBU)$, then it implies
	\begin{equation}\label{eq:726}
		\begin{aligned}
			2\hcBU \cdot \Bn_* + 2s_*
			& =
			\left[~
			\hBU + \alpha^{-1} \hBF + \cBU - \alpha^{-1} \cBF 
			~\right] 
			\cdot \Bn_* + 2 s_*
			\\
			& =  
			\hBU \cdot \Bn_* + \cBU \cdot \Bn_*
			+
			\alpha^{-1}
			\left(
			\hBF \cdot \Bn_*
			-
			\cBF \cdot \Bn_*
			+
			v_{\min} s_*
			-
			v_{\min} s_*
			\right)
			+
			2 s_*
			\\
			& =  
			\hBU \cdot \Bn_* + \cBU \cdot \Bn_*	+ 2s_*
			-
			\alpha^{-1} \left( \cBF \cdot \Bn_* + v_{\min} s_* \right)
			-
			\alpha^{-1} \left(-\hBF \cdot \Bn_*	- v_{\min} s_*	 \right)
			\\
			& \stackrel{\eqref{eq:669}}{\ge}  
			\left(\cBU \cdot \Bn_*+s_*\right)
			-
			\alpha^{-1}
			\left( \left|v(\cBU)\right|+\frac{s_*}{\hat{\rho}}\right)
			(\cBU \cdot \Bn_* + s_*)
			\\
			&
			\hspace{5mm} +
			\left(\hBU \cdot \Bn_*+s_*\right)
			-
			\alpha^{-1}
			\left( \left|v(\hBU)\right|+\frac{s_*}{\hat{\rho}}\right)
			(\hBU \cdot \Bn_* + s_*)
			\\
			& = 
			\alpha^{-1}
			\left(\hBU \cdot \Bn_*+s_*\right)
			\left(\alpha-\left|v(\hBU)\right|-\frac{s_*}{\hat{\rho}}\right)
			+
			\alpha^{-1}
			\left(\cBU \cdot \Bn_*+s_*\right)
			\left(\alpha-\left|v(\cBU)\right|-\frac{s_*}{\check{\rho}}\right).
		\end{aligned}
	\end{equation}
	Additionally, for any $\rho_* > 0$ and $z_* > 0$,
	\begin{align*}
		& 2 \Big( \hcBU \cdot \Bn_* + s_* \Big)
		=
		2 \Big[
		\hcr (\hcv-v_{\min}) + h(\rho_*,z_*;\hcr,\hcz)
		\Big]
		\\
		\stackrel{\rm \Cref{lem:791}}{\ge} &
		2 \Big[
		\hcr (\hcv-v_{\min}) + h(\hcr,\hcz;\hcr, \hcz)
		\Big]
		=
		2 \left[\hcBU \cdot \Bn(\hcr,\hcz) + s (\hcr,\hcz) \right]
		\\
		\stackrel{\eqref{eq:726}}{\ge} \hspace{3.5mm}
		&
		\alpha^{-1}\left(\hBU\cdot\hcBn+\hcs\right)
		\left(\alpha-\left|v(\hBU)\right|-\frac{\hcs}{\hat{\rho}}\right)
		+
		\alpha^{-1}\left(\cBU\cdot\hcBn+\hcs\right)
		\left(\alpha-\left|v(\cBU)\right|-\frac{\hcs}{\check{\rho}}\right),
	\end{align*}
	where $\hcBn := \Bn(\hcr,\hcz)$ and $\hcs := s(\hcr,\hcz)$.
	Thus, to prove \eqref{eq:656}, we only need to show that
	$
	\alpha 	\ge \left|v(\hBU)\right|+{\hcs}/{\hat{\rho}}
	$ and $
	\alpha \ge \left|v(\cBU)\right|+{\hcs}/{\check{\rho}}
	$ .
	Based on the inequality \eqref{eq:674}, we only need to show $\overline{s}(\hBU,\cBU) \ge \hcs$.
	To this end, according to $\overline{s}$ in \eqref{eq:686}, it is sufficient to prove that
	\begin{equation}\label{eq:1010}
		\hcr \in \big[\,\underline{\rho},\overline{\rho}\,\big],
		\qquad
		\hcz \in \big[\,\underline{z},\overline{z}\,\big].
	\end{equation}
	Notice that 
	\begin{align*}
		\hcr
		=
		\frac{\hat{\rho}}{2}\left(1+\frac{\hat{v}}{\alpha}\right) 
		+&
		\frac{\check{\rho}}{2}\left(1-\frac{\check{v}}{\alpha}\right)
		=
		\frac{\hat{\rho}+\check{\rho}}{2}
		+
		\alpha^{-1} \left( \frac{\hat{\rho}\hat{v}-\check{\rho}\check{v}}{2} \right),
		\\
		\hcz
		=
		\frac{\hat{z}}{2}\left(1+\frac{\hat{v}}{\alpha}\right) 
		+&
		\frac{\check{z}}{2}\left(1-\frac{\check{v}}{\alpha}\right)
		=
		\frac{\hat{z}+\check{z}}{2}
		+
		\alpha^{-1} \left( \frac{\hat{z}\hat{v}-\check{z}\check{v}}{2} \right)
	\end{align*}
	are both linear functions with respect to $\alpha^{-1}$.
	Based on the condition \eqref{eq:674}, $\alpha^{-1} \in \big(\,0,\alpha_{\rm std}^{-1}\,\big]$, which implies the value of $\hcr$ (resp. $\hcz$) is between $\rho_1$ and $\rho_2$ (resp. $z_1$ and $z_2$). Thus,  \eqref{eq:1010} is true. The proof is completed.
\end{proof}

We are now ready to prove the generalized LF splitting property for $\Omega$.

\begin{theorem}\label{thm:1471}
	If 
	$\hBU=(\hat{\rho},\hat{y},\hat{z})^\top \in \Omega$ and 
	$\cBU=(\check{\rho},\check{y},\check{z})^\top \in \Omega$, then 
	$\hcBU = (\hcr,\hcy,\hcz)^\top$ $:= \mathcal{S}_{\tt GLF}(\hBU,\cBU)\in \Omega$,
	under the condition $\alpha > \alpha_{\max}(\hBU,\cBU)$.
\end{theorem}

\begin{proof}
	On one hand, Since $\hBU$, $\cBU \in \Omega \subseteq \Omega_2 = \Omega_2^*$, \Cref{L2} implies that $\hcBU \in \Omega_2^* = \Omega_2$.
	On the other hand, since $\hBU$, $\cBU \in \Omega \subseteq \Omega_1$, and $\alpha > \alpha_{\max}(\hBU, \cBU) \ge \max\{ v(\hBU), v(\cBU) \}$, \Cref{L1} implies $\mathcal{S}^+_{\tt LF}(\hBU)$, $\mathcal{S}^-_{\tt LF}(\cBU) \in \Omega_1$. The convexity of $\Omega_1$ (\Cref{lem:466}) implies $\hcBU = \frac12\mathcal{S}^+_{\tt LF}(\hBU)+\frac12\mathcal{S}^-_{\tt LF}(\cBU)  \in \Omega_1$.
	Therefore, it holds that $\hcBU \in \Omega_1 \cap \Omega_2 = \Omega$, and the proof is completed.
\end{proof}

\section{High-order OEDG methods for ARZ models}\label{sec:OEDG}

This section introduces the OEDG methods for the AP ARZ model \eqref{system}, with the original ARZ model as a degenerate case.

Assume that the spatial domain $\mathcal{D}$ is divided into $N_x$ uniform cells $\big\{ \, I_j := [x_{j - \frac{1}{2}}, x_{j + \frac{1}{2}}] \, \big\}_{j=1}^{N_x}$ with size $\Delta x$ and center $x_j = \frac{1}{2} (x_{j - \frac{1}{2}} + x_{j + \frac{1}{2}})$. 
The time interval is divided into a mesh $\big\{ t_0 = 0, \; t_{n + 1} =$ $ t_n + \Delta t_n, \; n \geq 0\big\}$. 
Let $\overline{\BU}_j^n$ denote the cell-averaged approximation of the exact solution $\BU(x, t)$ over $I_j$ at time $t = t_n$. 
Let $\BU_h (x,\cdot)$ be a piecewise polynomial vector function of degree $k$, namely,
\begin{equation*}
	\BU_h (x,\cdot) \in \mathbb{V}_h^k := \left\{~ \BU(x) = (\rho_h(x), y_h(x), z_h(x))^\top: \rho_h \vert_{I_j}, y_h \vert_{I_j}, z_h \vert_{I_j} \in \mathbb{P}^k (I_j) \; \forall j ~\right\},
\end{equation*}
where $\mathbb{P}^k (I_j)$ denotes the space of polynomials of degree up to $k$ on $I_j$. 
We aim at finding $\BU_h(x,t) \in \mathbb{V}_h^k$ as an approximation to the solution $\BU(x,t)$. 
Consider the DG spatial discretization for $\BU_h$; see, e.g., \cite{cockburn2000development,cockburn1989tvb,cockburn1989tvbs,shu2009discontinuous}. 
Multiplying the system \eqref{system} with a test function $\phi(x) \in \mathbb{P}^k (I_j)$, and integrating by parts over $I_j$ yields
\begin{equation}\label{DG}
	\begin{aligned}
		\frac{\rm d}{{\rm d} t} \int_{I_j} \BU_h(x,t) \, \phi(x)  \, {\rm d} x 
		= &
		\int_{I_j} \BF (\BU_h) \frac{{\rm d} \phi}{{\rm d} x} {\rm d} x 
		+ 
		\BF \left(\BU_h (x_{j - \frac{1}{2}},t ) \right) \phi ( x_{j - \frac{1}{2}} ) 
		\\
		&
		- 
		\BF \left(\BU_h (x_{j + \frac{1}{2}},t ) \right) \phi ( x_{j + \frac{1}{2}} ),
	\end{aligned}
\end{equation} 
where we have replaced the exact solution $\BU(x,t)$ with the approximation solution $\BU_h(x,t)$.
Let $\big\{ \phi_j^{[\ell]}(x) \big\}_{\ell = 0}^k$ denote a local orthogonal basis of the polynomial space $\mathbb{P}^k (I_j)$, then one can decompose the DG approximate solution $\BU_h(x,t)$ as
\begin{equation}\label{basis}
	\BU_h (x,t) = \sum_{j} \mathbbm{1}_{I_j}(x) \, \BU_j (x, t),
	\quad
	\BU_j (x,t)
	:=
	\sum_{\ell = 0}^k \BU_j^{[\ell]} (t) \phi_j^{[\ell]} (x)
	.
\end{equation}
In this paper, we choose the first $(k+1)$ orthogonal Legendre polynomials as basis.
Substituting \eqref{basis} into \eqref{DG}, replacing $\phi(x)$ with basis function $\phi_j^{[\ell]} (x)$, and approximating $\BF \left(\BU_h (t, x_{j \pm \frac{1}{2}} ) \right)$ with numerical fluxes $\widehat{\BF} \left( \BU^-_{j \pm \frac{1}{2}}, \BU^+_{j \pm \frac{1}{2}} \right) $ results in
\begin{equation*}
	\begin{aligned}\label{DGmethod}
		\frac{\rm d}{{\rm d} t} \BU_j^{[\ell]} (t) 
		=  
		\frac{2 \ell + 1}{\Delta x} 
		& \bigg[ 
		\int_{I_j} \BF (\BU_h) \frac{{\rm d} \phi_j^{[\ell]} (x)}{{\rm d} x} {\rm d} x 
		+ 
		\widehat{\BF} \left( \BU^-_{j - \frac{1}{2}}, \BU^+_{j - \frac{1}{2}} \right) 
		\phi_j^{[\ell]} (x_{j - \frac{1}{2}}) 
		\\
		& - \widehat{\BF} \left( \BU^-_{j + \frac{1}{2}}, \BU^+_{j + \frac{1}{2}} \right) \phi_j^{[\ell]} (x_{j + \frac{1}{2}})
		\bigg],
	\end{aligned}
\end{equation*} 
where $\widehat{\BF} \left( \cdot, \cdot \right) $ is a two-point numerical flux, for example, the LF numerical flux:
\begin{equation*}\label{LF}
	\widehat{\BF} ( \BU^-, \BU^+ ) 
	= 
	\frac12 \left[ \BF ( \BU^- ) + \BF ( \BU^+ ) - \alpha ( \BU^+ - \BU^- ) \right].
\end{equation*}
Here, $\alpha$ denotes an estimate of maximum propagation speed.
The states $\BU^-_{j + \frac{1}{2}}$ and $\BU^+_{j + \frac{1}{2}}$ respectively denote the left- and right-limits of $\BU_h$ at $x_{j+\frac12}$:
$\BU^{+}_{j + \frac{1}{2}} := \BU_{j+1}(x_{j+\frac12})$, $\BU^{-}_{j + \frac{1}{2}}:=\BU_{j}(x_{j+\frac12})$.

We may rewrite the semidiscrete scheme \eqref{DG} in an ODE form $\frac{\rm d}{{\rm d} t} \BU_h = \mathcal{L} (\BU_h)$, and discretize it in time with, for example, the classic explicit third-order strong-stability-preserving (SSP)  Runge--Kutta (RK) method:
\begin{equation}
	\label{RK}
	\begin{aligned}
		&\BU_h^{n, 1} = \BU_h^{n} + \Delta t_n \mathcal{L} \left(\BU_h^{n} \right), \\
		&\BU_h^{n, 2} = \frac{3}{4} \BU_h^{n} + \frac{1}{4} \left[ \BU_h^{n, 1} + \Delta t_n \mathcal{L} \left(\BU_h^{n, 1} \right) \right], \\
		&\BU_h^{n + 1} = \frac{1}{3} \BU_h^{n} + \frac{2}{3} \left[ \BU_h^{n, 2} + \Delta t_n \mathcal{L} \left(\BU_h^{n, 2} \right) \right].
	\end{aligned}
\end{equation}

It is well-known that the numerical solution obtained by \eqref{RK} may be subject to spurious oscillations.
To eliminate the oscillations and enforce stability, we apply the OE procedure \cite{peng2023oedg} following each RK stage, leading to
\begin{equation}\label{OEOperator}
	\begin{aligned}
		\tBU_h^{n, 1} &= \BU_h^{n} + \Delta t_n \mathcal{L} \left(\BU^{n}_h \right), 
		& 
		\BU^{n, 1}_h &= \mathcal{F}_\tau \tBU_h^{n, 1}, 
		\\
		\tBU_h^{n, 2} &= \frac{3}{4} \BU_h^{n} + \frac{1}{4} \left( \BU^{n, 1}_h + \Delta t_n \mathcal{L} (\BU^{n, 1}_h ) \right), 
		& 
		\BU^{n, 2}_h &= \mathcal{F}_\tau \widetilde \BU_h^{n, 2}, 
		\\
		\tBU_h^{n,3} & = \frac{1}{3} \BU_h^{n} + \frac{2}{3} \left( \BU^{n, 2}_h + \Delta t_n \mathcal{L} (\BU^{n, 2}_h ) \right), 
		& 
		\BU^{n,3}_h & = \mathcal{F}_\tau \widetilde \BU_h^{n,3},
		&
		\BU^{n+1}_h & = \BU^{n,3}_h.
	\end{aligned}
\end{equation}
Here, the OE operator $\mathcal{F}_\tau: \mathbb{V}_h^k \rightarrow \mathbb{V}_h^k$ is the solution operator of the following damping equation \eqref{InitialValueProblem}.
Specifically, we define $(\mathcal{F}_\tau \tBU_h^{n,s}) (x) = \BU_\sigma (x, \tau)$, where $\BU_\sigma (x, \hat{t})$ is the solution of the following ordinary differential equation with respect to $\hat{t}$:
\begin{equation}
	\label{InitialValueProblem}
	\left\{
	\begin{aligned}
		& \frac{\mathrm{d}}{\mathrm{d} \hat{t}} 
		\int_{I_j} \BU_{\sigma} \cdot \bm{\phi} \, \mathrm{d} x 
		+ 
		\sum_{i = 0}^{k} \theta_j \frac{\sigma_j^{i} \big( \widetilde \BU_h^{n,s} \big)}{\Delta x} 
		\int_{I_j} \left( \BU_{\sigma} - \mathcal{P}^{i - 1}_{\sigma} \BU_{\sigma} \right) \cdot \bm{\phi} \, \mathrm{d} x 
		= 
		0, \\
		& \BU_{\sigma} (x, 0) = \widetilde \BU_h^{n,s} (x),
	\end{aligned}
	\right.
\end{equation}
where $\bm{\phi} \in (\mathbb{P}^k ( I_j ))^3$ is a test function, $\theta_j$ is the spectral radius of the Jacobian matrix $\partial \BF \big( \overline{\BU}_j^n \big) / \partial \BU$, and $\mathcal{P}^{\, i}_\sigma$ is the standard $L^2$ projection operator. 
To make sure that the OE operator $\mathcal{F}_\tau$ is both scale-invariant and evolution-invariant (see \cite{peng2023oedg}), the damping coefficient $\sigma^i_j \big( \widetilde  \BU_h^{n, s} \big)$ is defined as
$
\sigma^i_j \big( \widetilde \BU_h^{n, s} \big) 
:= 
\max
\left\{
\sigma^i_j \left( \tilde \rho_h^{n,s} \right), 
\sigma^i_j \left( \tilde y_h^{n,s} \right), 
\sigma^i_j \left( \tilde z_h^{n,s} \right) 
\right\}
$
with $\widetilde \BU^{n,s}_h(x) = \big(\tilde \rho_h^{n,s}(x),$ $\tilde y_h^{n,s}(x), \tilde z_h^{n,s}(x)\big)^\top$ and
\begin{equation*}
	\sigma_j^i (u) = \left\{
	\begin{aligned}
		& 0, \quad && {\rm if} \; u \equiv \mathrm{constant}, \\
		& 
		\frac{(2 i + 1) (\Delta x)^i}{(2 k - 1) \, i !} 
		\cdot
		\frac
		{
			\left \vert 
			[\![ \partial^i_x u ]\!]_{j - \frac{1}{2}} 
			\right \vert 
			+ 
			\left \vert 
			[\![ \partial^i_x u ]\!]_{j + \frac{1}{2}} 
			\right \vert
		}
		{
			2 
			\left \Vert 
			u - \frac{1}{\vert \mathcal{D} \vert} \int_{\mathcal{D}} u(x) \, \mathrm{d} x \right \Vert
			_{L^\infty (\mathcal{D})}
		}
		, 
		\quad && {\rm otherwise}.
	\end{aligned}
	\right.
\end{equation*}
Here, $\mathcal{D}$ denotes the computational domain, and $[\![ u ]\!]_{j+\frac12} = u_{j + \frac{1}{2}}^+ - u_{j + \frac{1}{2}}^-$ denotes the jump of $u$ at the cell interface $x = x_{j + \frac{1}{2}}$. 
The explicit solution of \eqref{InitialValueProblem} was found in \cite{peng2023oedg}:
\begin{equation*}
	\BU^{n,s}
	=
	\mathcal{F}_\tau \widetilde \BU_h^{n,s}
	=
	\sum_{j} \mathbbm{1}_{I_j}(x)
	 \left[
	\widetilde \BU_j^{[0]} \phi_j^{[0]} (x) + \sum^k_{i = 1} \exp\Bigg(  - \frac{ \tau \theta_j}{\Delta x} \sum_{q = 0}^i \sigma_j^q \big( \widetilde \BU_h^{n, s} \big)  \Bigg) \widetilde \BU_j^{[i]} \phi_j^{[i]} (x)
	\right] .
\end{equation*}
Here, $\widetilde \BU_j^{[i]}$ is the $i$-th modal coefficient of $\widetilde \BU_h^{n, s}$ in $I_j$. 
Please note that the $0$-th mode is not affected by the OE operator, leading to the following proposition.

\begin{proposition}\label{lem:1625b}
	Applying the OE operator to a numerical solution maintains its cell averages, that is, $\oBU^{n,s}_j = \hoBU^{n,s}_j$, for $j = 1,\dots,N_x$ and $s = 1,2,3$.
\end{proposition} 

\begin{remark}
	We adopt the OE procedure instead of the traditional TVB limiter to control numerical oscillations. The main reason for this choice is that while the TVB limiter requires the solution to be decomposed into its characteristics for an effective simulation, the OE procedure does not necessitate such a decomposition. 
	For more information on OEDG methods, readers are referred to \cite{peng2023oedg}.
\end{remark}

\begin{remark}
	Adding the OE operator to the scheme \eqref{RK} does not change its order of accuracy; see \cite{peng2023oedg}. 
\end{remark}

\section{High-order BP-OEDG methods for ARZ models}\label{sec:BP-OEDG}

In general, the numerical results of OEDG methods (\Cref{sec:OEDG}) may not satisfy the invariant domain $\Omega$. 
To enforce this BP property, we add a BP limiter to the numerical solution following each OE procedure in \eqref{OEOperator}, resulting in
\begin{equation}\label{BPOE}
	\begin{aligned}
		&\tBU_h^{n, 1}  = \BU^{n}_h + \Delta t_n \mathcal{L} \left(\BU^{n}_h \right), 
		& 
		 \hBU_h^{n, 1} & = \mathcal{F}_\tau \tBU_h^{n, 1}, 
		&
		 \BU_j^{n, 1} & = \mathcal{B} \left( \hBU_j^{n, 1}; {\Omega}_j^{n,1} \right) \; \forall j,
		\\
		& \tBU_h^{n, 2} = \frac{3}{4} \BU_h^{n} + \frac{1}{4} \left( \BU_h^{n, 1} + \Delta t_n \mathcal{L} (\BU_h^{n, 1} ) \right), 
		& 
		\hBU_h^{n, 2} & = \mathcal{F}_\tau \tBU_h^{n, 2}, 
		&
		\BU_j^{n, 2} & = \mathcal{B} \left( \hBU_j^{n, 2}; {\Omega}_j^{n,2} \right) \; \forall j,
		\\
		&\tBU_h^{n,3} = \frac{1}{3} \BU_h^{n} + \frac{2}{3} \left( \BU_h^{n, 2} + \Delta t_n \mathcal{L} (\BU_h^{n, 2} ) \right), 
		& 
		\hBU_h^{n,3} & = \mathcal{F}_\tau \tBU_h^{n,3},
		&
		\BU_j^{n,3} & = \mathcal{B} \left( \hBU_j^{n,3}; {\Omega}_j^{n,3} \right) \; \forall j,
		\\
		&\BU^{n+1}_h = \BU^{n,3}_h.
	\end{aligned}
\end{equation}
Here, $\BU^{n,s}_j$ (resp. $\hBU^{n,s}_j$) denotes the local polynomials of numerical solution $\BU^{n,s}_h$ (resp. $\hBU^{n,s}_h$) in $I_j$, $\mathcal{B}$ denotes the BP limiter, which modifies the local polynomial $\hBU_j^{n,s}$, such that the resulting local polynomial ${\BU}_j^{n,s}$ satisfies the invariant domain ${\Omega}_j^{n,s}$.
In the following, we will introduce the recipe of approximate invariant domain ${\Omega}_j^{n,s}$ in \Cref{sec:1596} and the BP limiter $\mathcal{B}$ in \Cref{sec:1594}.
The BP property of the scheme \eqref{BPOE} will be proven in \Cref{sec:1598}.

\subsection{Approximate invariant domains}\label{sec:1596}

In this section, we introduce the approximation of invariant domain $\Omega^{n,s}_j$ for $s = 1,\dots,3$ and $j = 1,\dots,N_x$.
For notational convenience, we consider the time level $t_n$ as the final (resp. initial) stage of preceding (resp. following) RK step, that is, $\BU^{n-1,3}_h = \BU^n_h = \BU^{n,0}_h$ and $\Omega_j^{n-1,3} = \Omega_j^n = \Omega_j^{n,0}$ for all $j$.

In Subsection~\ref{sec:88}, we introduced two types of invariant domains:
\begin{enumerate}[leftmargin=8mm]
	\item[(i)] The invariant domain $\mathcal{I}^{\rm G}$, defined in \eqref{eq:332}, where the lower and upper bounds of the Riemann invariants are determined by the global minimum and maximum of the Riemann invariants at $t = 0$, respectively.
	\item[(ii)] The invariant domain $\mathcal{I}^{\rm L}$, defined in \eqref{eq:154}, where the upper and lower bounds of the Riemann invariants are determined by the maximum and minimum of the Riemann invariants over the local domain of determination $\mathcal{X}$, respectively.
\end{enumerate}
Similarly, we can approximate the upper and lower bounds of the Riemann invariants in the definition \eqref{eq:270} of the invariant domain $\Omega$ in both global and local approaches, which are respectively introduced in the following Subsections~\ref{sec:global} and~\ref{sec:local}.

\subsubsection{Approximate global invariant domains}\label{sec:global}

When the {\em global} invariant domains are considered, the invariant domains $\Omega^{n,s}_j$ are identical over all cells, therefore, we temporarily omit the subscription $j$, and assume $\Omega^{n,s}_j = \Omega^{n,s}$ for all $j$.

At the beginning of computation, we estimate the initial global invariant domain 
\[
{\Omega}^{0}
=
\left\{
\BU \in \mathbb{R}^3
:
\begin{array}{r}
	\rho > 0, z > 0,
	w \in \left[ (w_{\min})^{0}, (w_{\max})^{0} \right] \\
	v \ge (v_{\min})^{0},
	c \in \left[ (c_{\min})^{0}, (c_{\max})^{0}  \right]
\end{array}
\right\}
\]
by sampling $\BU_h^0(x)$ over a uniform auxiliary mesh $\big\{x^{\rm s}_i\big\}_{i = 1}^{N_{\rm s}}$:
\begin{gather}
	(v_{\min})^0 = \max \Big\{ \min_{i = 1,\dots,N_s} v\left(\BU_h^0 (x_i^{\rm s})\right) - \varepsilon_0, 0 \Big\}, \label{eq:1769a}
	\\
	(w_{\min})^0 = \min_{i = 1,\dots,N_s} w \left(\BU_h^0 (x_i^{\rm s})\right) - \varepsilon_0, 
	\qquad
	(w_{\max})^0 = \max_{i = 1,\dots,N_s} w \left(\BU_h^0 (x_i^{\rm s})\right) + \varepsilon_0, \label{eq:1769b}
	\\
	(c_{\min})^0 = \min_{i = 1,\dots,N_s} c \left(\BU_h^0 (x_i^{\rm s})\right) - \varepsilon_0, 
	\qquad
	(c_{\max})^0 = \max_{i = 1,\dots,N_s} c \left(\BU_h^0 (x_i^{\rm s})\right) + \varepsilon_0. \label{eq:1769c}
\end{gather}
\begin{remark}
	The approximate global invariant domain $\Omega^0$ may not be identical to the exact global invariant domain $\Omega$.
	If $\Omega \backslash \Omega^0 \neq \emptyset$, enforcing BP property with respect to $\Omega^0$ may be overly strict: a reference solution with its range overlapping $\Omega \backslash \Omega^0$ may be subject to (unnecessary) modifications by the BP limiter, which may lead to accuracy degeneracy in the numerical results.
	To avoid such an issue, we introduce a small positive parameter $\varepsilon_0$ in the above formulas \eqref{eq:1769a}--\eqref{eq:1769c} to slightly expand $\Omega^{0}$.
	In all the numerical experiments reported in \Cref{sec:num}, we set $\varepsilon_0 = 10^{-9}$.
\end{remark}

When an initial boundary value problem is considered, the global invariant domain should also be adjusted according to boundary conditions.
We update the approximate global invariant domain using the following formula
\begin{equation}\label{eq:1795}
	\Omega^{n,s} :=
	\left\{
	\BU \in \mathbb{R}^3
	:
	\renewcommand\arraystretch{1.2}
	\begin{array}{r}
		\rho > 0, z > 0, w \in \left[ (w_{\min})^{n,s}, (w_{\max})^{n,s} \right] \\
		v \ge (v_{\min})^{n,s},	c \in \left[ (c_{\min})^{n,s}, (c_{\max})^{n,s}  \right]
	\end{array}
	\right\}
\end{equation}
with
\begin{equation}\label{eq:1807}
	\begin{aligned}
		(v_{\min})^{n,s} & = 
		\min \big\{
		(v_{\min})^{n,s-1},v\big(\BU^{n,s-1}_h(x_{\frac12}^{-})\big),v\big(\BU^{n,s-1}_h(x_{N_x+\frac12}^{+})\big)
		\big\},
		\\
		(w_{\min})^{n,s} & = 
		\min \big\{
		(w_{\min})^{n,s-1},w\big(\BU^{n,s-1}_h(x_{\frac12}^{-})\big),w\big(\BU^{n,s-1}_h(x_{N_x+\frac12}^{+})\big)
		\big\},
		\\
		(w_{\max})^{n,s} & = 
		\max \big\{
		(w_{\max})^{n,s-1},w\big(\BU^{n,s-1}_h(x_{\frac12}^{-})\big),w\big(\BU^{n,s-1}_h(x_{N_x+\frac12}^{+})\big)
		\big\},
		\\
		(c_{\min})^{n,s} & = 
		\min \big\{
		(c_{\min})^{n,s-1},c\big(\BU^{n,s-1}_h(x_{\frac12}^{-})\big),c\big(\BU^{n,s-1}_h(x_{N_x+\frac12}^{+})\big)
		\big\},
		\\
		(c_{\max})^{n,s} & = 
		\max \big\{
		(c_{\max})^{n,s-1},c\big(\BU^{n,s-1}_h(x_{\frac12}^{-})\big),c\big(\BU^{n,s-1}_h(x_{N_x+\frac12}^{+})\big)
		\big\},
	\end{aligned}
\end{equation}
where $\BU^{n,s-1}_h(x_{\frac12}^{-})$ and $\BU^{n,s-1}_h(x_{N_x+\frac12}^{+})$ denote respectively the states immediate outside the left and right boundaries, which are determined by boundary conditions.

\begin{proposition}\label{lem:1872}
	The global invariants defined in \eqref{eq:1795}--\eqref{eq:1807} satisfy
	$\BU^{n,s-1}_h(x_{\frac12}^{-}) \in \Omega^{n,s}$ and $\BU^{n,s-1}_h(x_{N_x+\frac12}^{+}) \in \Omega^{n,s}$.
\end{proposition}

Based on the formulas in \eqref{eq:1807}, the global invariant domains corresponding to all time steps and RK stages form a series of nested domains.
\begin{proposition}\label{lem:1889}
	The global invariants defined in \eqref{eq:1795}--\eqref{eq:1807} satisfy	$\Omega^{n_1,s_1} \subseteq \Omega^{n_2,s_2}$ and $\Omega^{n_1} \subseteq \Omega^{n_2}$ for any $0 \le n_1 \le n_2$ and $0\le s_1 \le s_2\le 3$.
\end{proposition}

\subsubsection{Approximate local invariant domains}\label{sec:local}

We express the approximate local invariant domain on $I_j$ at the $s$-th RK stage of the $n$-th time step by
\begin{equation}\label{eq:1851}
	\Omega^{n,s}_j :=
	\left\{
	\BU \in \mathbb{R}^3
	:
	\begin{array}{r}
		\rho > 0, z>0, w \in \left[ (w_{\min})^{n,s}_j, (w_{\max})^{n,s}_j \right] \\
		v \ge (v_{\min})^{n,s}_j, c \in \left[ (c_{\min})^{n,s}_j, (c_{\max})^{n,s}_j  \right]
	\end{array}
	\right\}
\end{equation}
which is defined by the following formulas
\begin{equation}\label{eq:1917}
	\Omega^{n,1}_j := \Omega^{n,1} \cap \Omega^{n,0}_{j,*},
	\quad
	\Omega^{n,s}_j := \Omega^{n,s} \cap \left( \Omega^{n,s-1}_{j} \cup \Omega^{n,s-1}_{j,*} \right) ~~
	\textrm{for} ~~ s = 2,3,
\end{equation}
where
\begin{equation}\label{eq:1927}
	\Omega^{n,s}_{j,*}
	:=
	\left\{ 
	\BU \in \mathbb{R}^3 
	:
	\renewcommand\arraystretch{1.5}
	\begin{array}{r}
		\rho > 0, z > 0, 
		v \ge (v_{\min})^{n,s}_{j,*}-\sqrt{\Delta x} \, \big|(v_{\min})^{n,s}_{j,*}\big| \\
		w \in 
		\left[
		(w_{\min})^{n,s}_{j,*}-\sqrt{\Delta x} \,\big|(w_{\min})^{n,s}_{j,*}\big|,
		(w_{\max})^{n,s}_{j,*}+\sqrt{\Delta x} \,\big|(w_{\max})^{n,s}_{j,*}\big|
		\right] \\
		c \in 
		\left[
		(c_{\min})^{n,s}_{j,*}-\sqrt{\Delta x} \,\big|(c_{\min})^{n,s}_{j,*}\big|,
		(c_{\max})^{n,s}_{j,*}+\sqrt{\Delta x} \,\big|(c_{\max})^{n,s}_{j,*}\big|
		\right]
	\end{array}
	\right\}.
\end{equation}
Here, $(v_{\min})^{n,s}_{j,*}$, $(w_{\min})^{n,s}_{j,*}$, $(w_{\max})^{n,s}_{j,*}$, $(c_{\min})^{n,s}_{j,*}$, $(c_{\max})^{n,s}_{j,*}$ are respectively estimates of minimum/maximum of $v$, $w$, and $c$ of intermediate solution $\BU^{n,s}_h$ in the vicinity of $I_j$:
\begin{equation}\label{eq:1982}
	\begin{aligned}
		(v_{\min})^{n,s}_{j,*}
		& =
		\min\Big\{
		\min\limits_{\ell = 1,\dots,L} v\big(\BU_{j}^{n,s}(x_{j,\ell}^{\tt GL})\big),
		\min\limits_{I_j} \mathcal{C} v\big(\BU_{j}^{n,s}\big),
		v\big(\BU_h^{n,s}(x_{j+\frac12}^{+})\big),
		v\big(\BU_h^{n,s}(x_{j-\frac12}^{-})\big)
		\Big\},
		\\
		(w_{\min})^{n,s}_{j,*}
		& =
		\min\Big\{
		\min\limits_{\ell = 1,\dots,L} w\big(\BU_{j}^{n,s}(x_{j,\ell}^{\tt GL})\big),
		\min\limits_{I_j} \mathcal{C} w\big(\BU_{j}^{n,s}\big),
		w\big(\BU_h^{n,s}(x_{j+\frac12}^{+})\big),
		w\big(\BU_h^{n,s}(x_{j-\frac12}^{-})\big)
		\Big\},
		\\
		(w_{\max})^{n,s}_{j,*}
		& =
		\max\Big\{
		\max\limits_{\ell = 1,\dots,L} w\big(\BU_{j}^{n,s}(x_{j,\ell}^{\tt GL})\big),
		\max\limits_{I_j} \mathcal{C} w\big(\BU_{j}^{n,s}\big),
		w\big(\BU_h^{n,s}(x_{j+\frac12}^{+})\big),
		w\big(\BU_h^{n,s}(x_{j-\frac12}^{-})\big)
		\Big\},
		\\
		(c_{\min})^{n,s}_{j,*}
		& =
		\min\Big\{
		\min\limits_{\ell = 1,\dots,L} c\big(\BU_{j}^{n,s}(x_{j,\ell}^{\tt GL})\big),
		\min\limits_{I_j} \mathcal{C} c\big(\BU_{j}^{n,s}\big),
		c\big(\BU_h^{n,s}(x_{j+\frac12}^{+})\big),
		c\big(\BU_h^{n,s}(x_{j-\frac12}^{-})\big)
		\Big\},
		\\
		(c_{\max})^{n,s}_{j,*}
		& =
		\max\Big\{
		\max\limits_{\ell = 1,\dots,L} c\big(\BU_{j}^{n,s}(x_{j,\ell}^{\tt GL})\big),
		\max\limits_{I_j} \mathcal{C} c\big(\BU_{j}^{n,s}\big),
		c\big(\BU_h^{n,s}(x_{j+\frac12}^{+})\big),
		c\big(\BU_h^{n,s}(x_{j-\frac12}^{-})\big)
		\Big\},
	\end{aligned}
\end{equation}
Here, $\mathcal{C}: C(I_j) \mapsto \mathbb{P}^3(I_j)$ denotes a cubic interpolation operator over $I_j$, such that $[\mathcal{C}f](x) = f(x)$ for $x \in \big\{x_{j-\frac12},x_{j-\frac16},x_{j+\frac16},x_{j+\frac12}\big\}$.
\begin{remark}
	Based on the formulas in \eqref{eq:1917}, please note that for $s = 1,2,3$, it holds that $\Omega^{n,s}_j \subseteq \Omega^{n,s}$, that is, the local invariant domain $\Omega^{n,s}_j$ always imposes a stronger restriction on the numerical solutions than the global one $\Omega^{n,s}$.
\end{remark}

\begin{remark}
	If the approximate invariant domain is too strict, enforcing such an invariant domain upon a numerical solution may lead to accuracy degeneracy. To avoid such an issue, we expand the local invariant domains in \eqref{eq:1927} by relaxing the upper and lower bounds with a multiplicative factor of $\sqrt{\Delta x}$.
\end{remark}

\subsection{BP limiter $\mathcal{B}$}\label{sec:1594}

Inspired by \cite{zhang2010maximum} and \cite{zhang2011maximum}, we present in this section a BP limiter $\mathcal{B}$, which is applied on the local polynomial $\hBU_j^{n,s}$ such that the result $\BU_j^{n,s}$ is BP with respect to the global invariant domain \eqref{eq:1795} or the local invariant domain \eqref{eq:1851}.
In fact, the constraints shown in the definition of these domains can be equivalently expressed by the positivities of the following functions:
\begin{align*}
	& H_1(\BU) := \rho, \;
	H_2(\BU) := z, \;
	H_{3}(\BU)
	:= \rho \big( w - (w_{\min})^{n,s}_j \big), \;
	H_{4}(\BU)
	:= \rho \big( (w_{\max})^{n,s}_j - w \big),
	\\
	& H_{5}(\BU)
	:= \rho \big( c - (c_{\min})^{n,s}_j \big), \;
	H_{6}(\BU)
	:= \rho \big( (c_{\max})^{n,s}_j - c \big), \;
	H_7(\BU) := 
	\rho \big(v-(v_{\min})_j^{n,s}\big).
\end{align*}
Here, the functions $H_1,\dots,H_6$ are all linear, and thus concave. The function $H_7$ is also concave for $\BU \in \Omega_0$, since it can be rewritten as a sum of concave functions: 
$H_7(\BU) = y + \rho v_{\min} - \rho (v_{\min})_j^{n,s} - \eta(\rho,z)$. The function $\eta$ was defined in \eqref{eq:406}, and its convexity was shown in the proof of \Cref{lem:466}.
Before introducing the BP limiter $\mathcal{B}$, we first review several properties of local scaling form $\mathcal{U}(\theta;\oBU,\hBU) := \theta \hBU + (1-\theta) \oBU$, which has been used intensively in constructing BP limiters, for example, in \cite{zhang2011maximum,zhang2010maximum}.

\begin{lemma}\label{lem:1467}
	For any $d \in \mathbb{N}^+$, consider a convex set $\mathcal{R} \in \mathbb{R}^d$, a continuous and concave function $H(\BU): \mathcal{R} \mapsto \mathbb{R}$, and a threshold $\varepsilon \in \mathbb{R}$ such that $\mathcal{R}_\varepsilon := \{\BU \in \mathcal{R} : H(\BU) \ge \varepsilon \}$ is not empty.
	For any $\oBU \in \mathcal{R}_\varepsilon$ and $\hBU \in \mathcal{R}$, let us define
	\[
	\Theta(\oBU,\hBU,H,\varepsilon) :=
	\begin{dcases}
		1 & \textrm{if }H(\hBU) \ge \varepsilon, \\
		\textrm{ the unique solution of } H\big(\mathcal{U}(\theta;\oBU,\hBU)\big) = \varepsilon \textrm{ in } [0,1] & \textrm{if }H(\hBU) < \varepsilon.
	\end{dcases}
	\]
	where $\mathcal{U}(\theta;\oBU,\hBU) := \theta \hBU + (1-\theta) \oBU$.
	It holds that
	(i) the set $\mathcal{R}_\varepsilon$ is convex; 
	(ii) the equation $H\big(\mathcal{U}(\theta;\oBU,\hBU)\big) = \varepsilon$ has a unique solution in $[0,1]$ for $\hBU \not \in \mathcal{R}_\varepsilon$; and
	(iii) for any $\theta \in \big[0,\Theta(\oBU,\hBU,H,\varepsilon)\big]$, $\mathcal{U}(\theta;\oBU,\hBU) \in \mathcal{R}_\varepsilon$, that is, $H(\mathcal{U}(\theta;\oBU,\hBU)) \ge \varepsilon$.
\end{lemma}

\begin{proof}
	(i) For any $\lambda \in [0,1]$ and any $\BU_0$, $\BU_1 \in \mathcal{R}_\varepsilon \subseteq \mathcal{R}$, the convexity of $\mathcal{R}$ implies that $\BU_\lambda = \lambda \BU_1 + (1-\lambda) \BU_0 \in \mathcal{R}$. Due to the concavity of function $H$, it holds that $H(\BU_\lambda) \ge \lambda H(\BU_1) + (1-\lambda) H(\BU_0) \ge \varepsilon$. Therefore, $\BU_\lambda \in \mathcal{R}_\varepsilon$ and $\mathcal{R}_\varepsilon$ is convex.
	
	(ii)
	For $\hBU \not \in \mathcal{R}_\varepsilon$, the set $\{\mathcal{U}(\theta;\oBU,\hBU) : \theta \in [0,1]\}$ forms a line segment connecting $\oBU \in \mathcal{R}_\varepsilon$ and $\hBU \not \in \mathcal{R}_\varepsilon$.
	This line segment has to intersect with the boundary of $\mathcal{R}_\varepsilon$, namely, $\{\BU\in\mathcal{R}:H(\BU) = \varepsilon\}$ and the convexity of $\mathcal{R}_\varepsilon$ also implies that the intersection happens only once. The claim (ii) is proven.
	
	(iii)
	Since $\BU_0 = \oBU \in \mathcal{R}_\varepsilon$ and $\mathcal{U}(\Theta(\oBU,\hBU,H,\varepsilon);\oBU,\hBU) \in \mathcal{R}_\varepsilon$, the convexity of $\mathcal{R}_\varepsilon$ immediately implies (iii).
\end{proof}

Now, We introduce the BP limiter $\mathcal{B}$, which consists of two steps:
\begin{enumerate}[label=(\roman*),leftmargin=6mm]
	
	\item[\textbf{Step 1}:]	Enforce the positivity of $\rho$ and $z$ by a local scaling
	\begin{equation}\label{eq:1706}
		\hBU_{j,*}^{n,s}(x)
		=
		\Big(
		\theta_1 \hat\rho_j^{n,s}(x) +(1-\theta_1)\hbr_j^{n,s}, 
		\hat y_j^{n,s}(x), 
		\theta_2 \hat z_j^{n,s}(x) +(1-\theta_2)\hbz_j^{n,s}
		\Big)^\top,
	\end{equation}
	where the local scaling coefficients $\{\theta_r\}_{r = 1}^2$ are defined by
	\[
	\theta_r = \min_{\ell = 1,\dots,L}
	\Theta\Big(\hoBU_j^{n,s},\hBU_j^{n,s} (x^{\tt GL}_{j,\ell}),H_r,\varepsilon_r\Big), \;\;
	\varepsilon_r = \min \Big\{ 10^{-12} , H_r(\hoBU_j^{n,s}) \Big\},\; r = 1,2.
	\]
	Here, the $L$-point Gauss--Lobatto nodes in $I_j$ are denoted by $\{x^{\tt GL}_{j,\ell} \}_{\ell = 1}^L$, and the positive parameter $\varepsilon_r$ is introduced to enhance the robustness of above limiter under the influence of round-off errors. Also, please note that the local scaling \eqref{eq:1706} does not change the cell average, that is, $\hoBU_{j,*}^{n,s} = \hoBU_{j}^{n,s}$.
	
	\item[\textbf{Step 2}:]Enforce the remaining constraints by
	\begin{equation}\label{eq:1733}
		\BU_j^{n,s}(x) 
		=
		\theta_3 \hBU_{j,*}^{n,s}(x)
		+
		(1-\theta_3) \hoBU_{j}^{n,s},
	\end{equation}
	where the local scaling coefficients $\theta_3$ is defined by
	\[
	\theta_3 = \min_{r = 3,\dots,7} \min_{\ell = 1,\dots,L}
	\Theta\Big(\hoBU_{j,*}^{n,s},\hBU_{j,*}^{n,s} (x^{\tt GL}_{j,\ell}),H_r,\varepsilon_r\Big), \;
	\varepsilon_r = \min \Big\{ 10^{-12} , H_r(\hoBU_{j,*}^{n,s}) \Big\},\; r = 3,\dots,7.
	\]
	The local scaling \eqref{eq:1733} does not change the cell average: $\oBU_{j}^{n,s} = \hoBU_{j,*}^{n,s} = \hoBU_{j}^{n,s}$.
\end{enumerate}

\begin{lemma}\label{lem:1775}
	If the local polynomial $\hBU^{n,s}_j$ satisfies $\hoBU^{n,s}_j \in \Omega^{n,s}_j$, then the result of the BP limiter, $\BU^{n,s}_j = \mathcal{B}\big(\hBU^{n,s}_j;\Omega^{n,s}_j\big)$, satisfies
	$\BU^{n,s}_j(x^{\tt GL}_{j,\ell}) \in \Omega^{n,s}_j$,  $\ell = 1,\dots,L$.
\end{lemma}

\begin{proof}
	Let us consider an arbitrary $\ell \in \{1,\dots,L\}$.
	Notice that $H_1(\hbr_j^{n,s}) \ge \varepsilon_1$ and $H_1(\cdot)$ is concave in $\mathbb{R}^3$, \Cref{lem:1467} implies $\hat{\rho}_{j,*}^{n,s}(x^{\tt GL}_{j,\ell}) = H_1\big(\mathcal{U}(\theta_1;\hoBU_j^{n,s},\hBU_{j}^{n,s}(x^{\tt GL}_{j,\ell}))\big) \ge \varepsilon_1$, since $0 \le \theta_1 \le \Theta\big(\hoBU_j^{n,s},\hBU_j^{n,s} (x^{\tt GL}_{j,\ell}),H_1,\varepsilon_1\big)$. Similarly, we can prove that $\hat{z}_{j,*}^{n,s}(x^{\tt GL}_{j,\ell}) \ge \varepsilon_2$.
	Therefore, $\hBU_{j,*}^{n,s}(x^{\tt GL}_{j,\ell}) \in \Omega_0$ and $\Theta\big(\hoBU_{j,*}^{n,s},\hBU_{j,*}^{n,s} (x^{\tt GL}_{j,\ell}),H_r,\varepsilon_r\big) = 1$ for $r = 1,2$.
	
	Let us consider an arbitrary $r \in \{1,\dots,7\}$. Since $H_r(\hoBU_{j,*}^{n,s}) = H_r(\hoBU_{j}^{n,s}) \ge \varepsilon_r$, $H_r(\cdot)$ is concave in $\Omega_0$, and $0 \le \theta_3 \le \Theta\big(\hoBU_{j,*}^{n,s},\hBU_{j,*}^{n,s} (x^{\tt GL}_{j,\ell}),H_r,\varepsilon_r\big)$
	, \Cref{lem:1467} implies 
	$
	H_r\big( \BU_{j}^{n,s}(x^{\tt GL}_{j,\ell}) \big) 
	= H_r\big(\mathcal{U}(\theta_3;\hoBU_{j,*}^{n,s},\hBU_{j,*}^{n,s}(x^{\tt GL}_{j,\ell}))\big) \ge \varepsilon_r$, that is, $\BU_{j}^{n,s}(x^{\tt GL}_{j,\ell}) \in \Omega_{j}^{n,s}$.
\end{proof}

\begin{remark}
	When evaluating the function $\Theta$, if the equation $H\big(\mathcal{U}(\theta;\oBU,\hBU)\big) = \varepsilon$ has no explicit solution, we employ the Newton's method to solve it numerically.
\end{remark}

\begin{remark}
	Applying the BP limiter $\mathcal{B}$ defined in \eqref{eq:1706}--\eqref{eq:1733} to a numerical solution maintains its original order of accuracy; see \cite{wu2021minimum,zhang2011maximum,zhang2010maximum}.
\end{remark}

\subsection{Proving BP property of BP-OEDG schemes}\label{sec:1598}

We first focus on a single forward Euler (FE) time step: $\BU_h^* = \BU_h + \Delta t \mathcal{L} (\BU_h)$.
The cell average of resulting solution $\BU_h^*$ in $I_j$ can be expressed as
\begin{equation}\label{eq:2121}
	\oBU^*_j
	=
	\oBU_j
	+
	\frac{\Delta t}{\Delta x}
	\left[
	\widehat\BF(\BU_{j-\frac12}^{-},\BU_{j-\frac12}^{+})
	-
	\widehat\BF(\BU_{j+\frac12}^{-},\BU_{j+\frac12}^{+})
	\right].
\end{equation}

\begin{lemma}[BP property of FE step]\label{lem:2154}
	Considering an invariant domain $\Omega$ in the form of \eqref{eq:270}, if 
	$\BU^{-}_{j-\frac12} \in \Omega$, $\BU^{+}_{j+\frac12} \in \Omega$, $\BU(x^{\tt GL}_{j,\ell}) \in \Omega$ for any $\ell = 1,\dots,L$,
	then the result of the FE step \eqref{eq:2121} with 
	\[
	\alpha \ge \max \big\{ \alpha_{\max}(\BU^{-}_{j-\frac12},\BU^{+}_{j-\frac12}), \alpha_{\max}(\BU^{-}_{j+\frac12},\BU^{+}_{j+\frac12}) \big\}
	\]
	satisfies $\oBU^*_j \in \Omega$, under the CFL condition 
	\begin{equation}\label{eq:2300}
		\alpha \frac{\Delta t}{\Delta x} \le \omega^{\tt GL}_1.
	\end{equation}
\end{lemma}

\begin{proof}
	Since $\BU_j(x) \in \mathbb{P}^k$, the $L$-point Gauss--Lobatto quadrature holds exactly:
	\[
	\oBU_j
	= \frac{1}{\Delta x} \int_{I_j} \BU_j(x) \, \textrm{d}x
	= \sum_{\ell = 1}^L \omega^{\tt GL}_\ell \BU_j(x^{\tt GL}_{j,\ell}).
	\]
	Since $\omega^{\tt GL}_1 = \omega^{\tt GL}_L$, $x^{\tt GL}_{j,1} = x_{j-\frac12}$, $x^{\tt GL}_{j,L} = x_{j+\frac12}$, we can rewrite \eqref{eq:2121} as
	\begin{equation}\label{eq:2265}
		\oBU^*_j
		=
		\sum_{\ell = 2}^L 
		\omega^{\tt GL}_\ell \BU_j(x^{\tt GL}_{j,\ell})
		+
		(\omega^{\tt GL}_1-\lambda)
		\BU_j(x^{\tt GL}_{j,1})
		+
		(\omega^{\tt GL}_1-\lambda)
		\BU_j(x^{\tt GL}_{j,L})
		+
		\lambda \BU_j^{-}
		+
		\lambda \BU_j^{+},
	\end{equation}
	where $\lambda := \alpha \frac{\Delta t}{\Delta x} \le \omega^{\tt GL}_1$, and
	$\BU_j^{\pm} := \mathcal{S}_{\tt GLF}(\BU^{\pm}_{j-\frac12},\BU^{\pm}_{j+\frac12})$.
	\Cref{thm:1471} implies that $\BU^{\pm}_j \in \Omega$. 
	The convex combination \eqref{eq:2265} implies that $\oBU^*_j \in \Omega$ due to the convexity of $\Omega$ (\Cref{lem:466}).
	The proof is completed.
\end{proof}

We are now ready to prove the BP property of BP-OEDG scheme \eqref{BPOE} with respect to the global and local invariant domains in the following two theorems.

\begin{theorem}[global BP property]\label{lem:2350}
	If the global invariant domains $\Omega^{n,s}$ defined in \eqref{eq:1795}--\eqref{eq:1807} are considered, then the scheme \eqref{BPOE} with $\alpha \ge \max_j \alpha_{\max}(\BU^-_{j+\frac12},\BU^+_{j+\frac12})$ and under the CFL condition \eqref{eq:2300} is globally BP, namely,
	\begin{equation}\label{eq:2177}
		\BU^n_j(x^{\tt GL}_{j,\ell}) \in \Omega^{n} 
		~~ \forall j ~ \forall \ell
		~~ \Longrightarrow ~~
		\BU^{n+1}_j(x^{\tt GL}_{j,\ell}) \in \Omega^{n+1} 
		~~ \forall j ~ \forall \ell.
	\end{equation}
	
\end{theorem}

\begin{proof}
	Since $\BU_j^n \in (\mathbb{P}^k)^3$ for all $j$, the $L$-point Gauss--Lobatto quadrature holds exactly, namely,
	$
	\oBU_j^n 
	= 
	\frac{1}{\Delta x} \int_{I_j} \BU^n_j(x) \, \textrm{d} x
	=
	\sum_{\ell = 1}^{L} \omega^{\tt GL} \BU^n_j(x^{\tt GL}_{j,\ell}),
	$
	which decomposes $\oBU_j^n$ as a convex combination of $\BU^n_j(x^{\tt GL}_{j,\ell}) \in \Omega^n$, which leads to $\oBU_j^n \in \Omega^{n} = \Omega^{n,0} \subseteq \Omega^{n,1}$ due to the convexity of $\Omega$ (\Cref{lem:466}) and the hierarchy of global invariant domains (\Cref{lem:1889}).
	Additionally, all the cell-edge values of $\BU^n_h$ belong to $\Omega^{n,1}$, since
	(i) For $j = 1,\dots,N_x$, $\BU^n_h(x_{j-\frac12}^{+}) = \BU^n_j(x^{\tt GL}_{j,1}) \in \Omega^{n} = \Omega^{n,0} \subseteq \Omega^{n,1}$, $\BU^n_h(x_{j+\frac12}^{-}) = \BU^n_j(x^{\tt GL}_{j,L}) \in \Omega^{n} = \Omega^{n,0} \subseteq \Omega^{n,1}$;
	(ii) By \Cref{lem:1872}, $\BU^n_h(x_{\frac12}^{-}) \in \Omega^{n,1}$ and $\BU^n_h(x_{N_x+\frac12}^{+}) \in \Omega^{n,1}$.
	Therefore, we can apply \Cref{lem:2154} on $\BU^n_h$ to obtain $\toBU^{n,1}_j \in \Omega^{n,1}$ for all $j$,
	and then \Cref{lem:1625b} to obtain $\hoBU^{n,1}_j \in \Omega^{n,1}$ for all $j$.
	After applying the BP limiter to $\hBU^{n,1}_h$, \Cref{lem:1775} implies that the first RK stage is BP, namely, $\oBU^{n,1}_j \in \Omega^{n,1}, \BU^{n,1}_h(x^{\tt GL}_{j,\ell}) \in \Omega^{n,1}$ for all $j$ and $\ell$.
	
	Using similar arguments, we can prove the BP properties of the second and third RK stages:
	$\BU^{n,s}_h(x^{\tt GL}_{j,\ell}) \in \Omega^{n,s}$ for any $j$, $\ell$, and $s = 2,3$.
	Noticing that $\BU^{n+1}_h = \BU^{n,3}_h$ and $\Omega^{n+1} = \Omega^{n,3}$, the claim \eqref{eq:2177} is proven.
\end{proof}

\begin{theorem}[local BP property]\label{lem:2394}
	If the local invariant domains $\Omega^{n,s}_j$ defined in \eqref{eq:1851} are considered, then the scheme \eqref{BPOE} with $\alpha \ge \max_j \alpha_{\max}(\BU^-_{j+\frac12},\BU^+_{j+\frac12})$ and under the CFL condition \eqref{eq:2300} is locally BP, namely,
	\begin{equation}\label{eq:2274}
		\BU^n_j(x^{\tt GL}_{j,\ell}) \in \Omega^{n}_j
		~~ \forall j ~ \forall \ell
		~~ \Longrightarrow ~~
		\BU^{n+1}_j(x^{\tt GL}_{j,\ell}) \in \Omega^{n+1}_j
		~~ \forall j ~ \forall \ell.
	\end{equation}
\end{theorem}

\begin{proof}
	Let us consider four cases:
	\begin{itemize}[leftmargin= 6mm]
		\item[(i)] If $j = 2,\dots,N_x$, then from \eqref{eq:1851}--\eqref{eq:1982}, we can obtain
		\begin{equation*}
			\left.
			\begin{aligned}
				&& \BU_h^{n,0}(x_{j-\frac12}^{-}) = \BU^{n}_{j-1}(x^{\tt GL}_{j-1,L})\\
				&& \BU^{n}_{j-1}(x^{\tt GL}_{j-1,L}) \in \Omega^{n, 0}_{j-1} \subseteq \Omega^{n,0} \subseteq \Omega^{n,1} \\
				&& \BU^{n}_{j-1}(x^{-}_{j-\frac12}) \in \Omega^{n, 0}_{j, *}
			\end{aligned}
			\right\}
			\; \Rightarrow \;
			\BU_h^{n,0}(x_{j-\frac12}^{-}) \in \Omega^{n,1} \cap \Omega^{n, 0}_{j, *} = \Omega^{n, 1}_{j}.
		\end{equation*}
		\item[(ii)] If $j = 1$, $\BU^{n,0}_h(x_{\frac12}^{-}) \in \Omega^{n,1} \cap \Omega^{n,0}_{1,*} = \Omega^{n,1}_{1}$. 
		\item[(iii)] If $j = 1,\dots,N_x-1$, then from \eqref{eq:1851}--\eqref{eq:1982}, we can obtain
		\begin{equation*}
			\left.
			\begin{aligned}
				&& \BU_h^{n,0}(x_{j+\frac12}^{+}) = \BU^{n}_{j+1}(x^{\tt GL}_{j+1,1}) \\
				&& \BU^{n}_{j+1}(x^{\tt GL}_{j+1,1}) \in \Omega^{n, 0}_{j+1} \subseteq \Omega^{n,0} \subseteq \Omega^{n,1} \\
				&& \BU^{n}_{j+1}(x^{+}_{j+\frac12}) \in \Omega^{n, 0}_{j, *}
			\end{aligned}
			\right\}
			\;\Rightarrow \;
			\BU_h^{n,0}(x_{j+\frac12}^{+}) \in \Omega^{n,1} \cap \Omega^{n, 0}_{j, *} = \Omega^{n, 1}_{j}.
		\end{equation*}
		\item[(iv)] If $j = N_x$, $\BU^{n,0}_h(x_{N_x +\frac12}^{+}) \in \Omega^{n,1} \cap \Omega^{n,0}_{N_x,*} = \Omega^{n,1}_{N_x}$.
	\end{itemize}
	To summarize, $\BU_h^{n,0}(x_{j-\frac12}^{-}) \in \Omega^{n,1}_j$ and $\BU_h^{n,0}(x_{j+\frac12}^{+}) \in \Omega^{n,1}_j$ for $j = 1,\dots,N_x$.
	On the other hand, \eqref{eq:1851}--\eqref{eq:1927} implies that 
	$\BU_h(x^{\tt GL}_{j,\ell}) \in \Omega^{n,0} \cap \Omega^{n,0}_{j,*} \subseteq \Omega^{n,1} \cap \Omega^{n,0}_{j,*} = \Omega^{n,1}_j$ for all $j$.
	Therefore, we can use \Cref{lem:2154} to obtain $\toBU^{n,1}_j \in \Omega^{n,1}_j$ for all $j$, and then \Cref{lem:1625b} to obtain $\hoBU^{n,1}_j \in \Omega^{n,1}_j$ for all $j$.
	Following this, we may apply \Cref{lem:1775} to conclude with the local BP property of the first RK stage, namely, $\BU^{n,1}_j(x^{\tt GL}_{j,\ell}) \in \Omega^{n,1}_j$ for all $j$ and $\ell$.
	
	Via the similar arguments, we can prove the local BP properties of the second and the third RK stages:
	$\BU^{n,s}_h(x^{\tt GL}_{j,\ell}) \in \Omega^{n,s}$ for all $j$, $\ell$, and $s = 2,3$.
	Noticing $\BU^{n+1}_h = \BU^{n,3}_h$  and $\Omega^{n+1} = \Omega^{n,3}$, the claim \eqref{eq:2274} is proven.
\end{proof}

\begin{remark}
	The global and local BP properties can be similarly proven for other high-order SSP (RK or multi-step) time discretizations that can be expressed as a convex combination of forward Euler steps.
\end{remark}

\section{Numerical Examples}\label{sec:num}

In this section, we conduct numerical experiments to demonstrate the high-order accuracy, BP property, and overall effectiveness of the proposed schemes for both the original and the AP ARZ models. 
We utilize the third-order DG spatial discretization (\( k=2 \)) combined with third-order SSP-RK time discretization. Unless noted otherwise, the CFL number is 0.08, and \( \kappa = \gamma+1 \). To generate the sampling mesh for estimating the global invariant domain at the initial time, we have used \( N_s = 1000 N_x \) in all the numerical results presented in this section. 
For comparison purposes, we consider four numerical schemes:
\begin{itemize}[topsep=3pt,itemsep=0pt,parsep=1pt,leftmargin=3mm]
	\item {\bf BP-OEDG schemes}: the numerical schemes introduced in \Cref{sec:BP-OEDG}. If not specified, the local invariant domains defined in \Cref{sec:local} are adopted;
	\item {\bf nonBP-OEDG schemes}: the OEDG schemes without the BP limiter;
	\item {\bf BPDG schemes}: the DG schemes with BP limiter and without OE procedure;
	\item {\bf conventional DG schemes}: the DG schemes without the OE procedure and BP limiter.
\end{itemize}

\subsection{Numerical experiments on a single road segment}\label{sec:single}

\begin{exa} \label{exam1}
	Consider the following smooth initial condition:
	\begin{equation*}
		\rho(x,0) = 0.05 \left(1 + \cos{\pi x} \right) + 10^{-8}, \quad v(x,0) \equiv 0.15, \quad c(x,0) \equiv 1,
	\end{equation*}
	with parameters \( v_{\rm ref} = 0.01 \), and \( \gamma \in \{0,1,2\} \). The computational domain is \([0,1]\) with periodic boundary conditions at both boundaries. The exact solution reads
	\begin{equation*}
		\rho(x, t) = 0.05 \Big[ 1 + \cos{(\pi (x - 0.15 \, t))} \Big] + 10^{-8}, \quad v(x, t) \equiv 0.15, \quad c(x,t) \equiv 1.
	\end{equation*}
	We compute the numerical solution up to \( t = 0.05 \) using the BP-OEDG scheme. The \( L_1 \) errors of the obtained numerical results are presented in Table \ref{Table1}, confirming the expected third-order accuracy. 
	Please note that near-vacuum states appear in the solution. Employing nonBP-OEDG schemes for this example may produce numerical solutions with negative density, which can lead to simulation failure. For example, if the nonBP-OEDG method is used, the simulation would fail due to negative density in the numerical results; see \Cref{tab:2830} for the failure time.

	\begin{table}[th!]
		\renewcommand\arraystretch{1.2}
		\caption{\sf Example \ref{exam1}, the $L_{1}$ errors in $\rho$ and the convergence rates for various $\gamma$ values.}
		\begin{center}
			\begin{tabular}{clclclc}
				
				\toprule[1.5pt]
				
				\multicolumn{1}{c}{\multirow{2} * {$N$}} & \multicolumn{2}{c}{ $\gamma = 0$} & \multicolumn{2}{c}{$\gamma = 1$} & \multicolumn{2}{c}{$\gamma = 2$} \\ 
				\cmidrule(lr){2-3} \cmidrule(lr){4-5} \cmidrule(lr){6-7}
				\multicolumn{1}{c}{} &$L_1$ error & order & $L_1$ error & order & $L_1$ error & order \\ 
				
				\midrule[1.5pt]
				
				10  &     9.95e-6  &       -- &     1.02e-4  &       -- &     1.02e-4  &       -- \\ 
				20  &     2.24e-6  &     2.15 &     1.73e-5  &     2.57 &     1.72e-5  &     2.57 \\ 
				40  &     4.64e-7  &     2.27 &     3.48e-7  &     5.63 &     3.48e-7  &     5.63 \\ 
				80  &     5.13e-8  &     3.18 &     4.68e-8  &     2.89 &     4.66e-8  &     2.90 \\ 
				160 &     7.70e-9  &     2.74 &     9.97e-9  &     2.23 &     5.54e-9  &     3.07 \\ 
				320 &     9.79e-10 &     2.97 &     5.65e-10 &     4.14 &     5.64e-10 &     3.30 \\
				
				\bottomrule[1.5pt]
				
			\end{tabular}
		\end{center}
		\label{Table1}
	\end{table}
	
	\begin{table}[th!]
		\renewcommand\arraystretch{1.2}
		\caption{\sf Example \ref{exam1}, the time of simulation failure due to negative density in numerical results of the nonBP-OEDG scheme.}
		\begin{center}
			\begin{tabular}{cccc}
				
				\toprule[1.5pt]
				
				$N$ & $\gamma = 0$ & $\gamma = 1$ & $\gamma = 2$ \\ 
				
				\midrule[1.5pt]
				
				160 
				& $t \approx 2.17 \times 10^{-2}$ 
				& $t \approx 1.81 \times 10^{-2}$
				& $t \approx 1.51 \times 10^{-2}$  
				\\
				320 
				& $t \approx 1.74 \times 10^{-2}$ 
				& $t \approx 9.03 \times 10^{-3}$ 
				& $t \approx 9.08 \times 10^{-3}$ 
				\\
				\bottomrule[1.5pt]
			\end{tabular}
		\end{center}
		\label{tab:2830}
	\end{table}
	
\end{exa}

\begin{exa} \label{exam2}
	This example verifies the effectiveness and necessity of the OE procedure for eliminating oscillations. 
	Consider the Riemann problem with initial data:
	\[
	(\rho,v,c)(x,0)
	=
	\begin{dcases}
		(0.8,0.1,1)
		&
		{\rm if }~ x \le 0.5,
		\\
		(0.5,0.2,1)
		&
		{\rm otherwise},
	\end{dcases}
	\]
	with $\gamma = 2$. 	
	\Cref{oscillationFig} presents the numerical solutions at $t = 0.4$ obtained using BP-OEDG, BPDG, and conventional DG schemes. 
	One can see that the BPDG and conventional DG schemes introduce spurious numerical oscillations, while the results of BP-OEDG scheme are free of oscillations. 
	
	\begin{figure}[t!]
		\begin{center}
			\includegraphics[width = 0.49\textwidth,trim=25 15 35 15,clip]{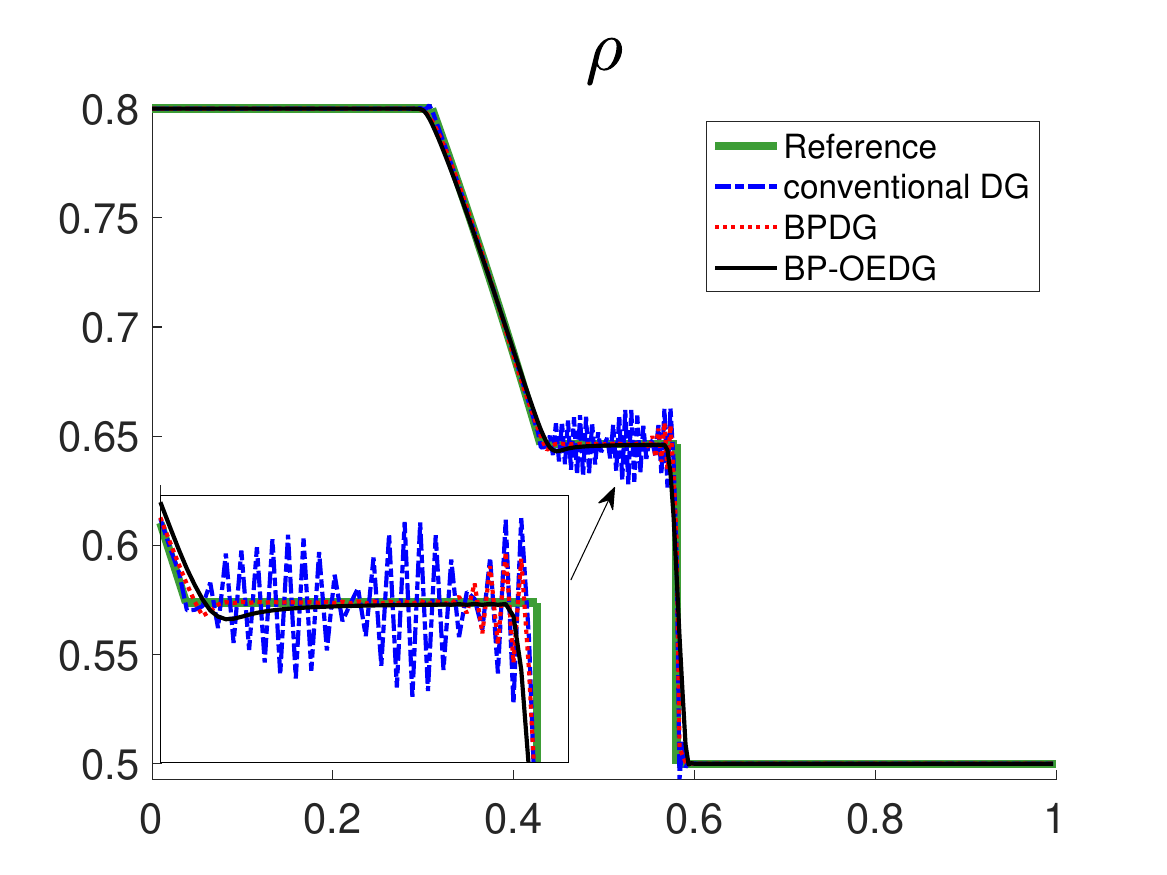}
			\includegraphics[width = 0.49\textwidth,trim=25 15 35 15,clip]{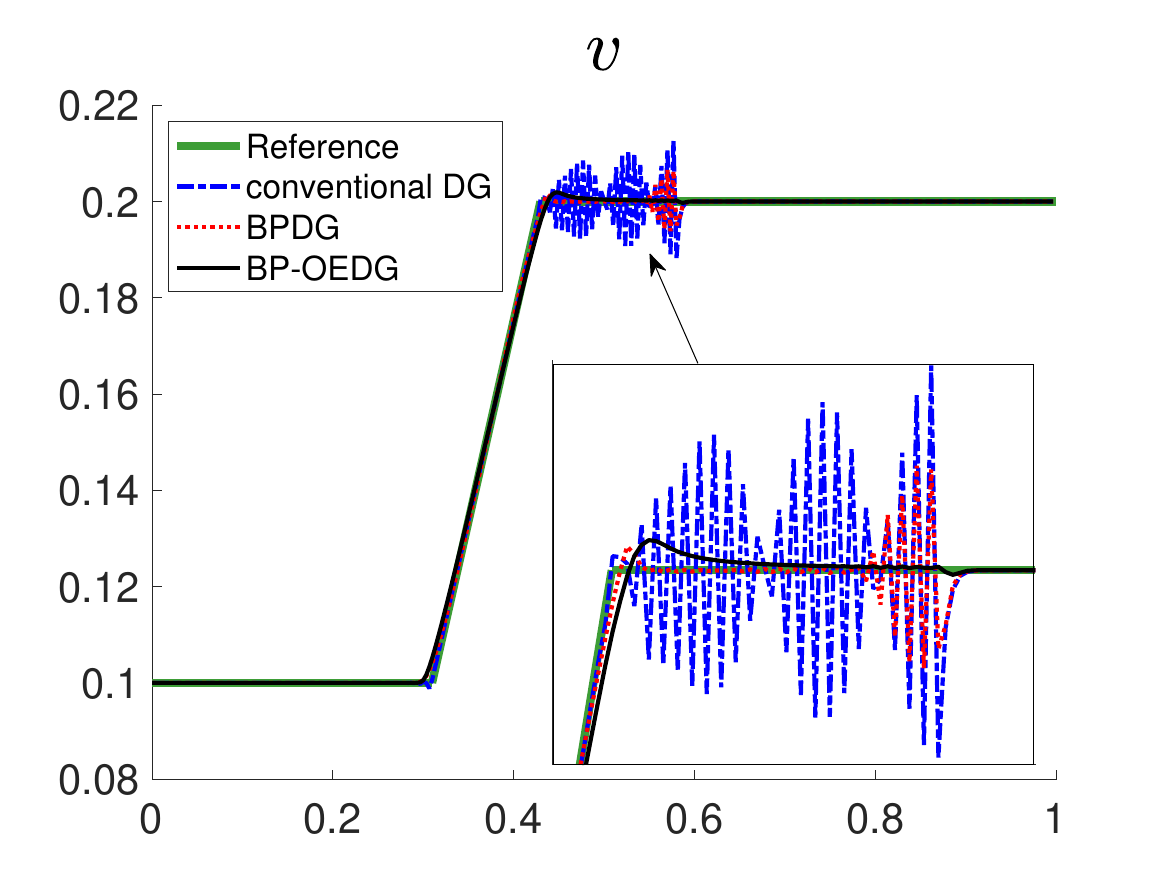}
			\caption{\sf Example \ref{exam2}, the numerical results of density and velocity obtained by using the BP-OEDG, BPDG, and conventional DG schemes, $t = 0.4$, $\Delta x = 1/300$.}
			\label{oscillationFig}
		\end{center}
	\end{figure}
	
\end{exa}

\begin{exa} \label{exam2b}
	In this example, we consider several Riemann problems with initial conditions in the following form:
	\begin{equation*}\label{eq:1802}
		(\rho,v,c)(x,0)
		\equiv
		\begin{dcases}
			(\,\rho_L,v_L,c_L)
			&
			{\rm if }~ x \le 0.65,
			\\
			(\,\rho_R,v_R,c_R)
			&
			{\rm otherwise}.
		\end{dcases}
	\end{equation*}
	Table \ref{tab:exam2b} lists several sets of parameters $(\rho_L,v_L,c_L,\rho_R,v_R,c_R)$.
	We use the BP-OEDG scheme to solve these test problems, and the obtained numerical results are depicted in Figures \ref{fig:exam2b_1}--\ref{fig:exam2b_2}, demonstrating the ability of the proposed BP-OEDG scheme to effectively capture various wave structures.
	
	\begin{table}[h!]
		\renewcommand\arraystretch{1.2}
		\centering
		\caption{\sf Example \ref{exam2b}, initial conditions of Riemann problem tests.}
		\begin{tabular}{l lll lll c}
			
			\toprule[1.5pt]
			
			\multirow{2}{*}{ Test } & \multicolumn{3}{c}{Riemann problem left state} & \multicolumn{3}{c}{Riemann problem right state} & \multirow{2}{*}{ $\gamma$ values } \\
			
			\cmidrule(lr){2-4} \cmidrule(lr){5-7}
			
			& $\rho_L$ & $v_L$ & $c_L$ & $\rho_R$ & $v_R$ & $c_R$ & \\
			
			\midrule[1.5pt]
			
			{T1a} & 0.5 & 0.2 & 1.0 & 0.9 & $2\times 10^{-11}$ & 1.0 & 0 and 2 \\
			{T1b} & 0.5 & 0.2 & 1.1 & 0.9 & $2\times 10^{-11}$ & 0.9 & 1 and 2 \\
			{T2a} & $1\times10^{-12}$ & 0.4 & 1.0  & 0.8 & 0.1 & 1.0 & 0, 1, and 2 \\ 
			{T2b} & $1\times10^{-12}$ & 10 & 1.0  & 0.8 & 0.5 & 1.0 & 0, 1, and 2 \\ 
			{T3 } & 0.8 & $5\times 10^{-12}$ & 1.0  & $1\times10^{-11}$ & 0.25 & 1.0 & 0 and 2 \\ 
			
			\bottomrule[1.5pt]
			
		\end{tabular}
		\label{tab:exam2b}
	\end{table}
	
	\begin{figure}[th!]
		\centerline{
			\begin{subfigure}[t][][t]{0.46\textwidth}
				\centerline{
					\includegraphics[trim=21 15 41 15,clip,width=0.5\textwidth]{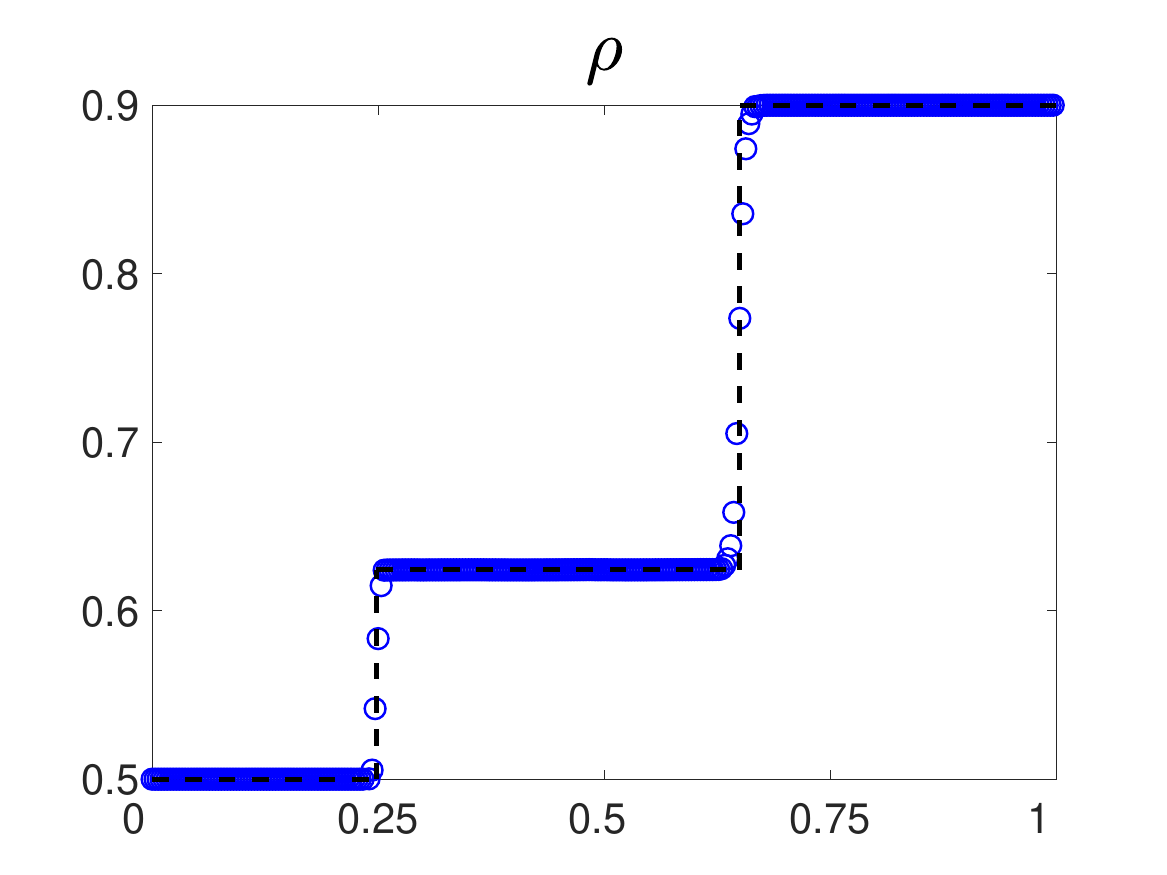}
					\includegraphics[trim=21 15 41 15,clip,width=0.5\textwidth]{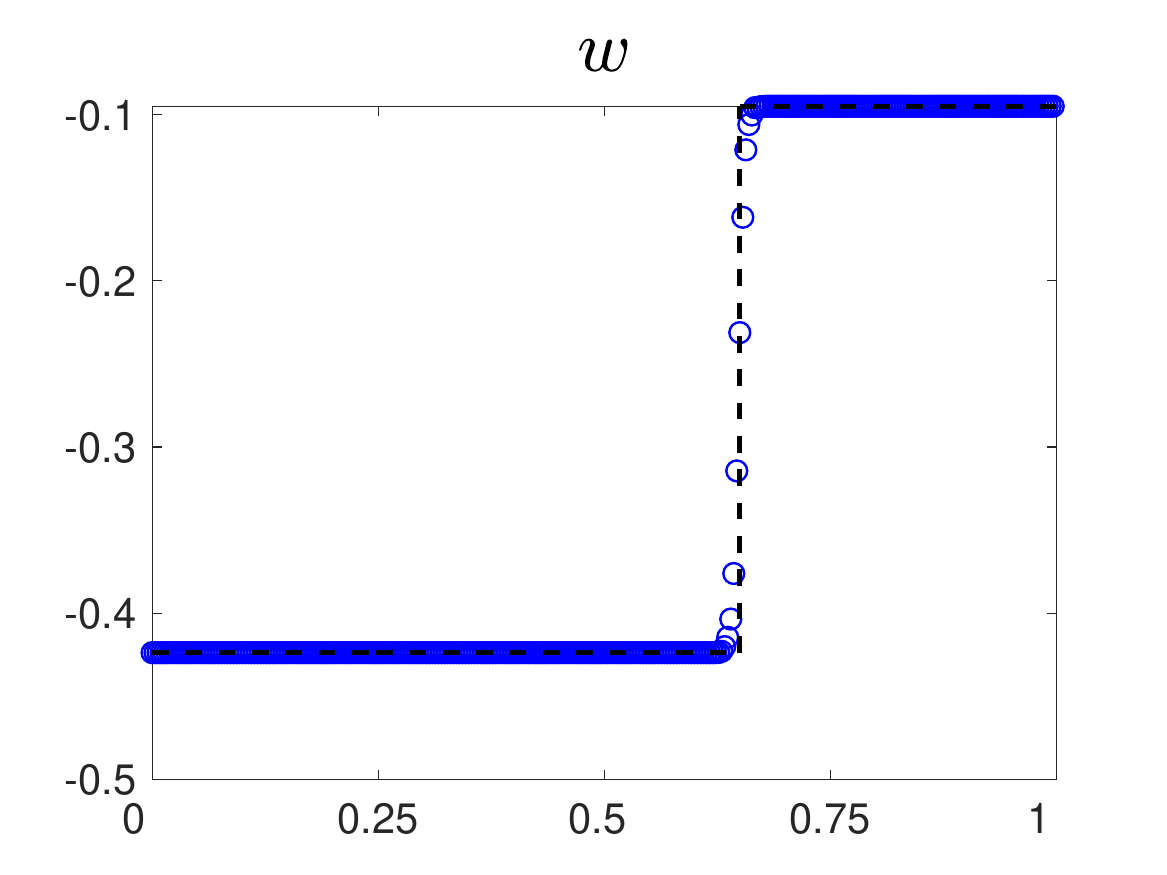}
				}
				\centerline{
					\includegraphics[trim=21 15 41 15,clip,width=0.5\textwidth]{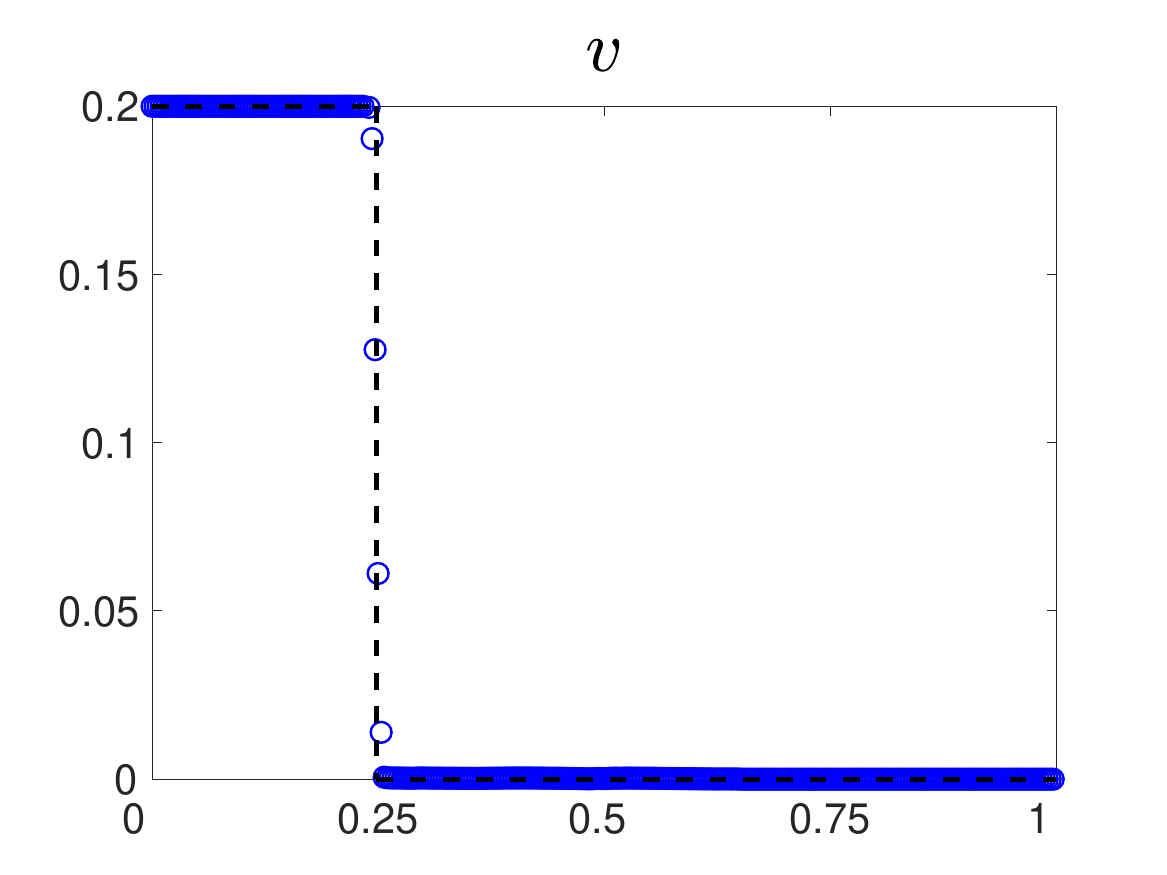}
					\includegraphics[trim=21 15 41 15,clip,width=0.5\textwidth]{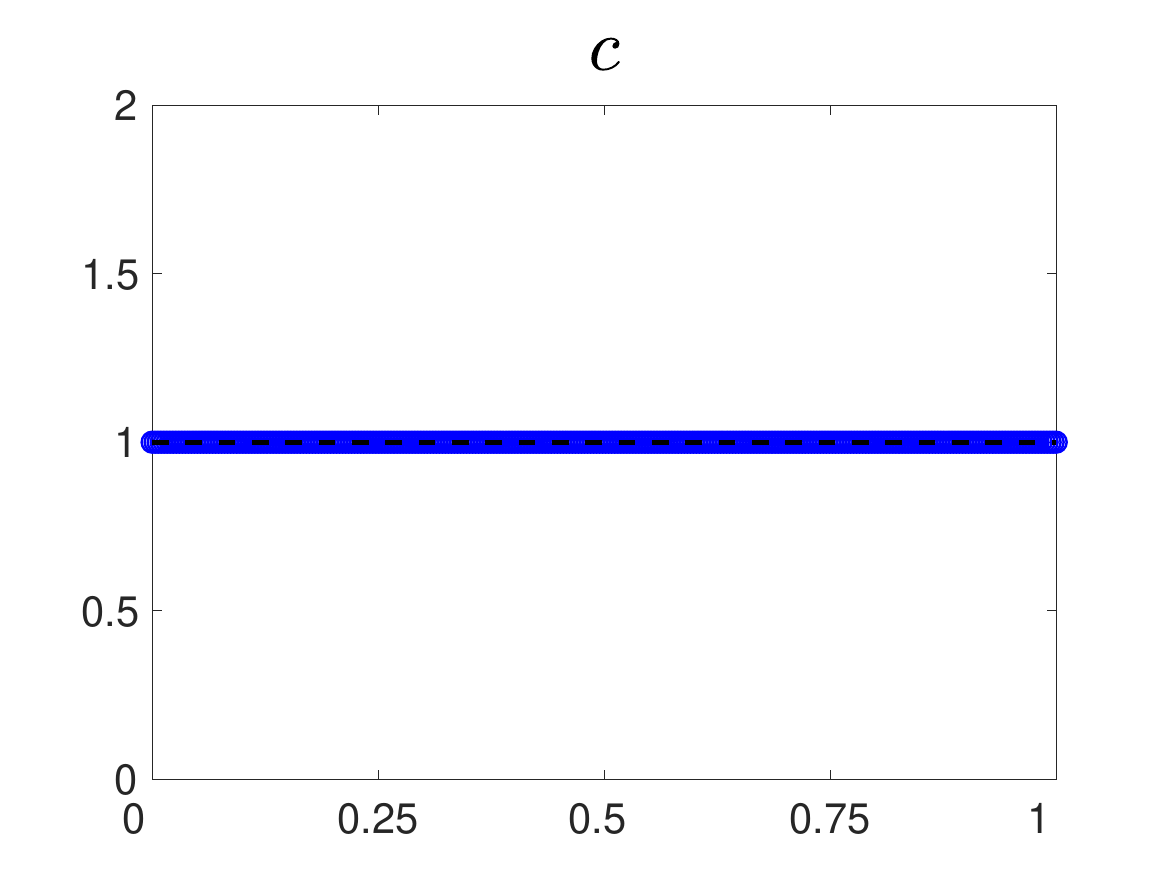}
				}
				\caption{\sf Test T1a, $\gamma = 0$.}
			\end{subfigure}
			\begin{subfigure}[t][][t]{0.46\textwidth}
				\centerline{
					\includegraphics[trim=21 15 41 15,clip,width=0.5\textwidth]{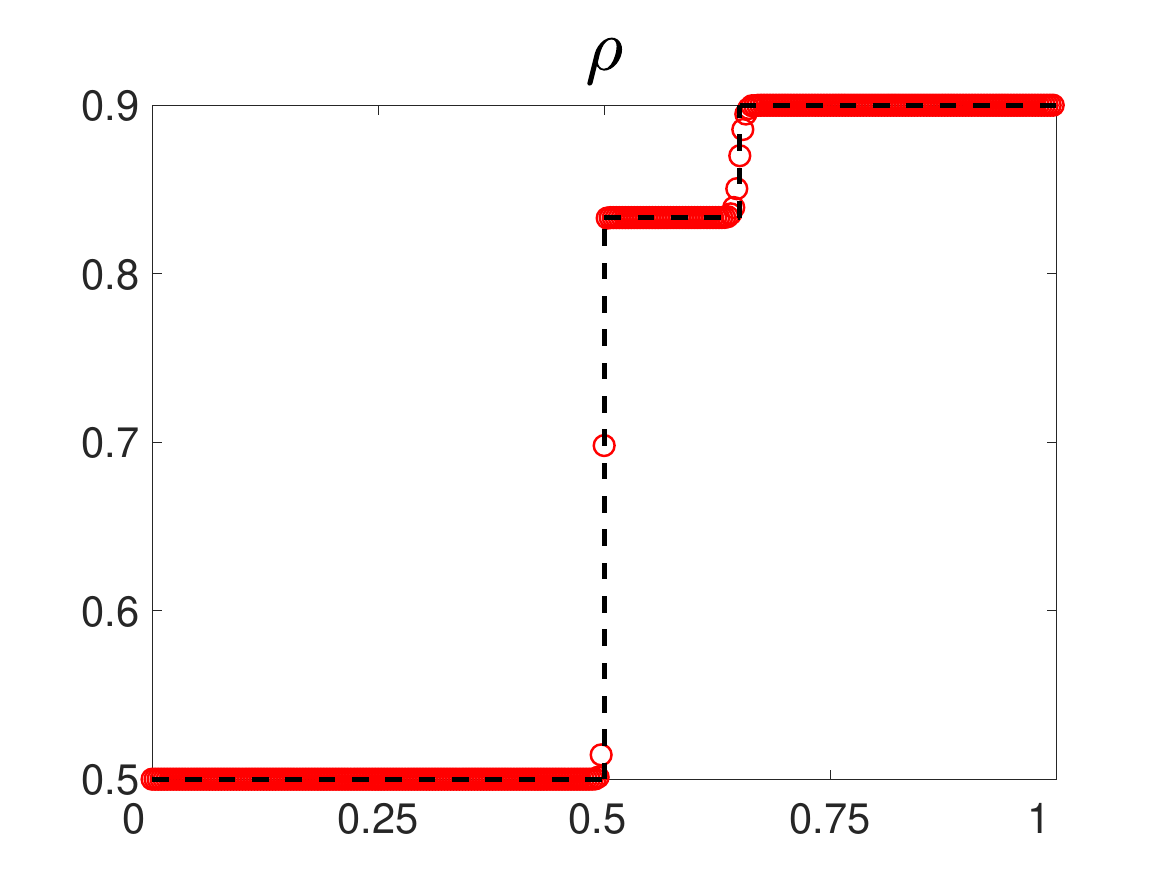}
					\includegraphics[trim=21 15 41 15,clip,width=0.5\textwidth]{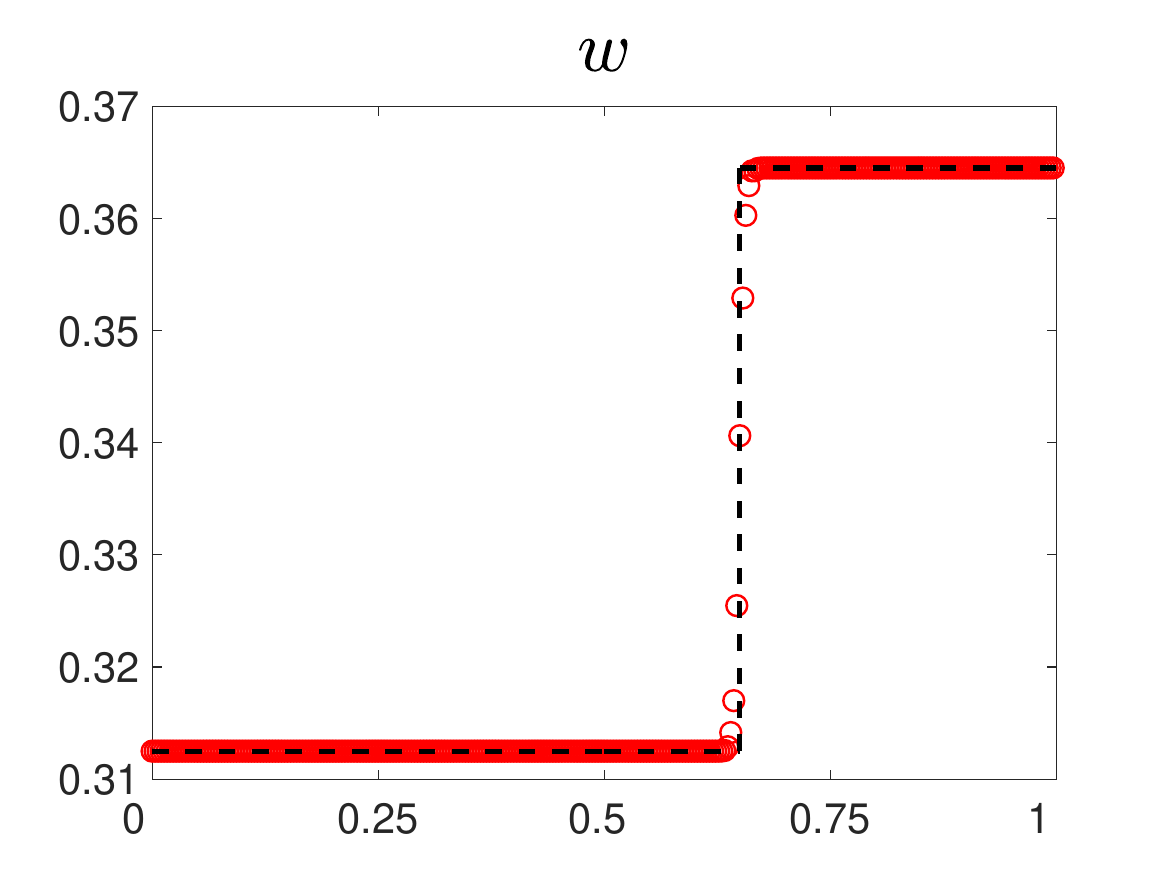}
				}
				\centerline{
					\includegraphics[trim=21 15 41 15,clip,width=0.5\textwidth]{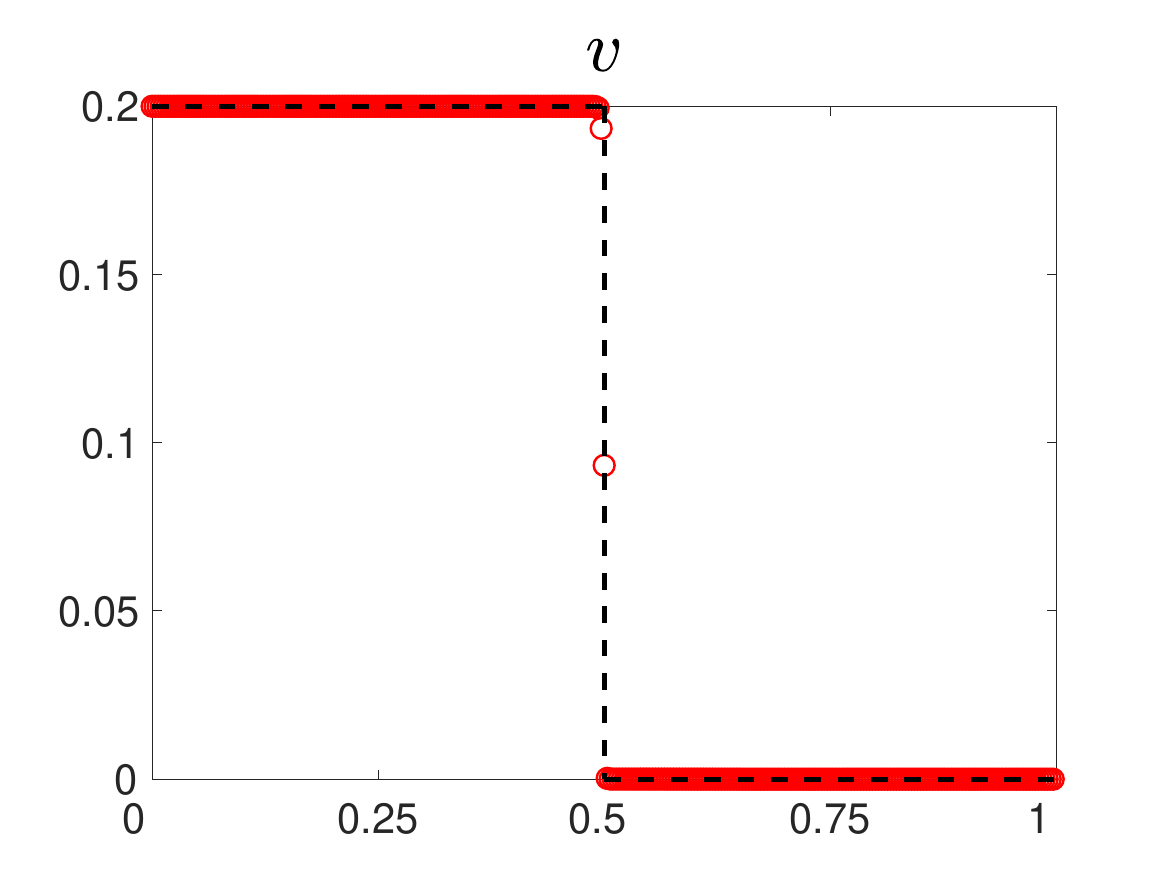}
					\includegraphics[trim=21 15 41 15,clip,width=0.5\textwidth]{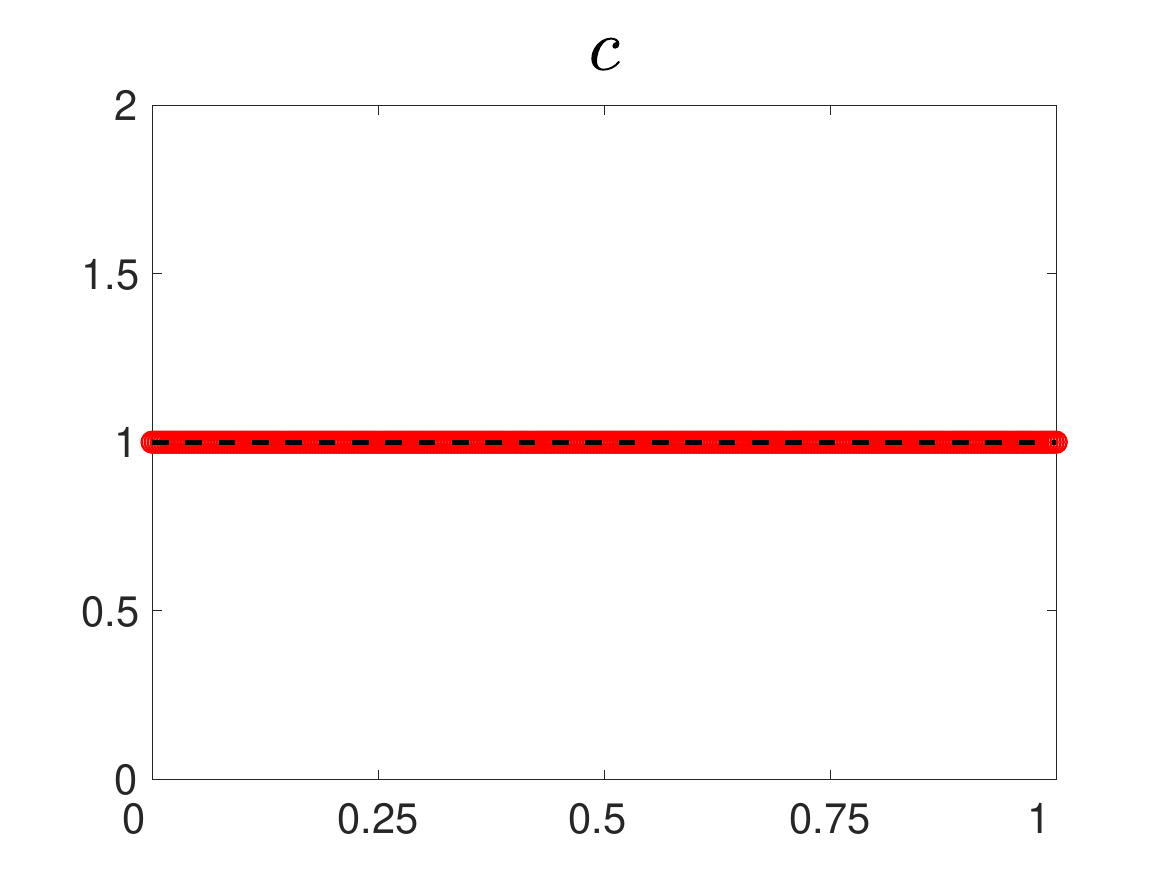}
				}
				\caption{\sf Test T1a, $\gamma = 2$.}
			\end{subfigure}
		}
		\centerline{
			\begin{subfigure}[t][][t]{0.46\textwidth}
				\centerline{
					\includegraphics[trim=25 15 35 15,clip,width=0.5\textwidth]{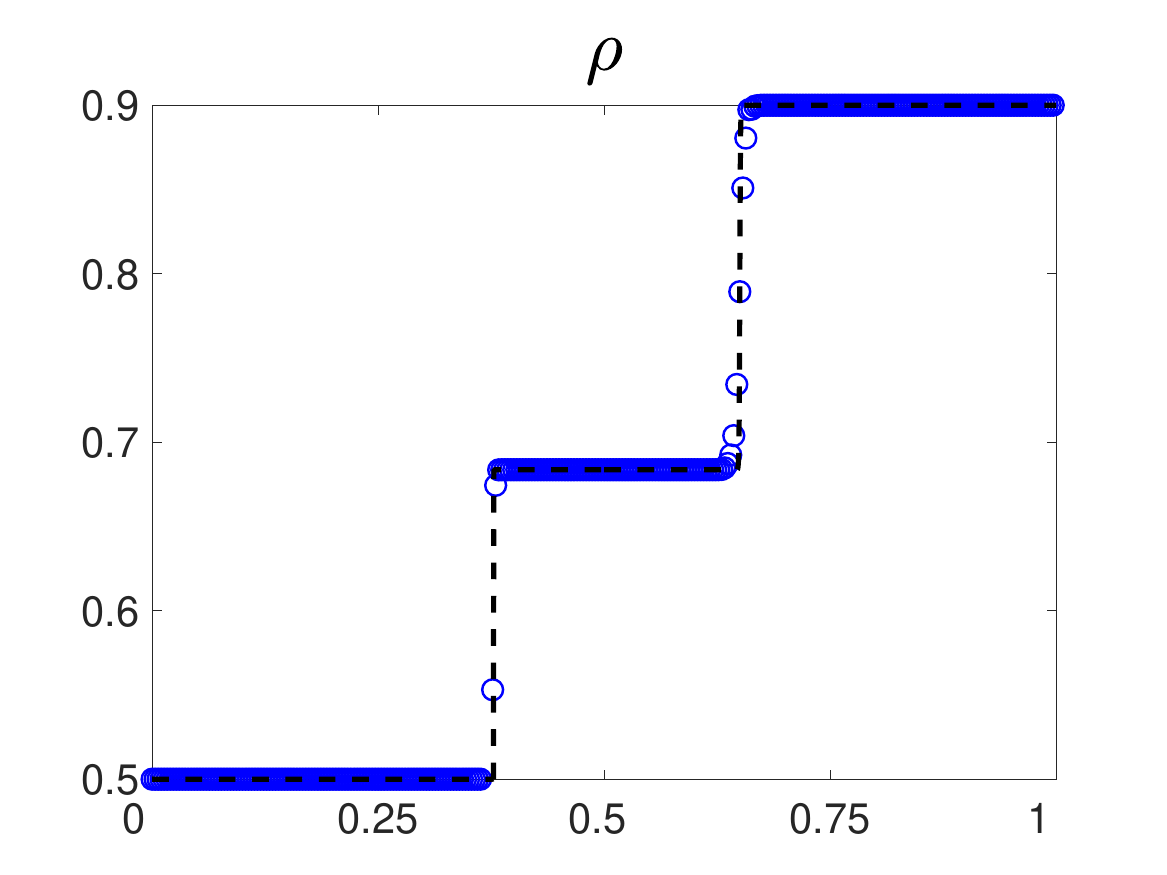}
					\includegraphics[trim=25 15 35 15,clip,width=0.5\textwidth]{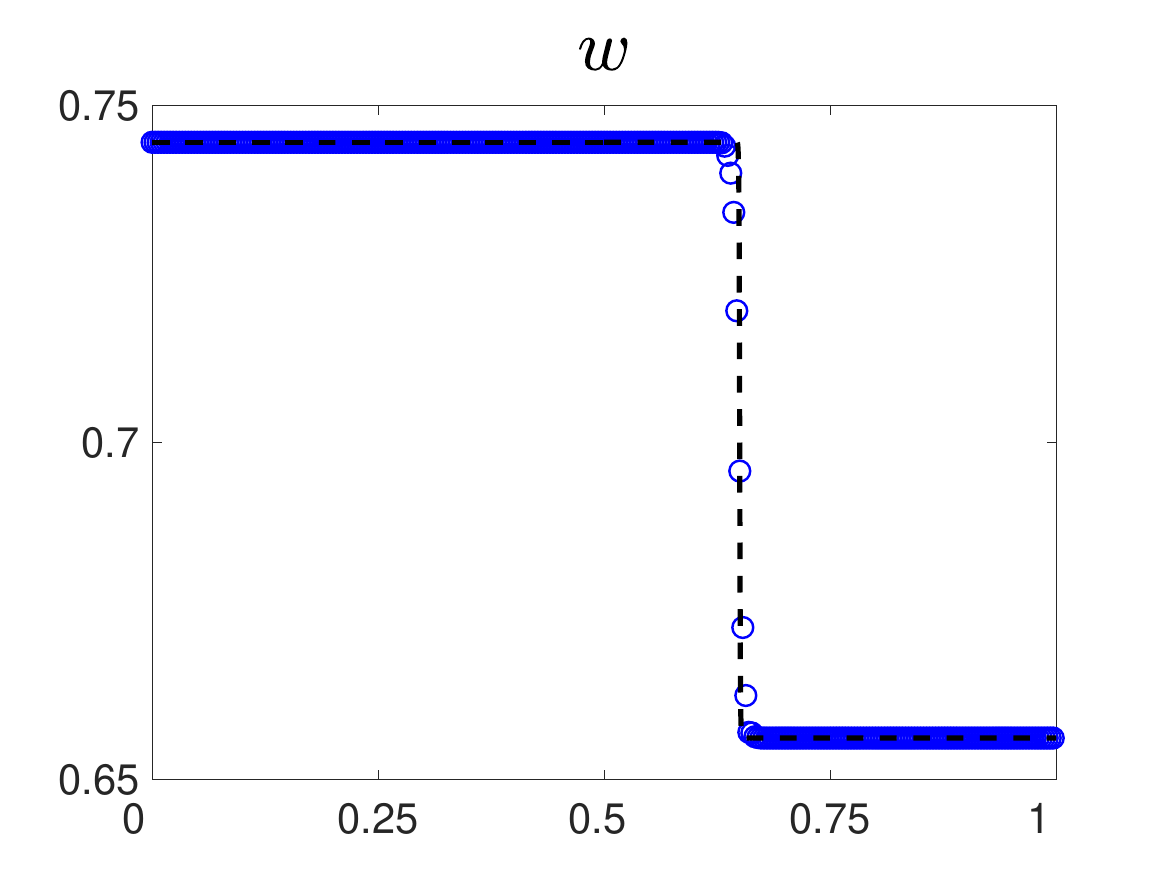}
				}
				\centerline{
					\includegraphics[trim=25 15 35 15,clip,width=0.5\textwidth]{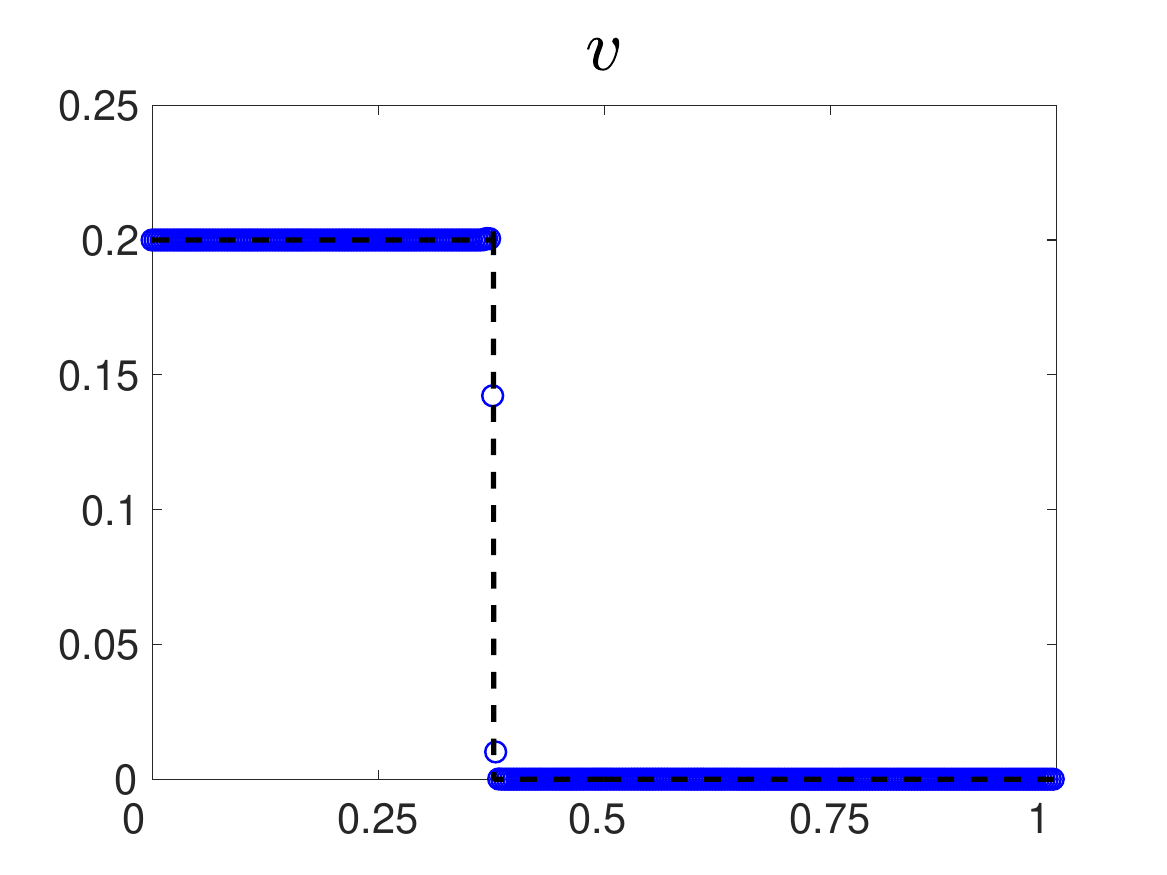}
					\includegraphics[trim=25 15 35 15,clip,width=0.5\textwidth]{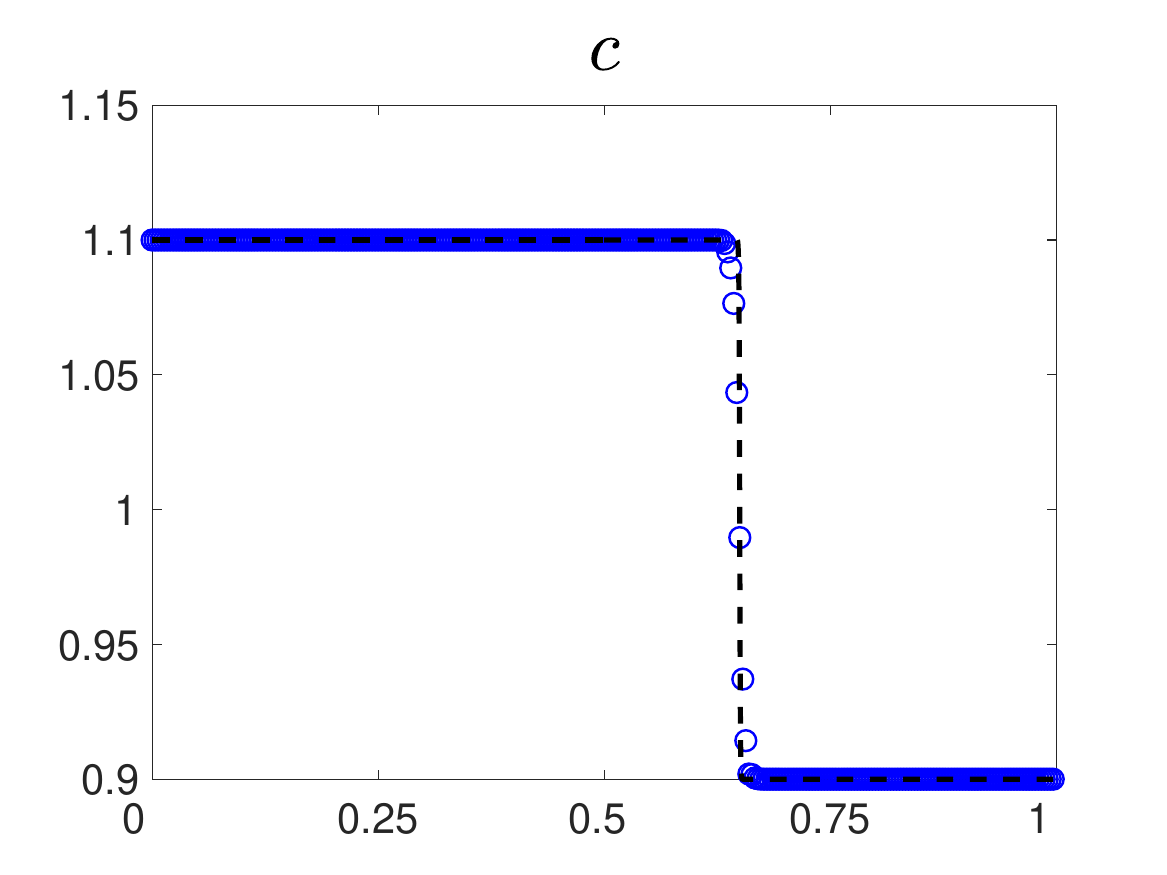}
				}
				\caption{\sf Test T1b, $\gamma = 1$.}
			\end{subfigure}
			\begin{subfigure}[t][][t]{0.46\textwidth}
				\centerline{
					\includegraphics[trim=21 15 41 15,clip,width=0.5\textwidth]{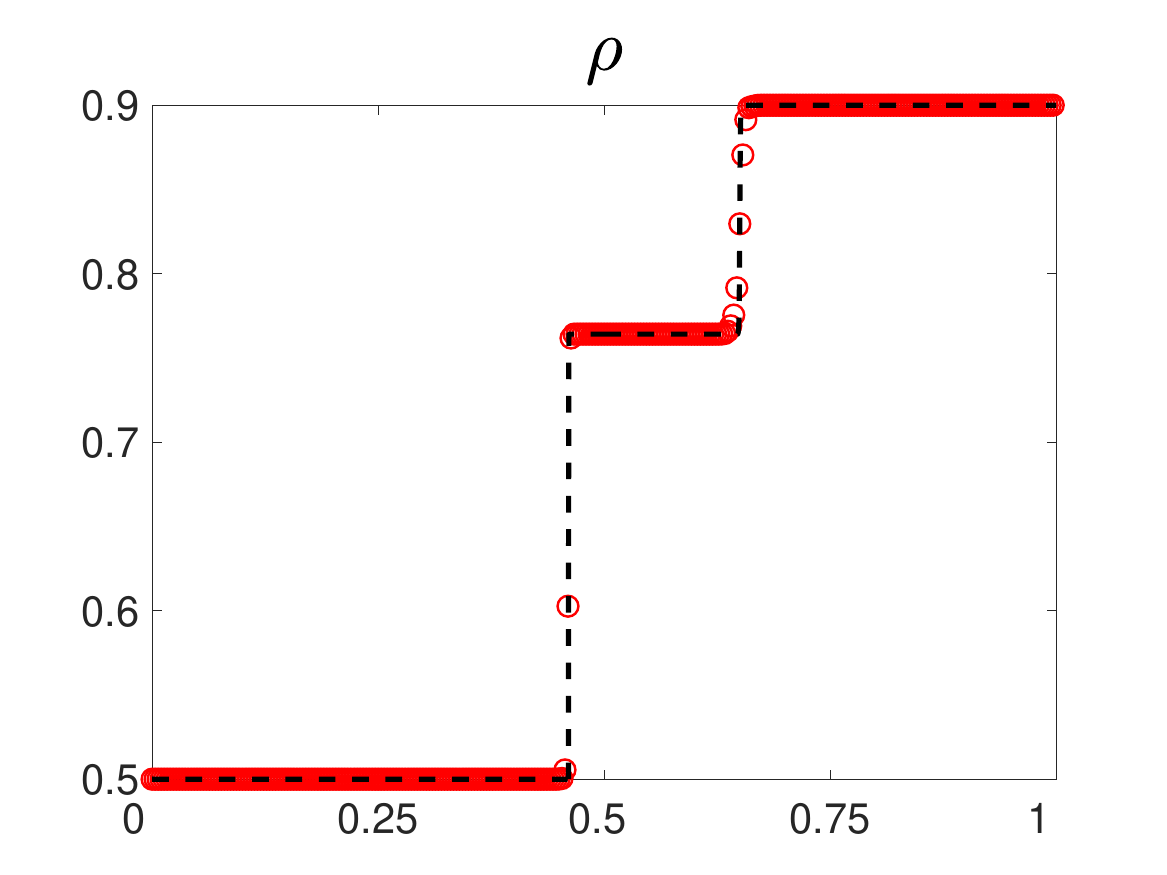}
					\includegraphics[trim=21 15 41 15,clip,width=0.5\textwidth]{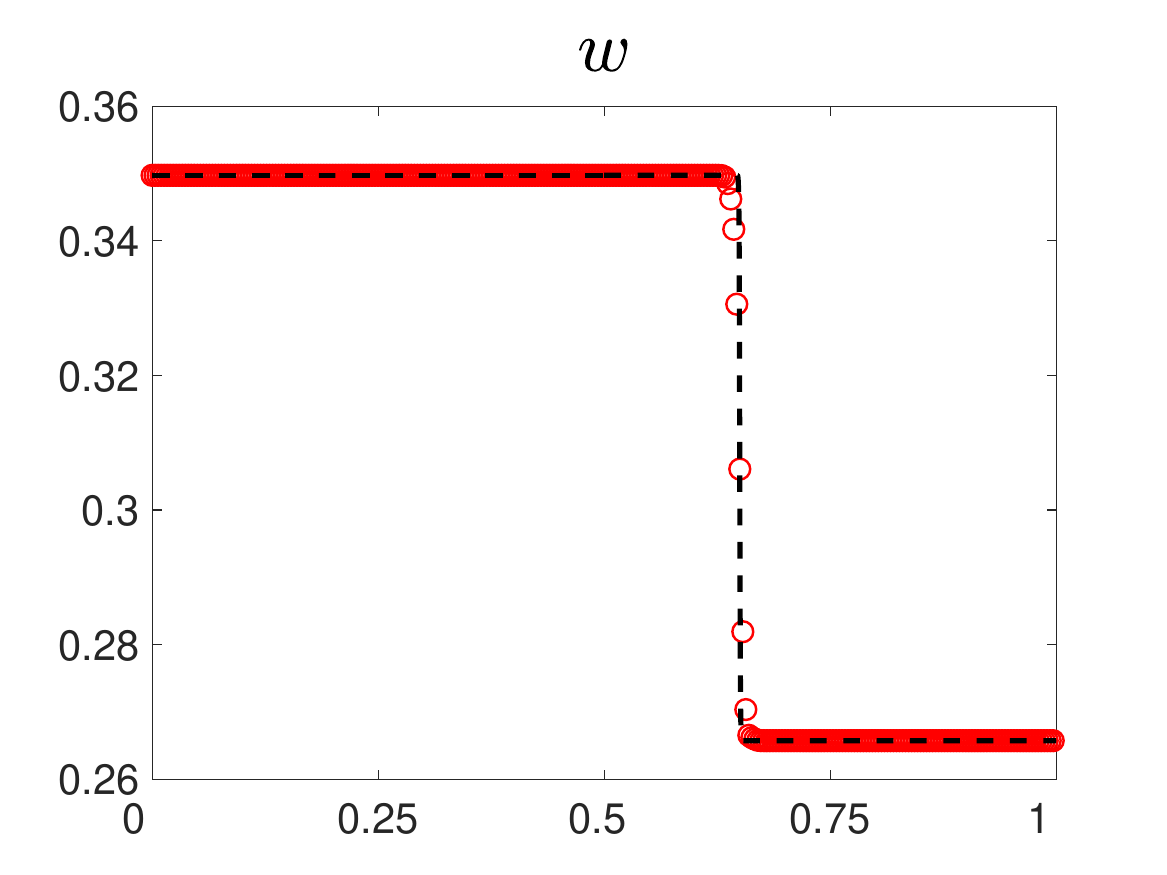}
				}
				\centerline{
					\includegraphics[trim=21 15 41 15,clip,width=0.5\textwidth]{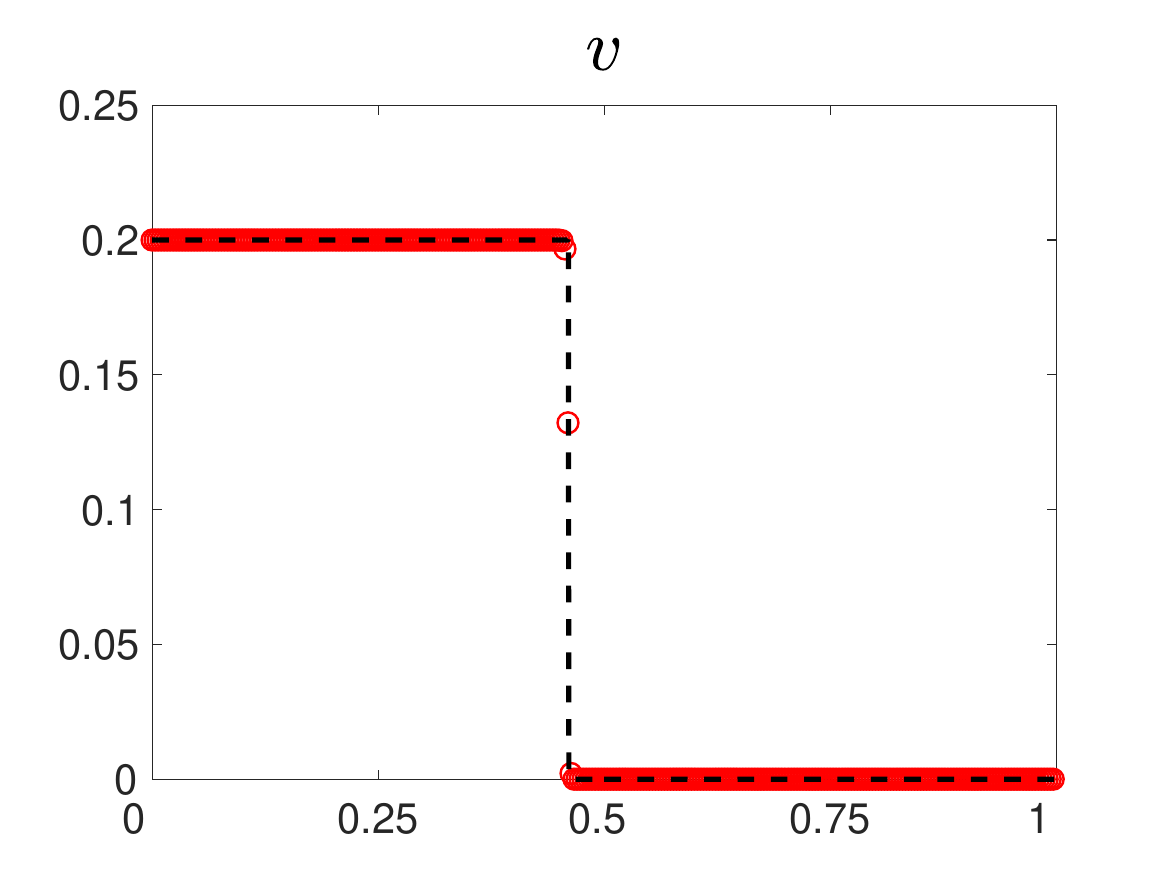}
					\includegraphics[trim=21 15 41 15,clip,width=0.5\textwidth]{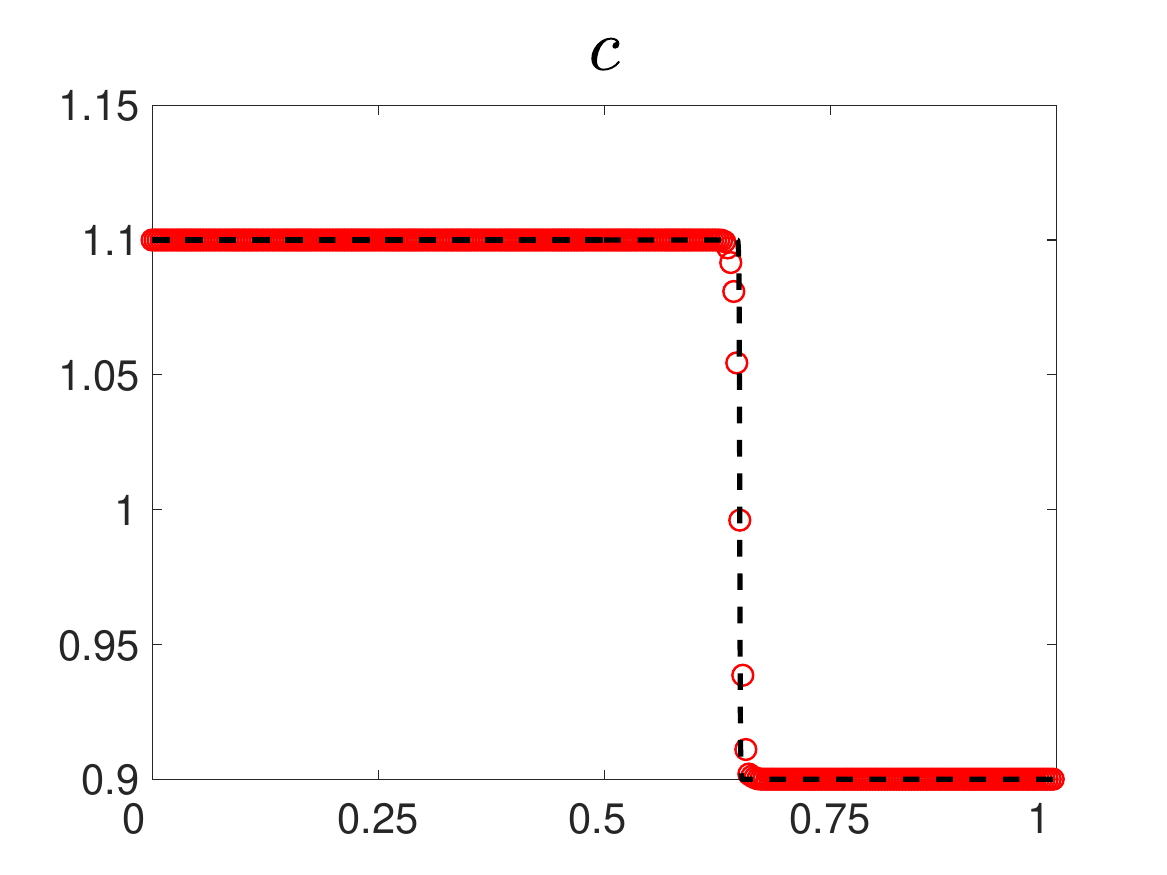}
				}
				\caption{\sf Test T1b, $\gamma = 2$.}
			\end{subfigure}
		}
		\caption{\sf Example \ref{exam2b}, BP-OEDG scheme, $t=0.1$, $\Delta x = 1/300$. Reference solutions are depicted in black dashed lines.}
		\label{fig:exam2b_1}
	\end{figure}
	\begin{figure}
		\centerline{
			\begin{subfigure}[t][][t]{0.23\textwidth}
				\includegraphics[trim=25 15 41 15,clip,width=\textwidth]{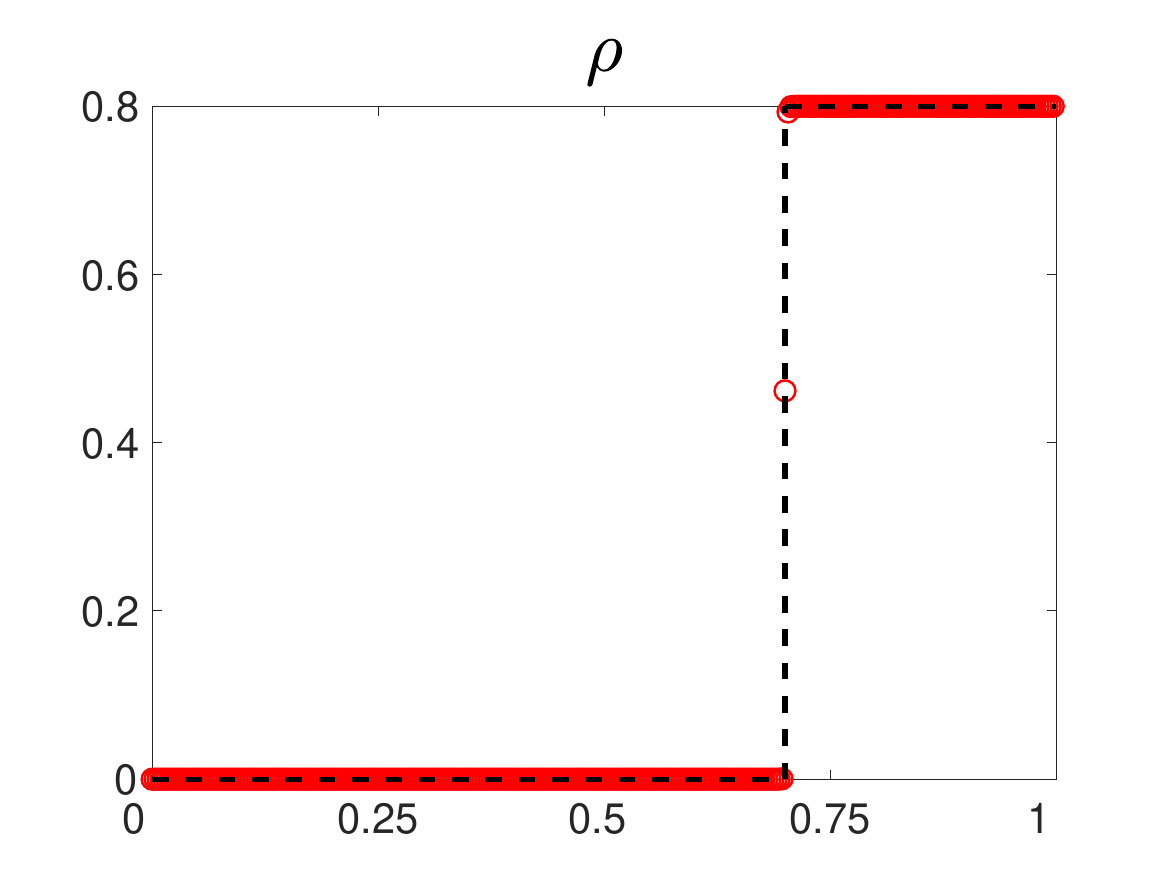}
				\includegraphics[trim=25 15 41 15,clip,width=\textwidth]{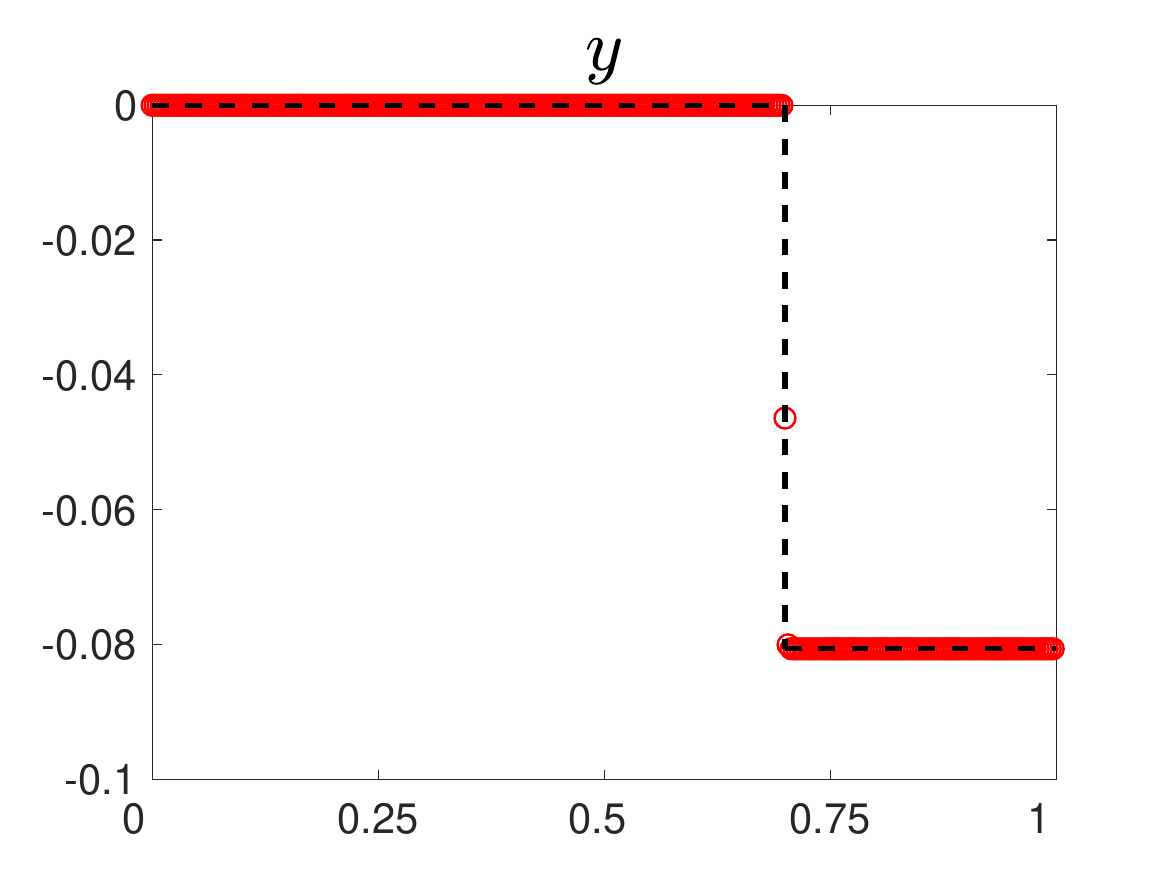}
				\caption{\sf Test T2a, $\gamma = 0$.}
			\end{subfigure}
			\begin{subfigure}[t][][t]{0.23\textwidth}
				\includegraphics[trim=21 15 41 15,clip,width=\textwidth]{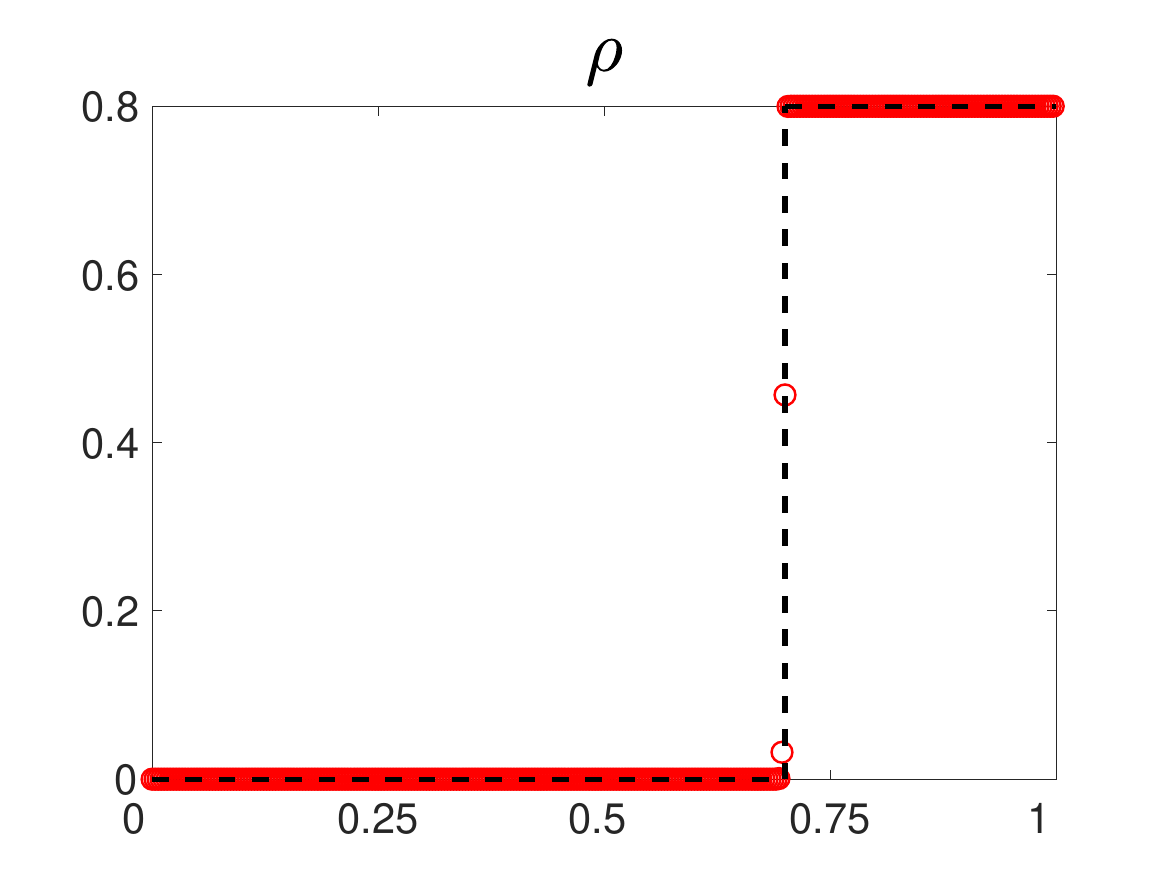} \\
				\includegraphics[trim=21 15 41 15,clip,width=\textwidth]{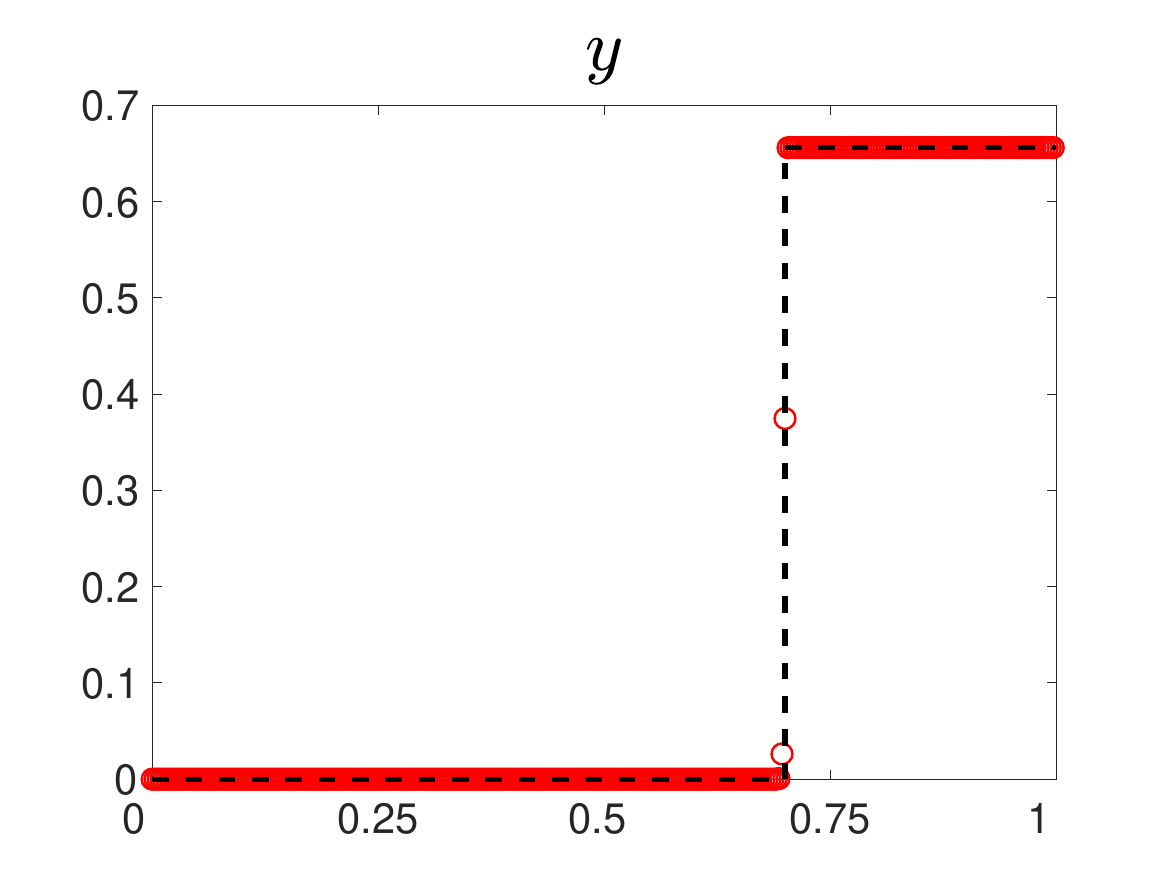}
				\caption{\sf Test T2a, $\gamma = 1$.}
			\end{subfigure}
			\begin{subfigure}[t][][t]{0.23\textwidth}
				\includegraphics[trim=25 15 41 15,clip,width=\textwidth]{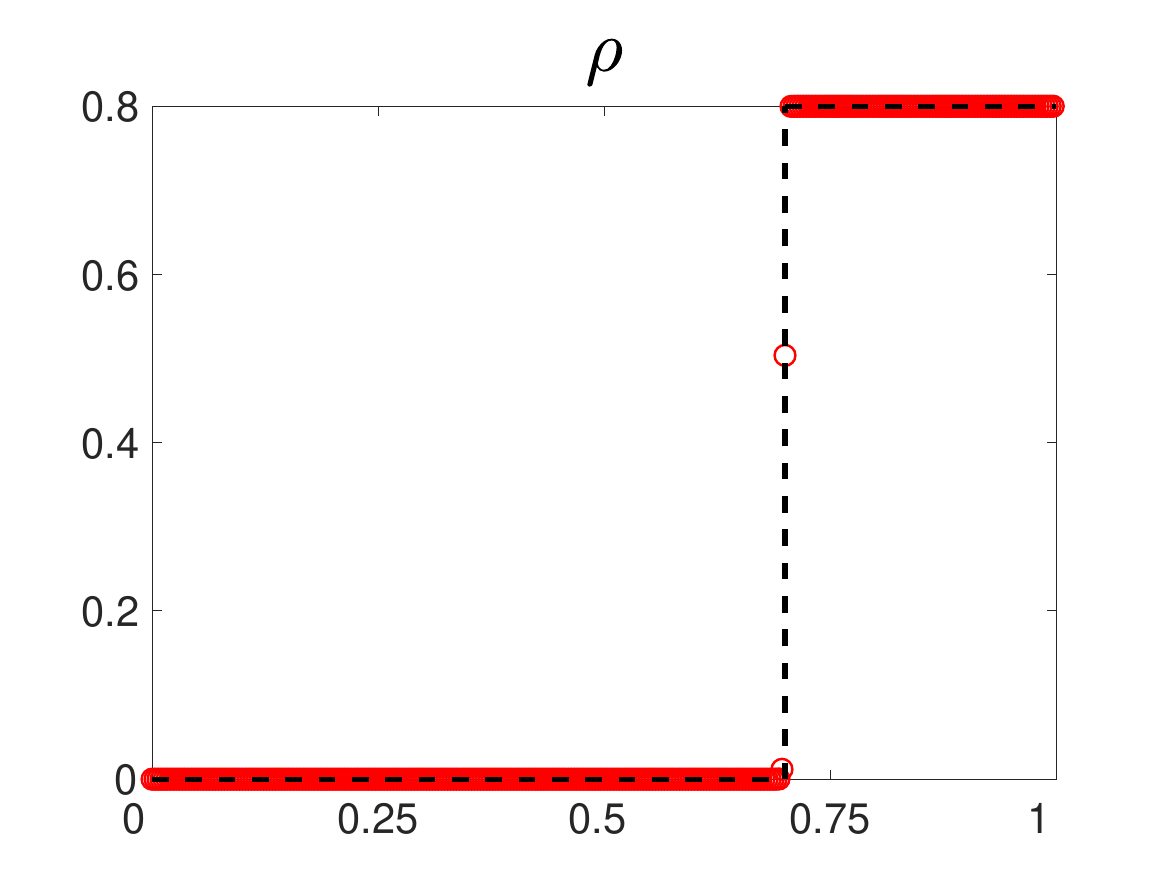}
				\includegraphics[trim=25 15 41 15,clip,width=\textwidth]{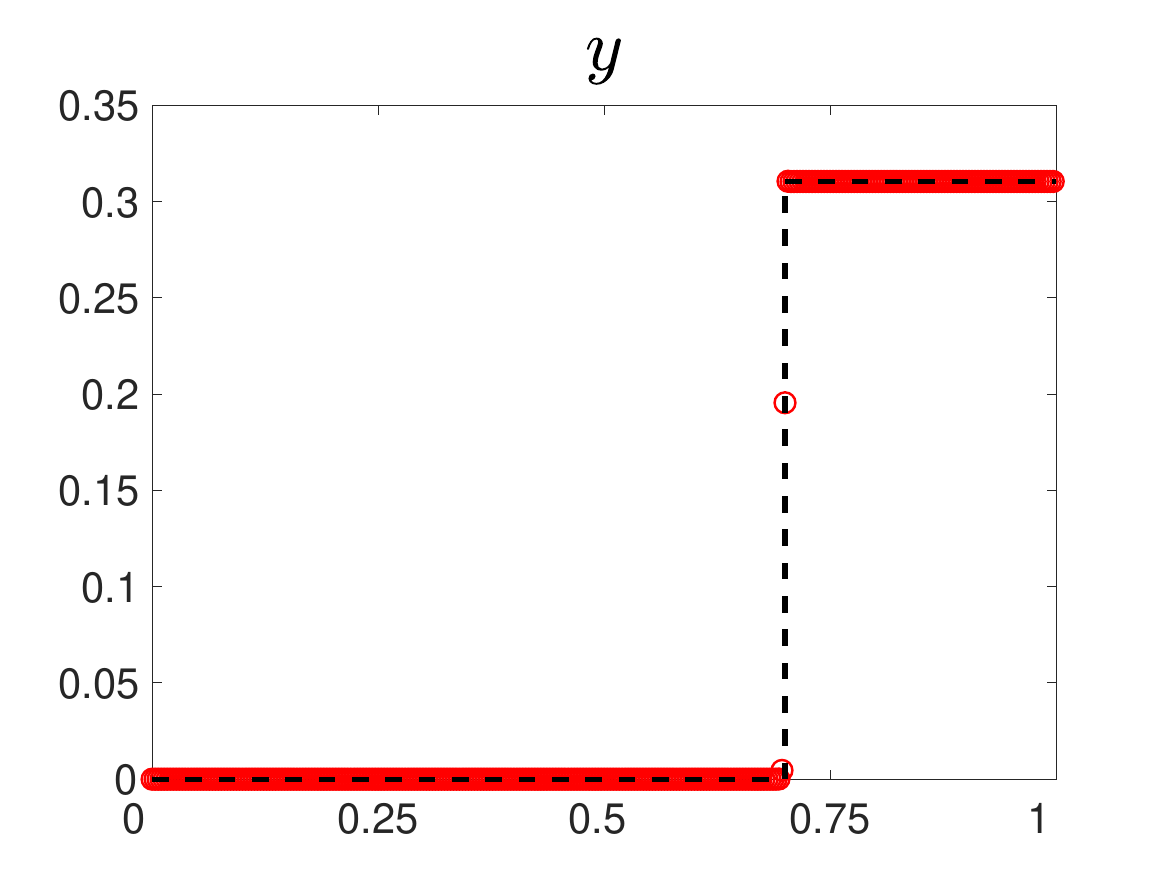}
				\caption{\sf Test T2a, $\gamma = 2$.}
			\end{subfigure}
			\begin{subfigure}[t][][t]{0.23\textwidth}
				\includegraphics[trim=21 15 41 15,clip,width=\textwidth]{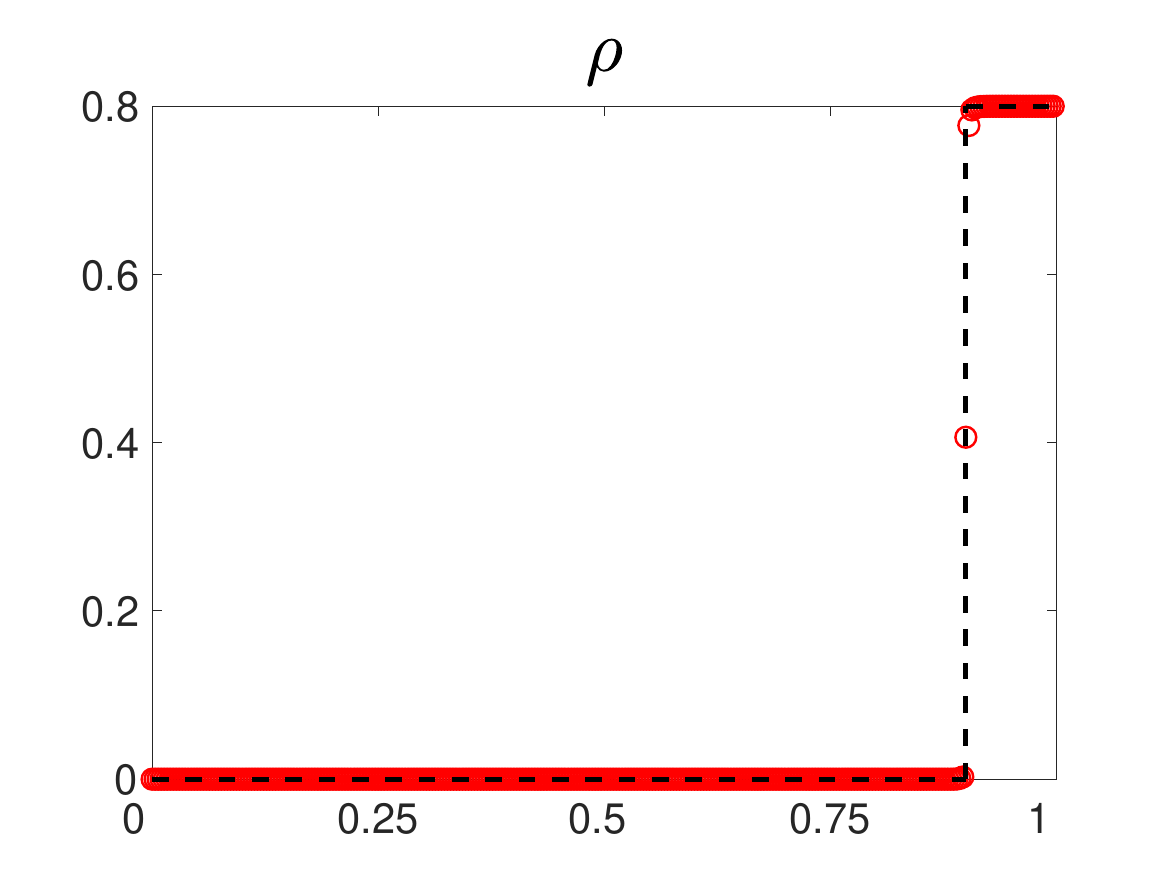} \\
				\includegraphics[trim=21 15 41 15,clip,width=\textwidth]{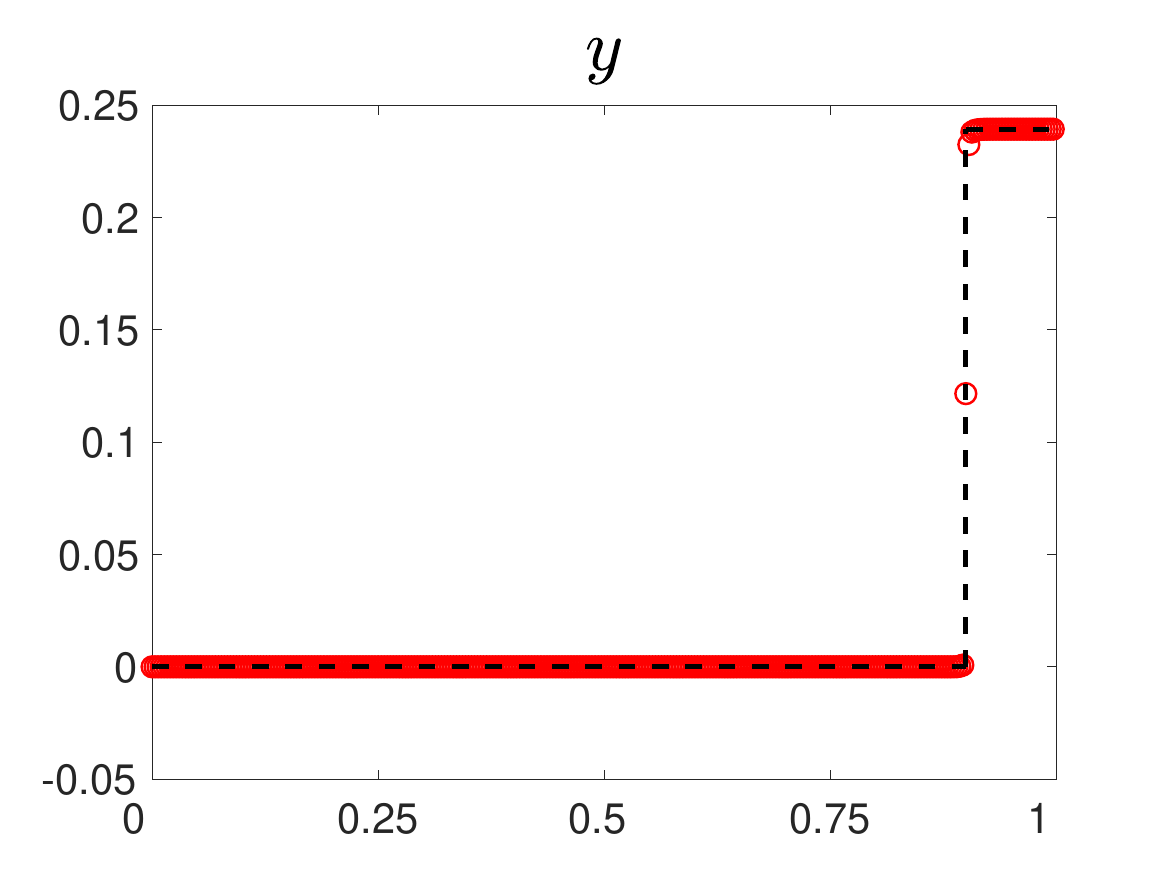}
				\caption{\sf Test T2b, $\gamma = 0$.}
			\end{subfigure}	
		}
		\centerline{
			\begin{subfigure}[t][][t]{0.23\textwidth}
				\includegraphics[trim=25 15 41 15,clip,width=\textwidth]{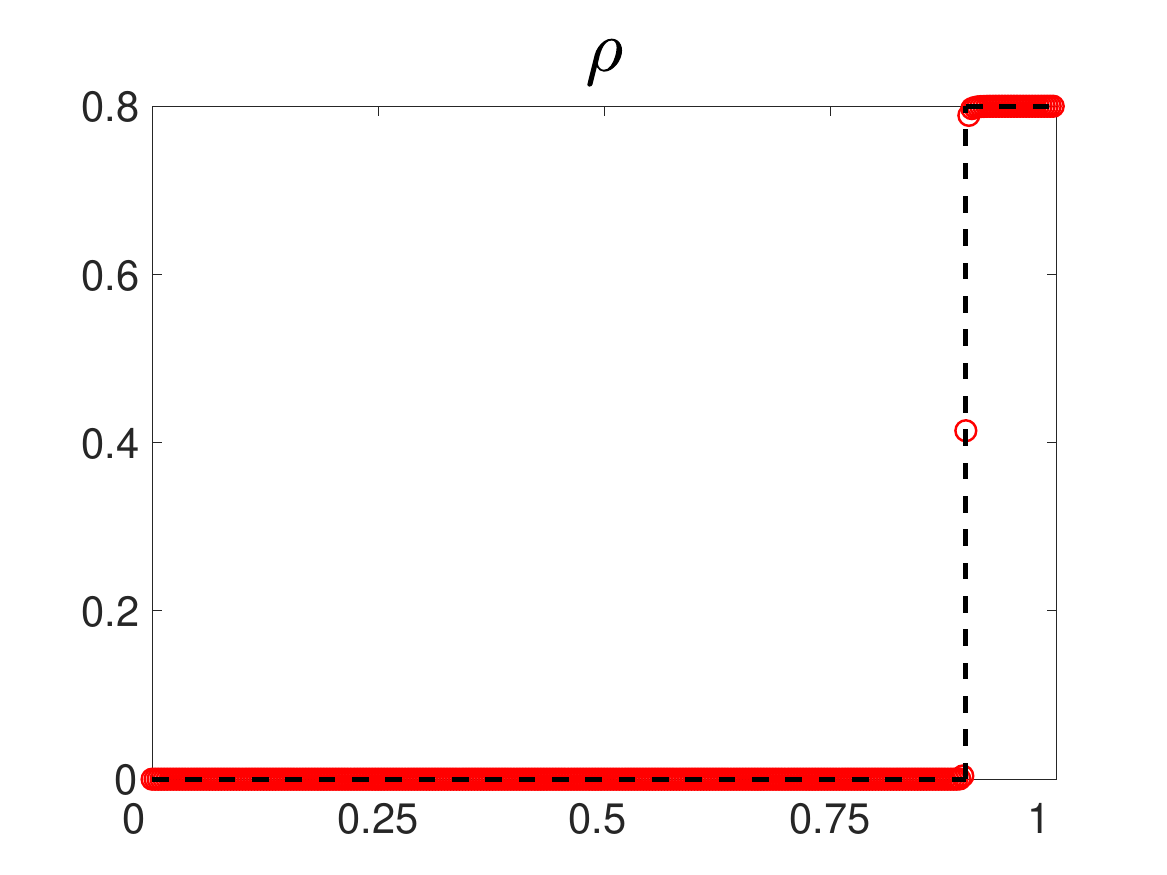}
				\includegraphics[trim=25 15 41 15,clip,width=\textwidth]{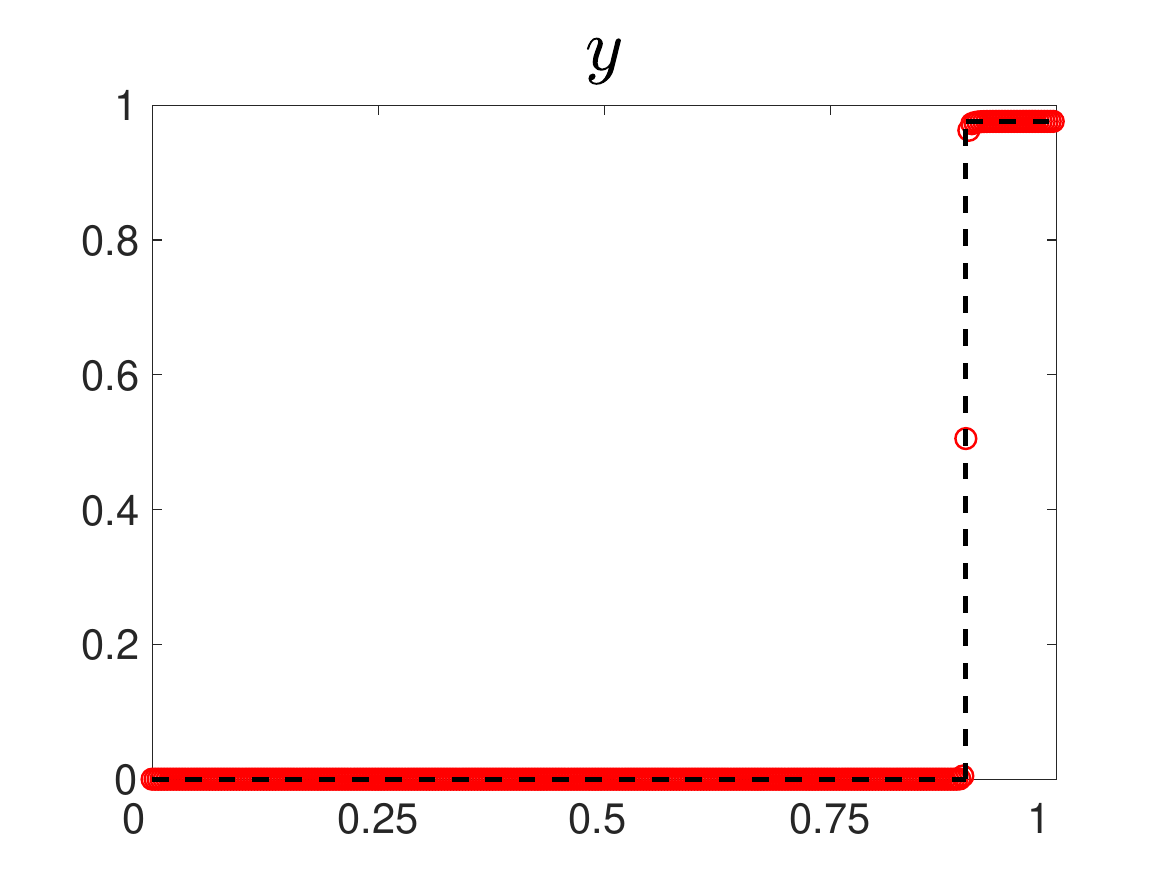}
				\caption{\sf Test T2b, $\gamma = 1$.}
			\end{subfigure}
			\begin{subfigure}[t][][t]{0.23\textwidth}
				\includegraphics[trim=21 15 41 15,clip,width=\textwidth]{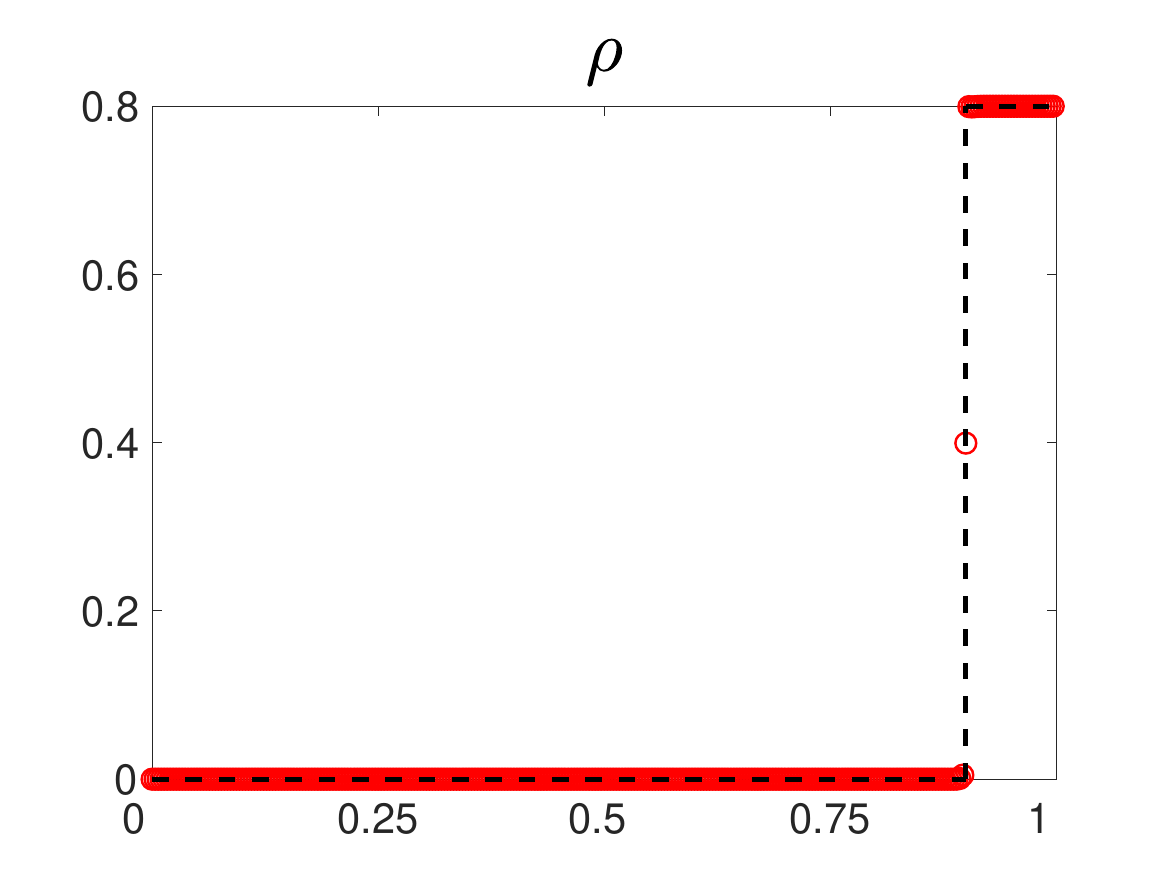} \\
				\includegraphics[trim=21 15 41 15,clip,width=\textwidth]{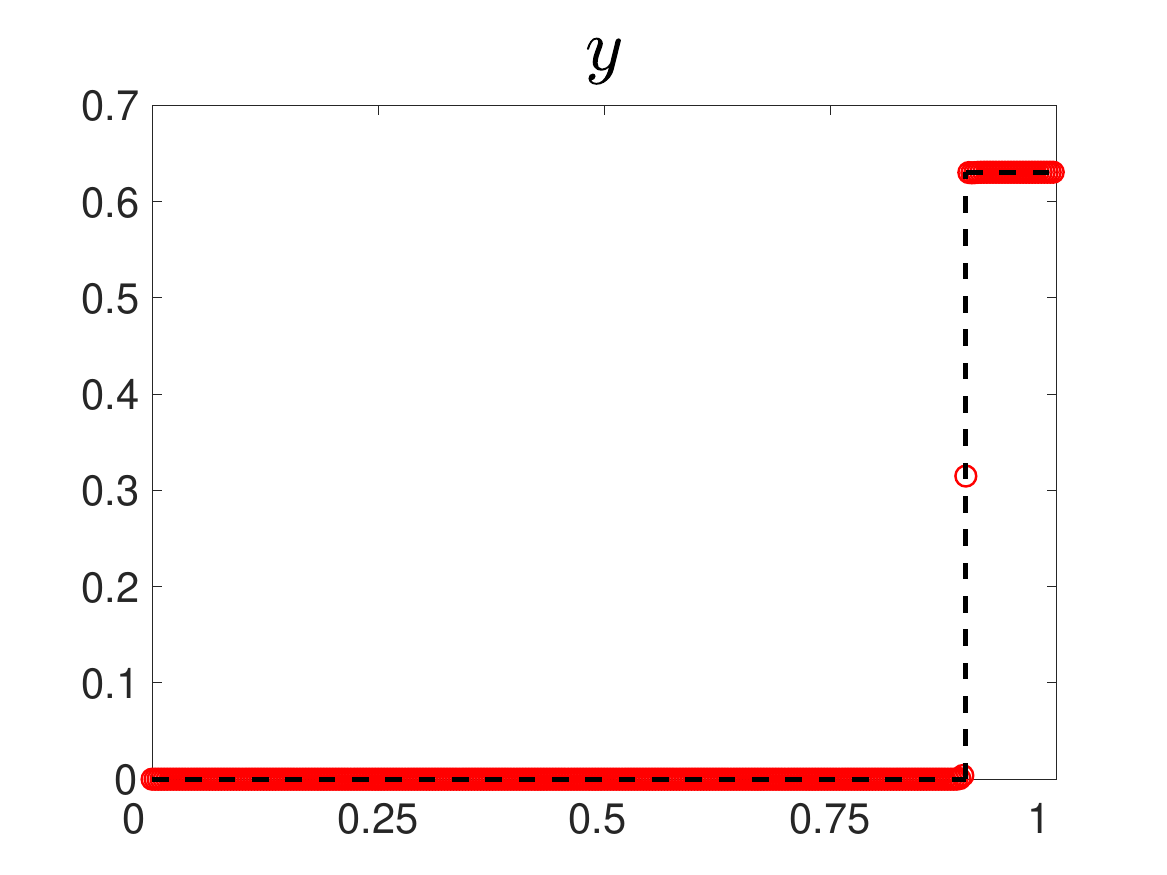}
				\caption{\sf Test T2b, $\gamma = 2$.}
			\end{subfigure}
			\begin{subfigure}[t][][t]{0.23\textwidth}
				\includegraphics[trim=25 15 41 15,clip,width=\textwidth]{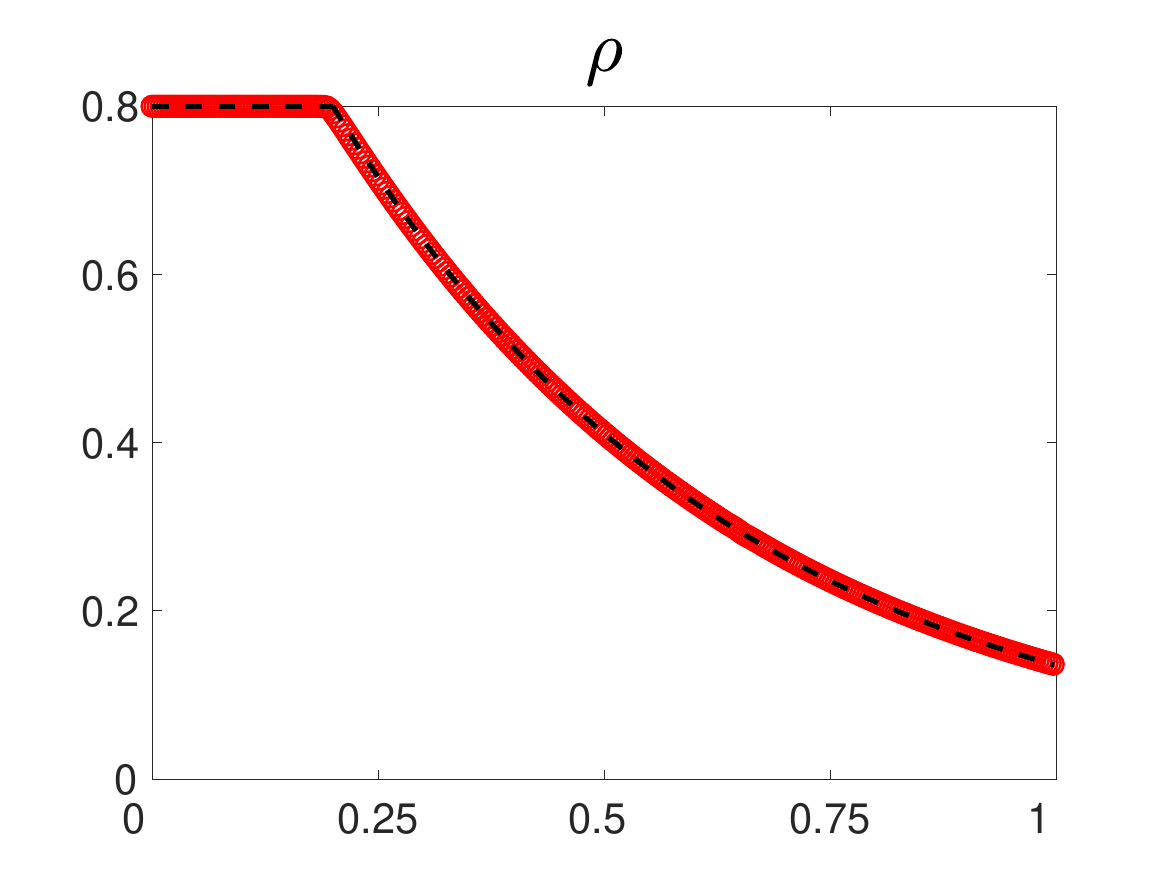}
				\includegraphics[trim=25 15 41 15,clip,width=\textwidth]{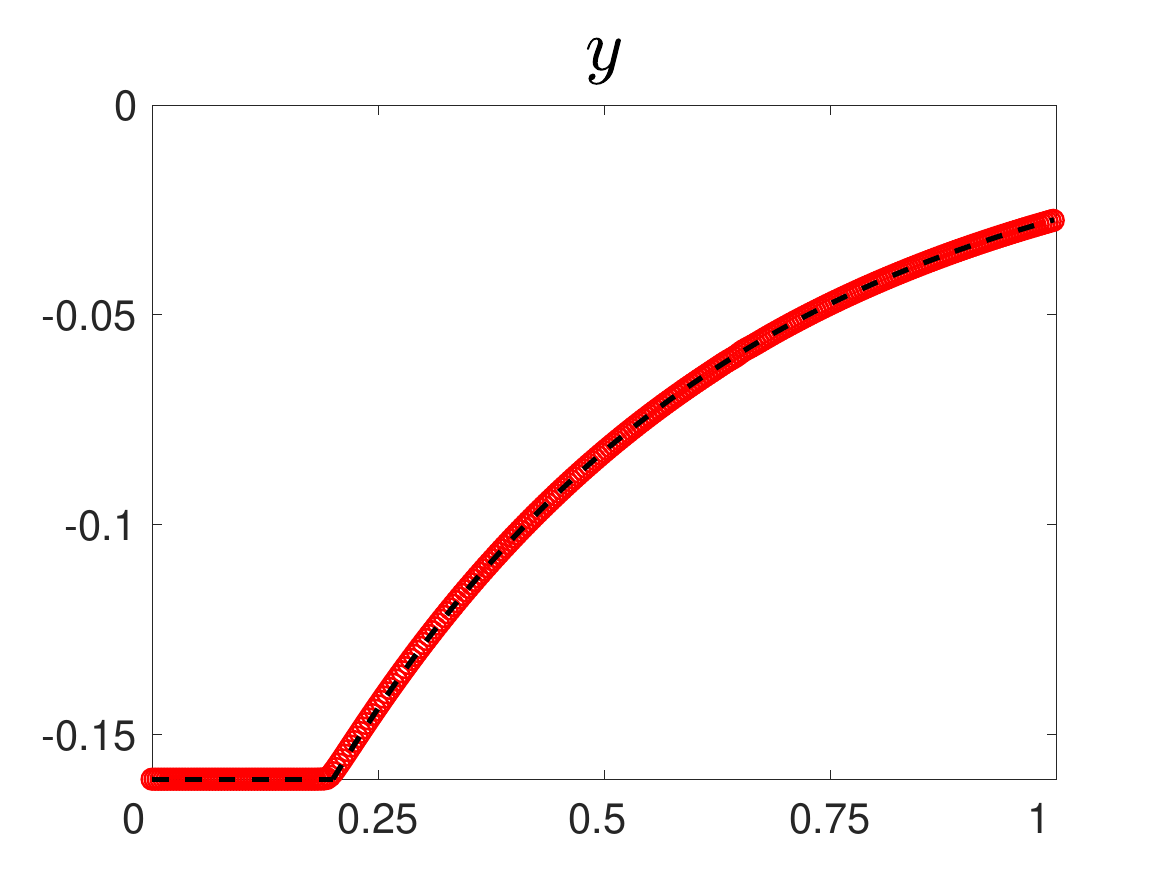}
				\caption{\sf Test T3, $\gamma = 0$.}
			\end{subfigure}
			\begin{subfigure}[t][][t]{0.23\textwidth}
				\includegraphics[trim=21 15 41 15,clip,width=\textwidth]{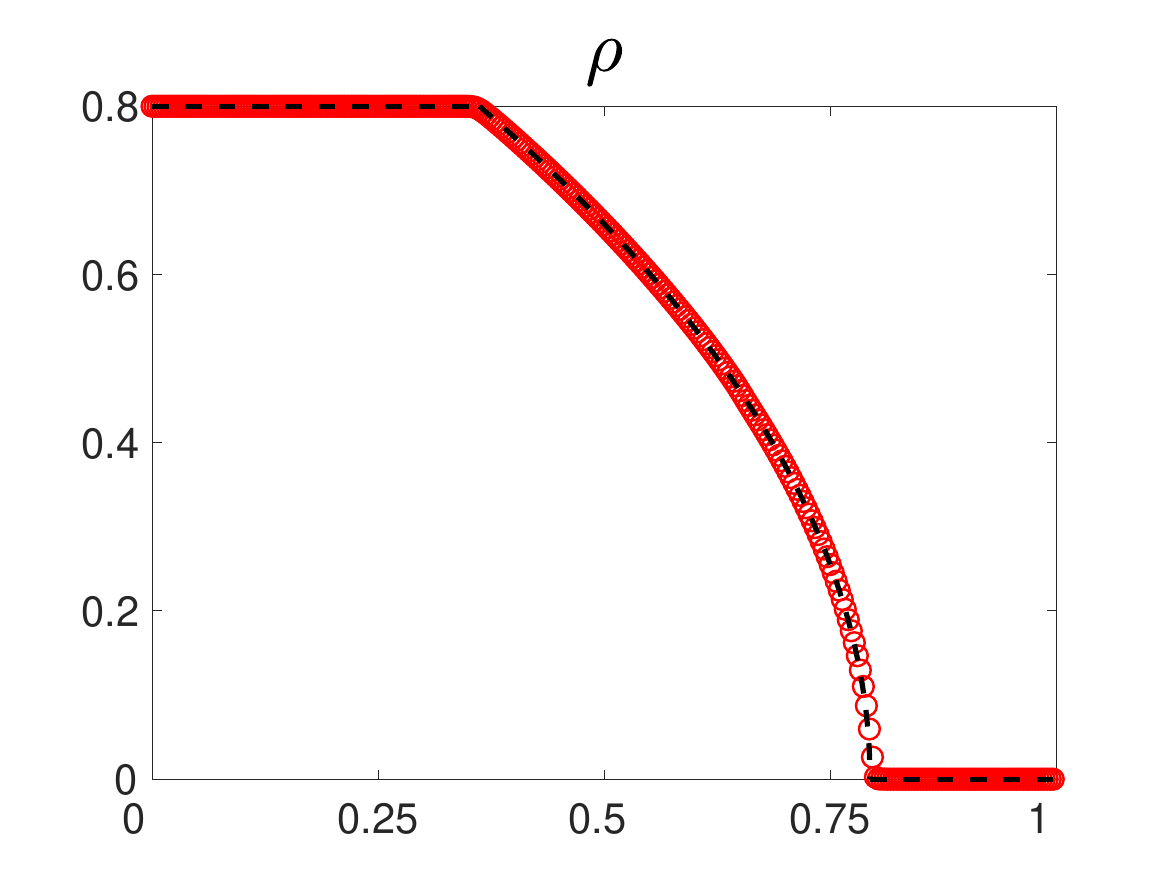} \\
				\includegraphics[trim=21 15 41 15,clip,width=\textwidth]{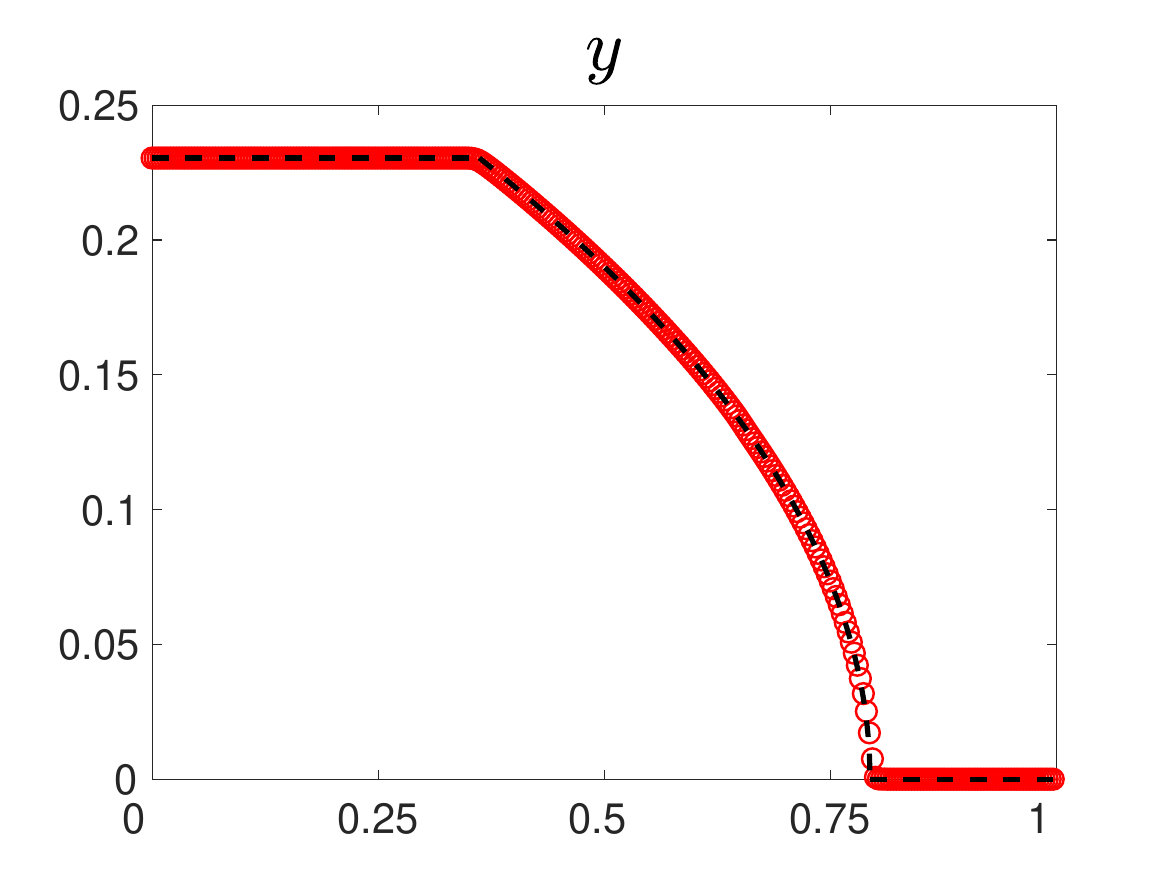}
				\caption{\sf Test T3, $\gamma = 2$.}
			\end{subfigure}
		}
		\caption{\sf Example \ref{exam2b}, BP-OEDG scheme, $t=0.1$, $\Delta x = 1/300$. Reference solutions are depicted in black dashed lines.}
		\label{fig:exam2b_2}
	\end{figure}
	
	For these test problems, the BP property of the employed numerical scheme is crucial for successful simulations. In fact, if the nonBP-OEDG scheme is used instead, the numerical results may breach the invariant domain in just a few time steps. The observed BP violation cases are listed in \Cref{Table3}.

	\begin{table}[h!]
		\renewcommand{\arraystretch}{1.3}
		\centering
		\caption{\sf Example \ref{exam2b}, invariant domain violations detected in the numerical results, nonBP-OEDG scheme, $\Delta x = 1/300$.}
		\begin{tabular}{lll}
			
			\toprule[1.5pt]
			
			test problem & numerically violated constraints & failure time \\
			
			\midrule[1.5pt]
			
			T1a, $\gamma = 0$ & $w \ge w_{\min}$ & $t \approx 6.17\times 10^{-3}$ \\
			T1a, $\gamma = 2$ & $v \ge v_{\min}$ & $t \approx 5.99\times 10^{-3}$ \\
			T1b, $\gamma = 1$ & $c \le c_{\max}$ & $t \approx 1.37\times 10^{-2}$ \\
			T1b, $\gamma = 2$ & $v \ge v_{\min}$ & $t \approx 6.05\times 10^{-3}$ \\
			T2a, $\gamma = 0$ & $\rho > 0$, $w \ge w_{\min}$ & the first time step \\
			T2a, $\gamma = 1$ & $\rho > 0$, $v \ge v_{\min}$, $w \ge w_{\min}$, $c \ge c_{\min}$ & the first time step \\
			T2a, $\gamma = 2$ & $\rho > 0$ & $t \approx 1.00\times 10^{-2}$ \\
			T2b, $\gamma = 0$ & $\rho > 0$ & $t \approx 4.45\times 10^{-5}$ \\
			T2b, $\gamma = 1$ & $w \le w_{\max}$ & $t \approx 3.88\times 10^{-4}$ \\
			T2b, $\gamma = 2$ & $w \le w_{\max}$ & $t \approx 5.33\times 10^{-4}$ \\
			T3, $\gamma = 0$ & $\rho > 0$ & $t \approx 4.22\times 10^{-4}$ \\
			T3, $\gamma = 2$ & $v \ge v_{\min}$ & $t \approx 2.52\times 10^{-3}$ \\
			
			\bottomrule[1.5pt]
			
		\end{tabular}
		\label{Table3}
	\end{table}
	
	As discussed in \Cref{rmk:435}, enforcing the constraint $w \le w_{\max}$ may help mitigate the velocity overshoots.	For Test T2b, if the BP-OEDG scheme is employed without enforcing $w \le w_{\max}$, a non-physical velocity overshoot of significant magnitude will appear in the numerical results; see \Cref{fig:T2b_c}.
	
	\begin{figure}[t!]
		\centerline{
			\begin{subfigure}[t][][t]{0.47\textwidth}
				\includegraphics[trim=25 15 41 15,clip,width=\textwidth]{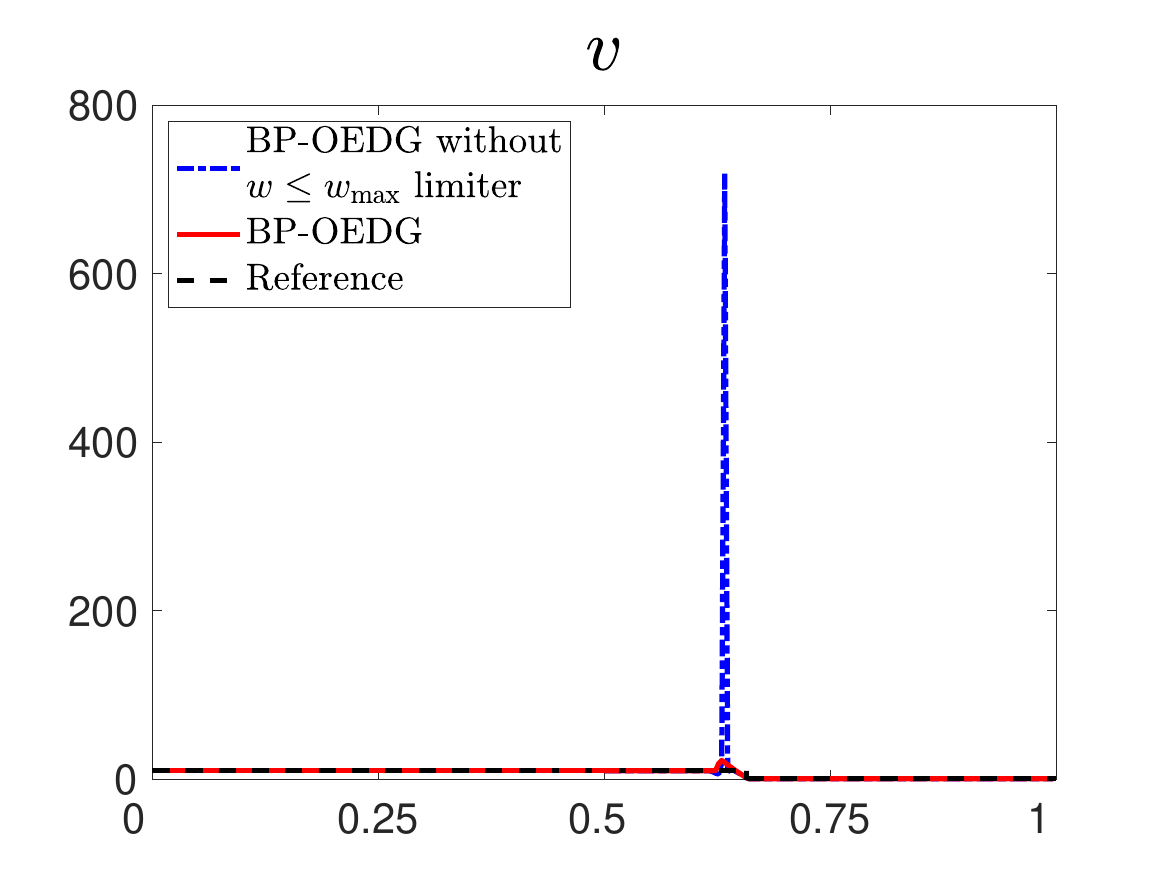} 
				\caption{\sf linear scale}
			\end{subfigure}
			\begin{subfigure}[t][][t]{0.47\textwidth}
				\includegraphics[trim=25 15 41 15,clip,width=\textwidth]{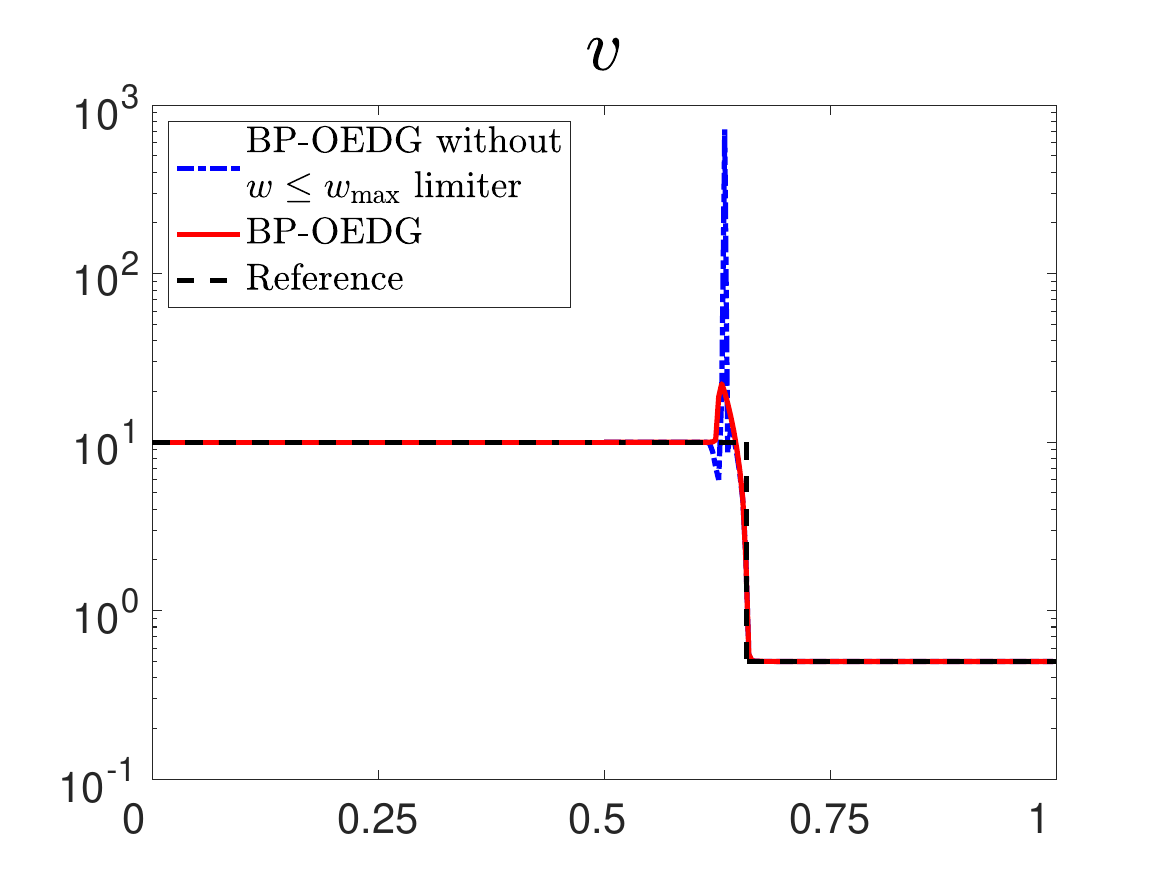}
				\caption{\sf logarithmic scale}
			\end{subfigure}
		}
		\caption{\sf Example \ref{exam2b}, Test T2b, $\gamma = 0$, BP-OEDG scheme compared with BP-OEDG scheme without enforcing the constraint $w \le w_{\max}$, $t =0.0309$, $\Delta x = 1/300$.}
		\label{fig:T2b_c}
	\end{figure}
\end{exa}

\begin{exa} \label{exam_LG}
	In this example, we compare the BP-OEDG scheme with global (\Cref{sec:global}) and local invariant domains (\Cref{sec:local}). Consider the following initial condition:
	\[
	(\rho,v,c)(x,0)
	=
	\begin{dcases}
		(1\times 10^{-12},10,1)
		&
		{\rm if } ~ x < 0.25,
		\\
		(0.8,0.5,1)
		&
		{\rm if } ~ 0.25 \le x < 0.985,
		\\
		(0.5,400,1)
		&
		{\rm if } ~ x \ge 0.985.
	\end{dcases}
	\]
	The parameter $\gamma = 0$. The numerical results obtained by the locally and globally BP-OEDG schemes are shown in \Cref{fig:LG}.
	As one can see, even with a global constraint $w \le w_{\max}$ enforced, a non-physical velocity overshoot can still be observed in the numerical results obtained by the globally BP-OEDG scheme. 
	This is due to the global $w$ upper bound ($\approx$ 399.37) being much larger than the local $w$ upper bound  ($\approx$ 0.30) near $x = 0.25$. In this case, enforcing a global constraint $w \le w_{\max}$ can no longer help mitigate the $v$-overshoot in the vicinity.
	
	\begin{figure}[th!]
		
		\centerline{
			
			\hfill
			
			\begin{subfigure}[t][][t]{0.47\textwidth}
				\includegraphics[trim=25 15 41 15,clip,width=\textwidth]{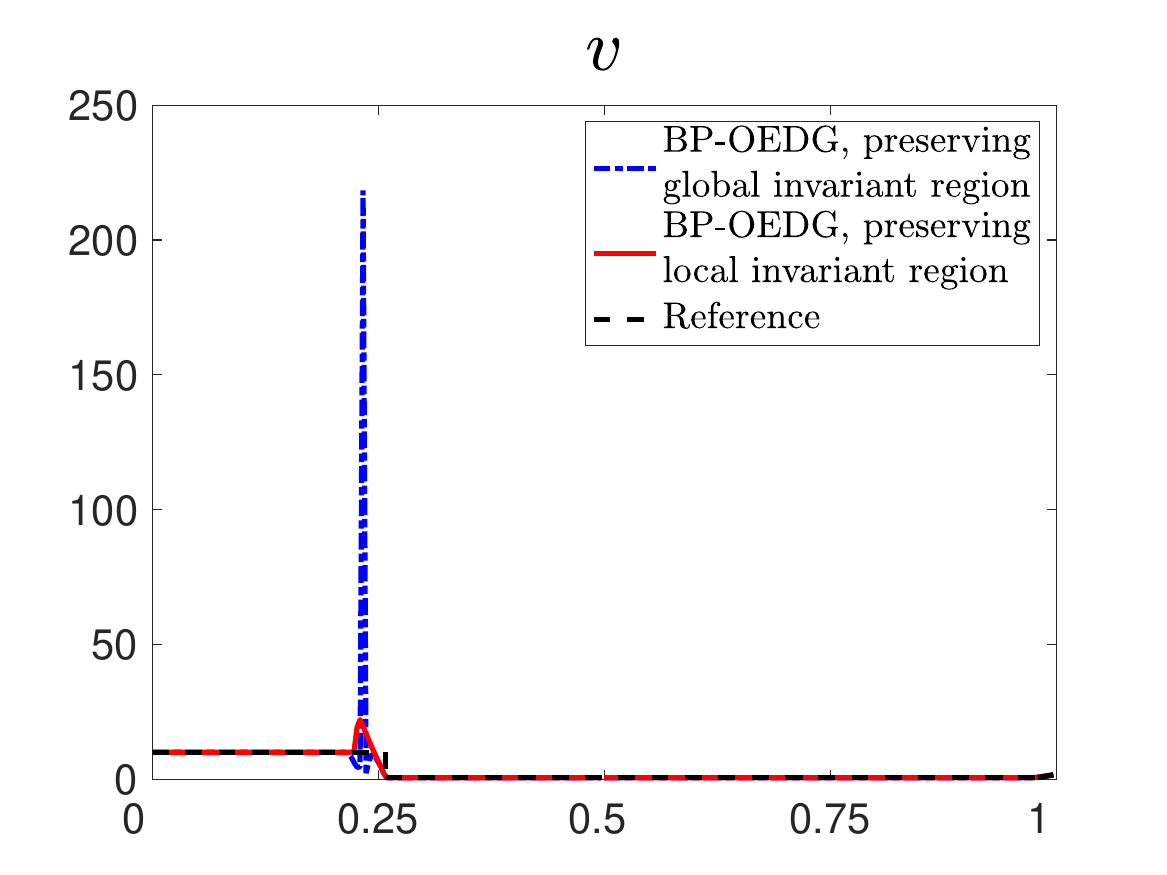} 
				\caption{\sf linear scale}
			\end{subfigure}
			
			\hfill
			
			\begin{subfigure}[t][][t]{0.47\textwidth}
				\includegraphics[trim=25 15 41 15,clip,width=\textwidth]{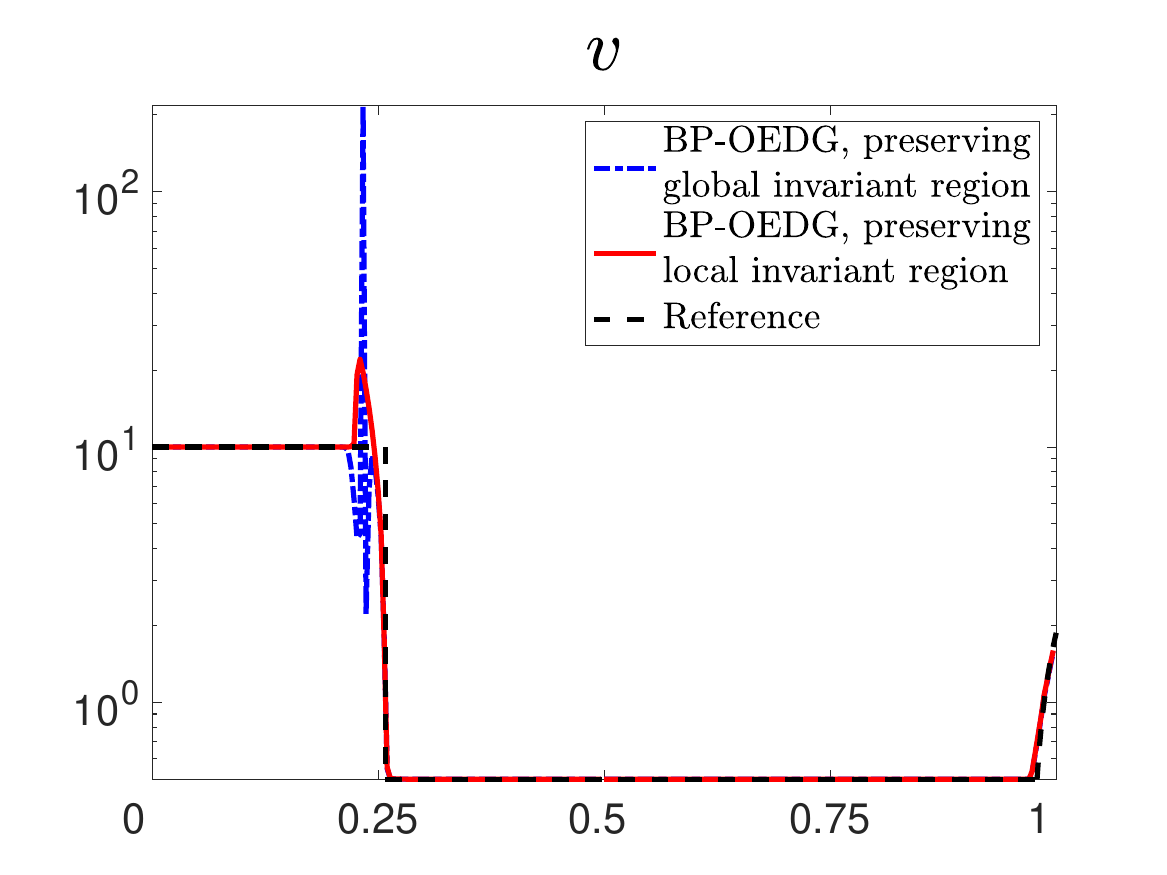}
				\caption{\sf logarithmic scale}
			\end{subfigure}
			
			\hfill
			
		}
		
		\caption{\sf Example \ref{exam_LG}, BP-OEDG scheme with global and local invariant domains, $t =0.0305$, $\Delta x = 1/300$.}
		\label{fig:LG}
	\end{figure}
	
\end{exa}

\subsection{Numerical experiments on road networks}\label{network}

This section studies the applications of the proposed BP-OEDG methods on road networks. 
Consider a road network consists of $N_R$ road segments and $N_J$ junctions, we partition the spatial domain on the $r$-th road by $N_x^{(r)}$ uniform cells $\big\{I_j^{(r)} := [x^{(r)}_{j - \frac{1}{2}}, x^{(r)}_{j + \frac{1}{2}}] \big \}_{j = 1}^{N_x^{(r)}}$. 
Here, $x^{(r)}_\frac12$ and $x^{(r)}_{N_x^{(r)}+\frac12}$ respectively label the locations of entry and exit of Road $r$.
The numerical solution on the $r$-th road segment is denoted by $\BU^{(r)}_h \in \mathbb{V}^k_h$. Let $\delta_{\ell}^-$ (resp. $\delta_{\ell}^+$) denote the set of indices representing incoming (resp. outgoing) roads connected to the Junction $\ell$. 
At road junctions, proper coupling conditions are needed to allocate the traffic fluxes from incoming roads ($i \in \delta^-_\ell$) to outgoing roads ($j \in \delta^+_\ell$).
Given the traffic states in the vicinity of the junction, namely,
\[
\textrm{the input of coupling condition:}
\quad
\left \{ \BU^{(i)}_h \left(x_{N_x^{(i)} + \frac{1}{2}}^{(i), -} \right) \right \}_{i \in \delta_\ell^-} {\rm and} \quad
\bigg\{ \BU^{(j)}_h \bigg(x_{\frac{1}{2}}^{(j), +} \bigg) \bigg\}_{j \in \delta_\ell^+},
\]
a coupling condition should be employed to determine the corresponding traffic states immediately outside the computational domain, that is,
\[
\textrm{the output of coupling condition:}
\quad
\left \{ \BU^{(i)}_h \left(x_{N_x^{(i)} + \frac{1}{2}}^{(i), +} \right) \right \}_{i \in \delta_\ell^-} {\rm and} \quad
\bigg\{ \BU^{(j)}_h \bigg(x_{\frac{1}{2}}^{(j), -} \bigg) \bigg\}_{j \in \delta_\ell^+},
\]
which are required to evolve the numerical solutions $\{\BU^{(r)}_h\}_{r=1}^{N_r}$ to the next time step, and are also used in updating global invariant domains (see \eqref{eq:1807}) and local invariant domains (see \eqref{eq:1917}--\eqref{eq:1982}). Various options for coupling condition are available in the literature, for example, \cite{garavello2006traffic,herty2006coupling,herty2006optimization,haut2007second,gottlich2021second}. 
In order to validate the robustness of the proposed BP-OEDG schemes on road network applications, we consider two different settings in the following numerical experiments: 
\begin{itemize}[leftmargin=9mm]
	\item{\bf HB network}: The original ARZ model with the coupling condition from \cite{haut2007second} on networks with diverging and/or merging junctions.
	
	\item{\bf GHMW network}: The AP ARZ model with the coupling condition from \cite{gottlich2021second} on networks with diverging, merging, and/or hybrid junctions.
\end{itemize}

For $j \in \delta_\ell^+$ and $i \in \delta_\ell^-$, $q_{ji} \geq 0$ denotes the traffic flux from Road $i$ to Road $j$, and we denote by $q_i^{-}$ (resp. $q_j^+$) the total incoming (resp. outgoing) flux of Road $i$ (resp. $j$), namely, $q_i^- = \sum_{j \in \delta_\ell^+} q_{ji}$, $q_j^+ = \sum_{i \in \delta_\ell^-} q_{ji}$.
For a junction with multiple incoming and outgoing roads, one should specify a traffic distribution matrix $\mathcal{A} = (a_{ji})_{i \in \delta_\ell^-, j \in \delta_\ell^+}$, where $0 \leq a_{ji} \leq 1$, and $q_{ji} = a_{ji} q_i^-$. Note that $a_{ji}$ denotes the fraction of vehicles on Road $i$ going to Road $j$, and $\sum_{j \in \delta_\ell^+} a_{ji} = 1$.
For the GHMW network, the fluxes $\mathbf{q}_j = (q_{ji})_{i \in \delta_\ell^-}$ of an outgoing Road $j$, are proportional to a given vector $\boldsymbol{\beta} = (\beta_i)_{i \in \delta_\ell^-}$.
For the HB network, the vector $\boldsymbol{\beta}$ is determined by demand from Road $i$.
For simplicity, in the following examples, we assume the lengths of all road segments to be 1.

\begin{exa} \label{exam3a} 
	In this example, we use the AP ARZ model ($\gamma = 1$) to simulate traffic flows near a diverging junction with one incoming road (Road 1) and two outgoing roads (Roads 2 and 3). The initial conditions on these roads are given by
	\begin{equation*}
		\begin{gathered}
			\rho^{(1)}(x,0) = \rho^{(3)}(x,0) = 0.1, \quad 
			\rho^{(2)}(x,0) = 
			\left\{ \begin{aligned}
				& 0.2, \; &&{\rm if} \; x \in [0.2, 0.4] \cup [0.6, 0.8],\\
				& 0.1, \; &&{\rm otherwise},
			\end{aligned}\right. 
			\\
			w^{(r)}(x,0) = c^{(r)}(x,0) = 1, \; r = 1, 2, 3.
		\end{gathered}
	\end{equation*}
	The traffic flow entering the junction from incoming Road 1 is equally distributed between outgoing Roads 2 and 3.
	We implement the proposed BP-OEDG scheme to solve this problem, and the obtained numerical results, presented in \Cref{fig:T4}, align consistently with the results presented in \cite{buli2020discontinuous} and \cite{canic2015runge}.
	
	\begin{figure}[h!]
		\centering
		\centerline{
			\includegraphics[trim=25 55 41 75,clip,width = 0.5\textwidth]{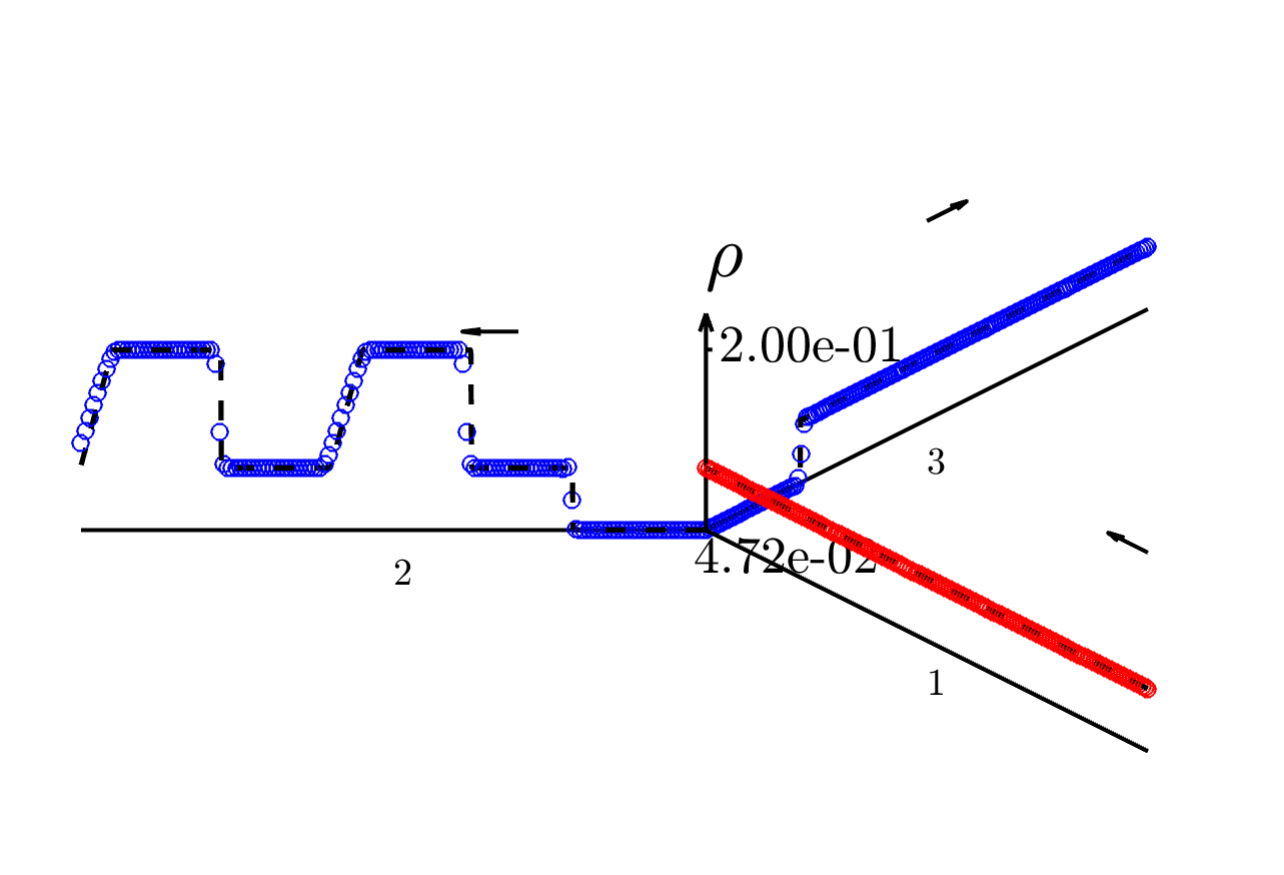}
			\includegraphics[trim=25 55 41 75,clip,width = 0.5\textwidth]{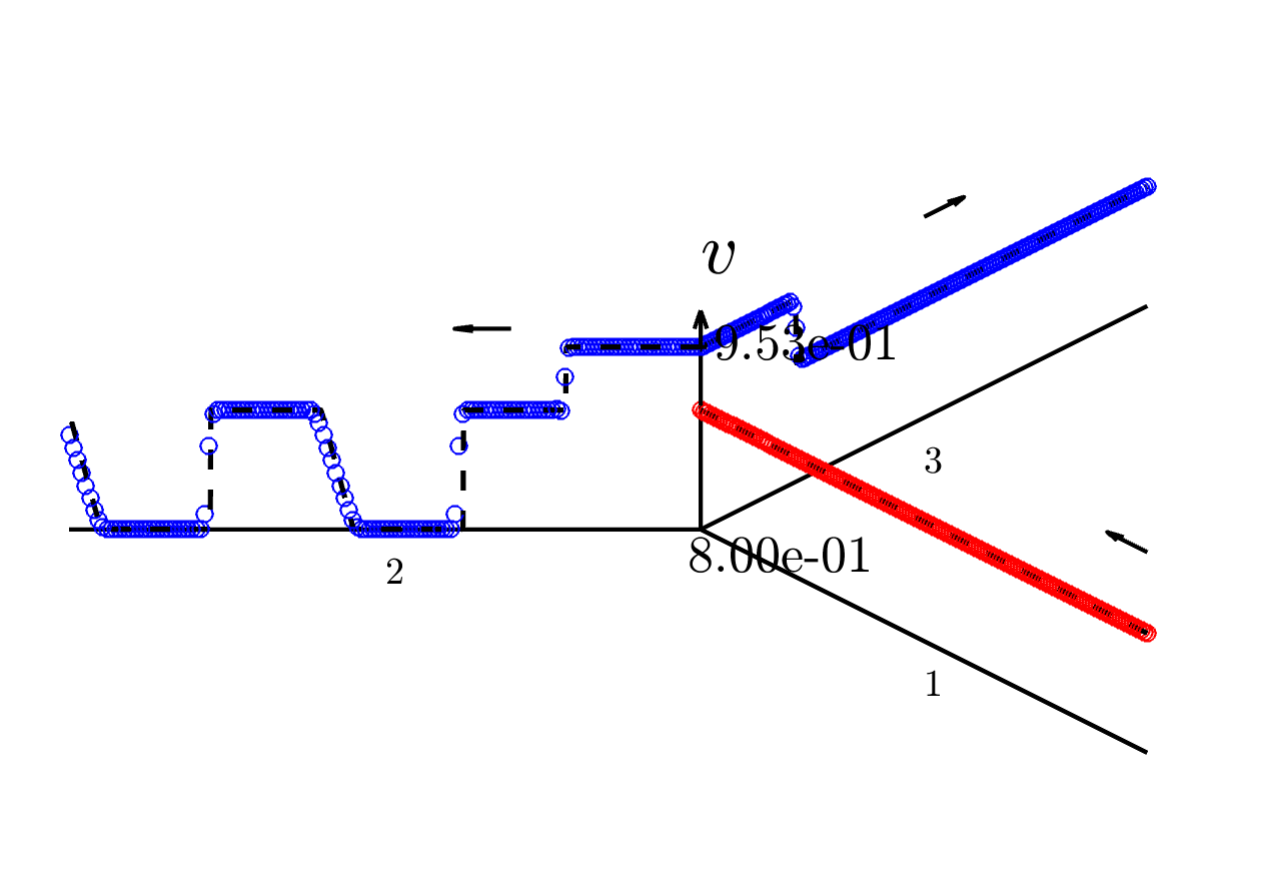}
		}
		\caption{\sf Example \ref{exam3a}, BP-OEDG scheme,
			$t = 0.25$, $\Delta x = 1/300$. Reference solutions are depicted in black dashed lines.}
		\label{fig:T4}
	\end{figure}	
\end{exa}

\begin{exa} \label{exam3b}
	We consider an HB network consisting of one diverging junction, one incoming road (Road 1) and three outgoing roads (Roads 2--4). The initial conditions on these roads are given by
	\begin{gather*}
		\rho^{(1)}(x,0) = \rho^{(3)}(x,0) = 0.1, 
		\quad
		\rho^{(2)}(x,0) = 
		\begin{dcases}
			10^{-10} 
			&
			{\rm if} \; x \in [0.2, 0.4] \cup [0.6, 0.8],
			\\
			0.1
			&
			{\rm otherwise},
		\end{dcases} \\
		\rho^{(4)}(x,0) = 0.05 \big[1 + \sin{(5 \pi x)}\big] + 10^{-10}, 
		\quad
		w^{(r)}(x,0) = 0.5,\; r = 1, 2, 3, 4.
	\end{gather*}
	The traffic flow entering the junction from incoming Road 1 is equally distributed among outgoing Roads 2--4.
	The BP-OEDG scheme is used to solve this problem with $\gamma = 0$ and $\gamma = 2$, respectively. The obtained numerical results are presented in \Cref{fig:T5}. As one can see, the wave structures are sharply resolved, and they match well with the results reported in \cite{buli2020discontinuous,canic2015runge}.
	
	\begin{figure}[th!]
		\centering
		\centerline{
			\includegraphics[trim=75 40 61 50,clip,width = 0.5\textwidth]{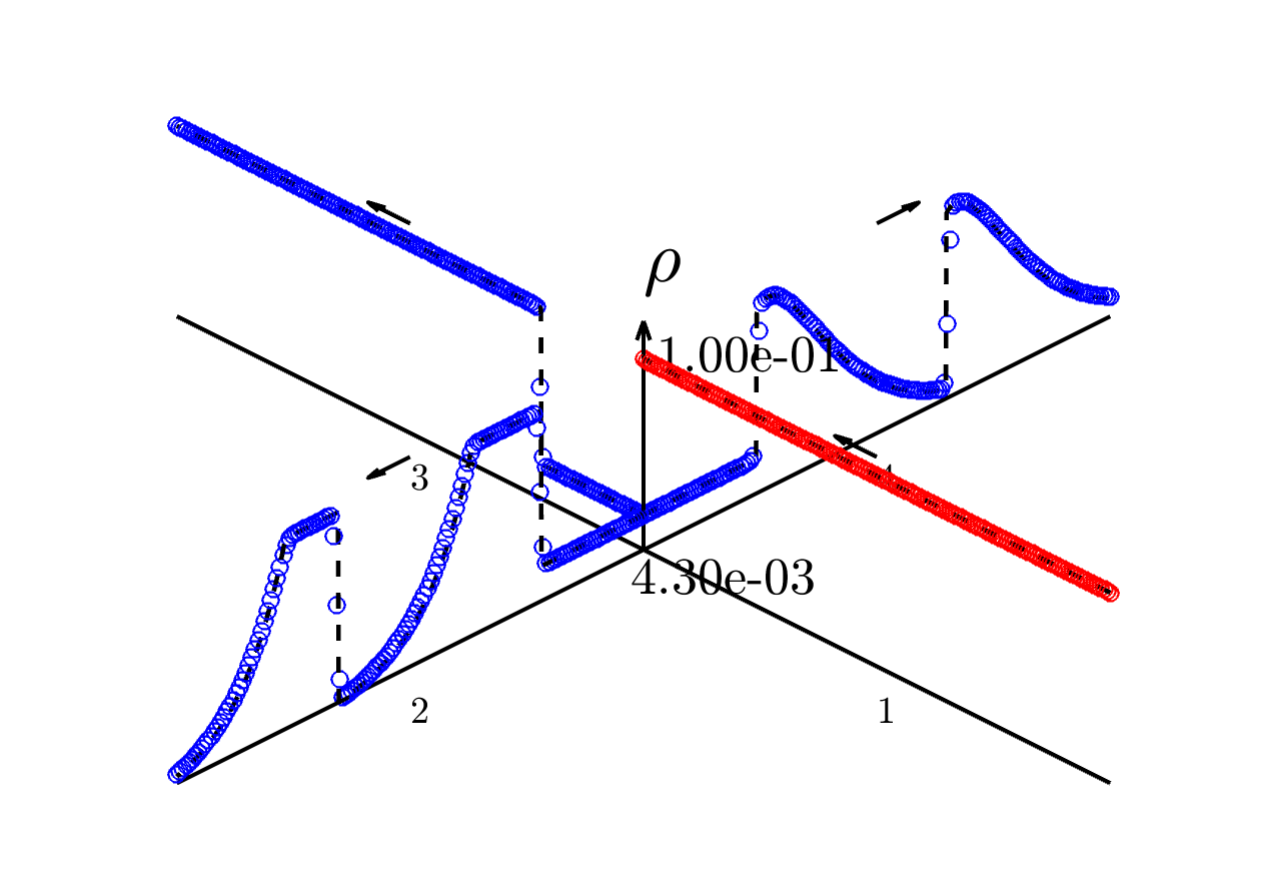}
			\includegraphics[trim=75 40 51 50,clip,width = 0.5\textwidth]{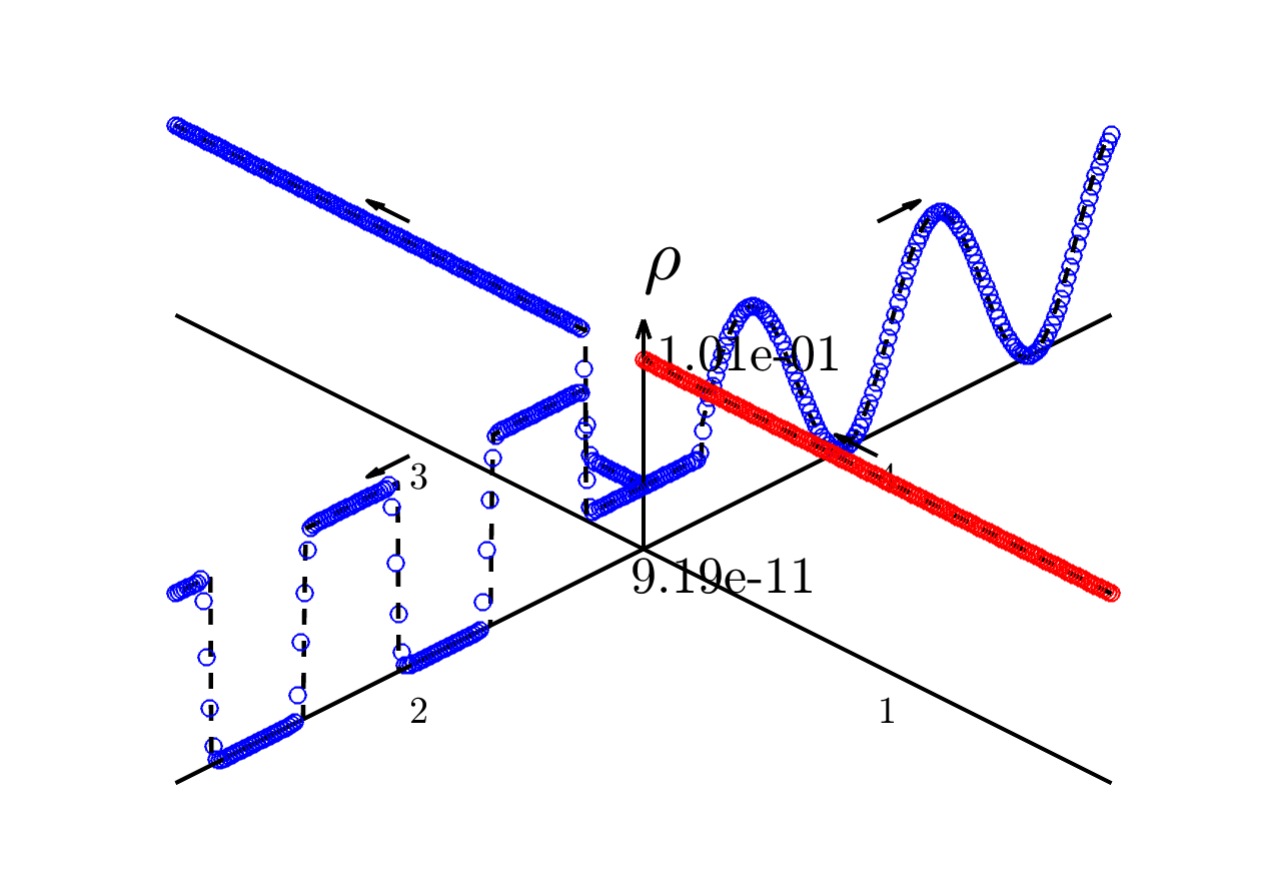}
		}
		\caption{\sf Example \ref{exam3b}, BP-OEDG scheme, $t = 0.1$, $\Delta x = 1/150$. Left: $\gamma = 0$; Right: $\gamma = 2$. Reference solutions are depicted in black dashed lines.}
		\label{fig:T5}
	\end{figure}
	
\end{exa}

\begin{exa}\label{ex:T7}
	
	In this example, we compare the HB and GHMW junction conditions. We consider a diverging junction connected to one outgoing road (Road 1) and two incoming roads (Road 2--3). The initial conditions on these roads are given by $(\rho^{(1)},\rho^{(2)},\rho^{(3)})(x,0)
	=
	(2,4,6)$
	and
	$(y^{(1)},y^{(2)},y^{(3)})(x,0)
	=
	(12,24,72)$.
	
	With respect to the GHMW networking setting, the priority vector is $\boldsymbol{\beta} = (1/2,1/2)^\top$.
	For HB network setting, please note that the priority vector $\boldsymbol{\beta}$ is determined by the traffic demands on the incoming roads.
	The BP-OEDG scheme is adopted to solve these problems, and the obtained numerical results are presented in \Cref{fig:T7}.
	As one can see, the discontinuities on Roads 2 and 3 are all well resolved without oscillations, while the contact discontinuities on Road 1 are slightly smeared. In \Cref{fig:T7b}, the discontinuity in $c$ is resolved without overshoots or undershoots, demonstrating that the proposed BP-OEDG method preserves the minimum and maximum principles of $c$.
	
	\begin{figure}[th!]
		\centering
		\begin{subfigure}[t][][t]{0.32\textwidth}
			\includegraphics[trim=25 45 55 82,clip,width = \textwidth]{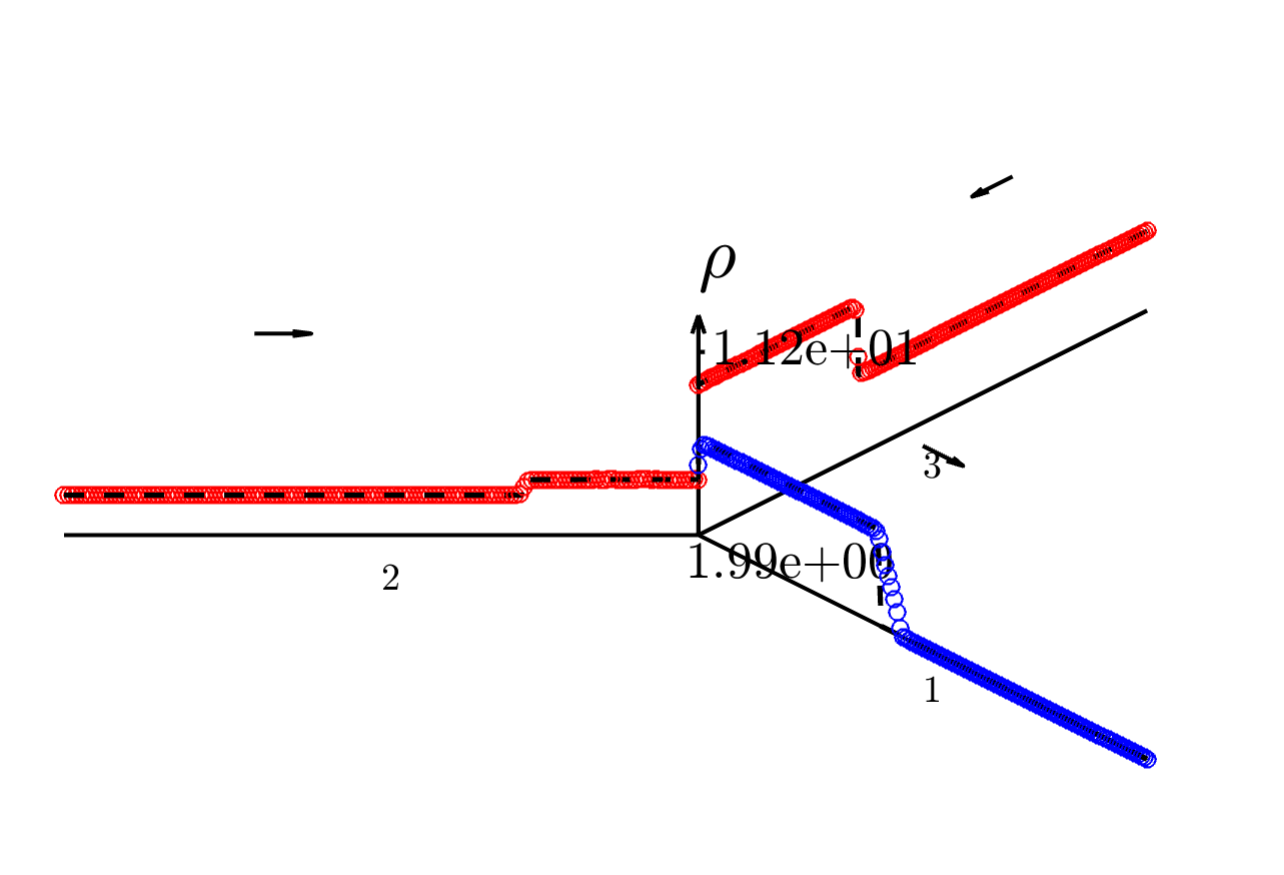}
			\caption{\sf HB network, $c \equiv 1$.}
			\label{fig:T7a}
		\end{subfigure}
		\begin{subfigure}[t][][t]{0.64\textwidth}
			\includegraphics[trim=25 45 55 82,clip,width = 0.48\textwidth]{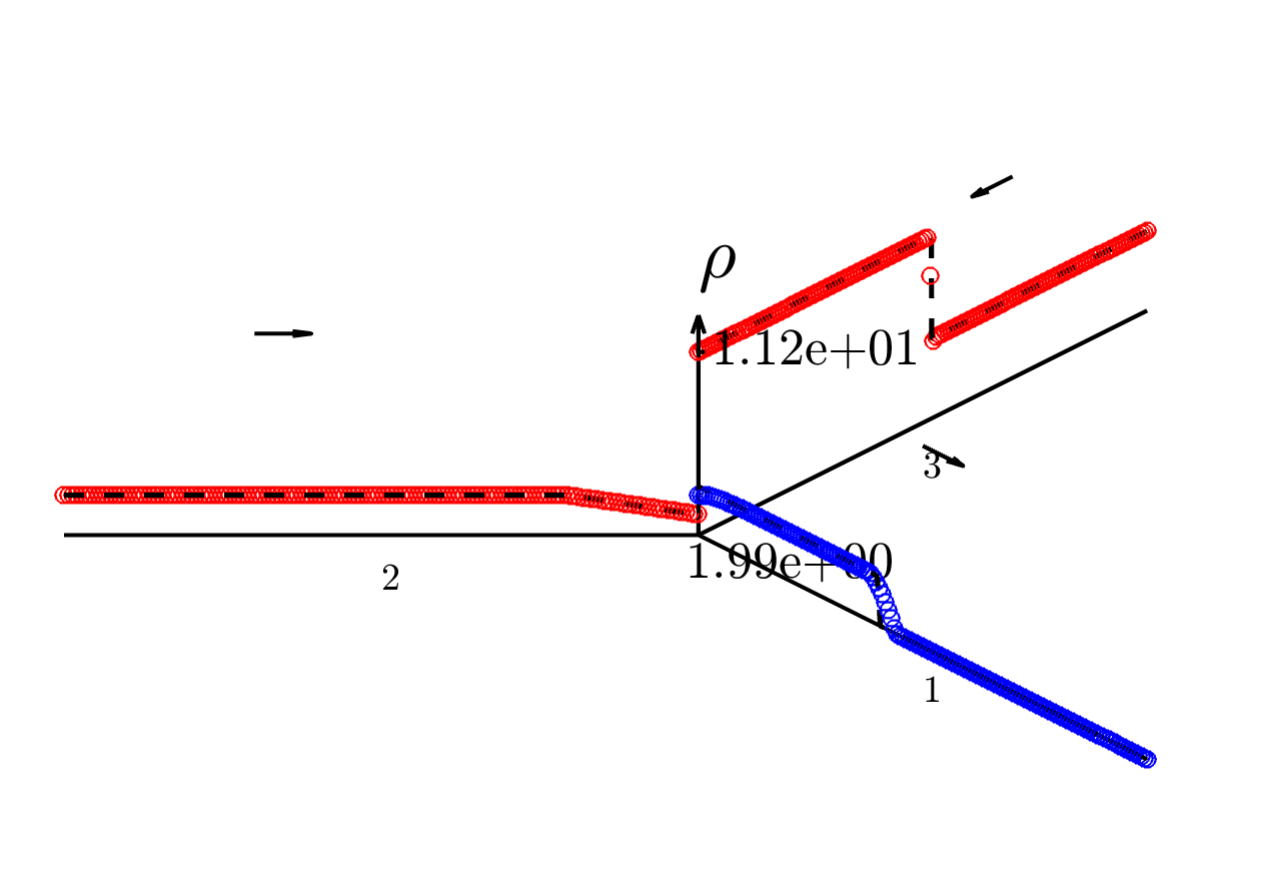}
			\includegraphics[trim=25 45 55 82,clip,width = 0.48\textwidth]{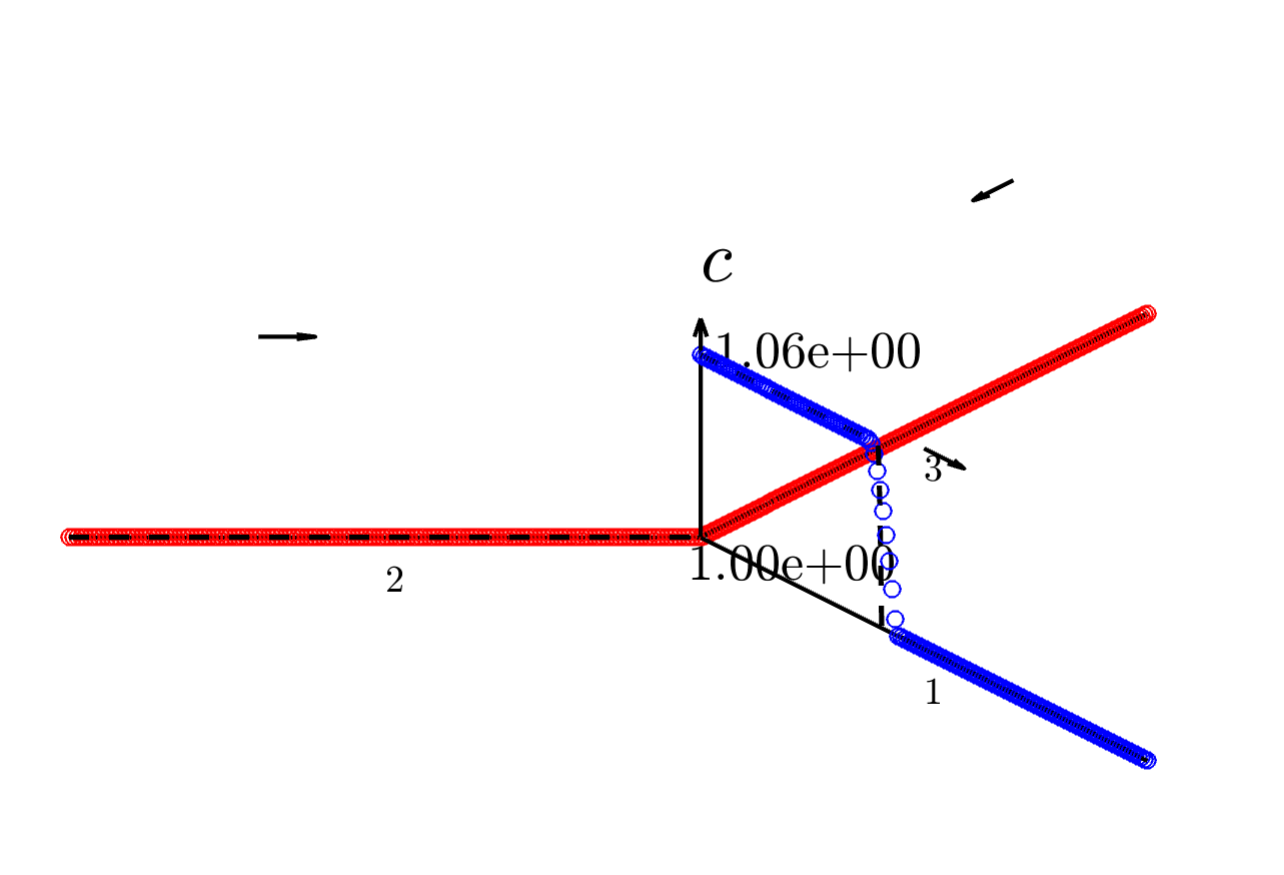}
			\caption{\sf GHMW network.}
			\label{fig:T7b}
		\end{subfigure}
		\caption{\sf Example \ref{ex:T7}, BP-OEDG scheme, $t = 0.1$, $\gamma = 2$. Reference solutions are depicted in black dashed lines.}
		\label{fig:T7}
	\end{figure}
	
\end{exa}

\begin{exa} \label{ex:T6} 
	In this example, we consider a diverging junction which is connected to one outgoing road (Road 1) and three incoming roads (Road 2--4). All four road segments have length of 1. The initial conditions on these roads are given by
	\begin{gather*}
		\rho_0^{(1)}(x) = 0.02, 
		\quad
		\rho_0^{(2)}(x) 
		= 
		\rho_0^{(3)}(x) = \rho_0^{(4)}(x) =
		\begin{dcases}
			0.005 & {\rm if} \; x \in [0, 0.1],\\
			0.01 & {\rm if} \; x \in (0.1, 0.4] \cup (0.7, 1],\\
			0.02 & {\rm otherwise}.
		\end{dcases}
		\\
		w^{(r)}(x,0) \equiv 0, \; r = 1, 2, 3, 4. 
	\end{gather*}
	The HB junction condition is adopted. Specifically, all three incoming roads contribute identical traffic influx into the outgoing Road 1. The BP-OEDG scheme is used to solve this problem, and the obtained numerical results are demonstrated in \Cref{fig:T6}.
	
	\begin{figure}[h!]
		\centering
		\centerline{
			\includegraphics[trim=55 30 61 30,clip,width = 0.5\textwidth]{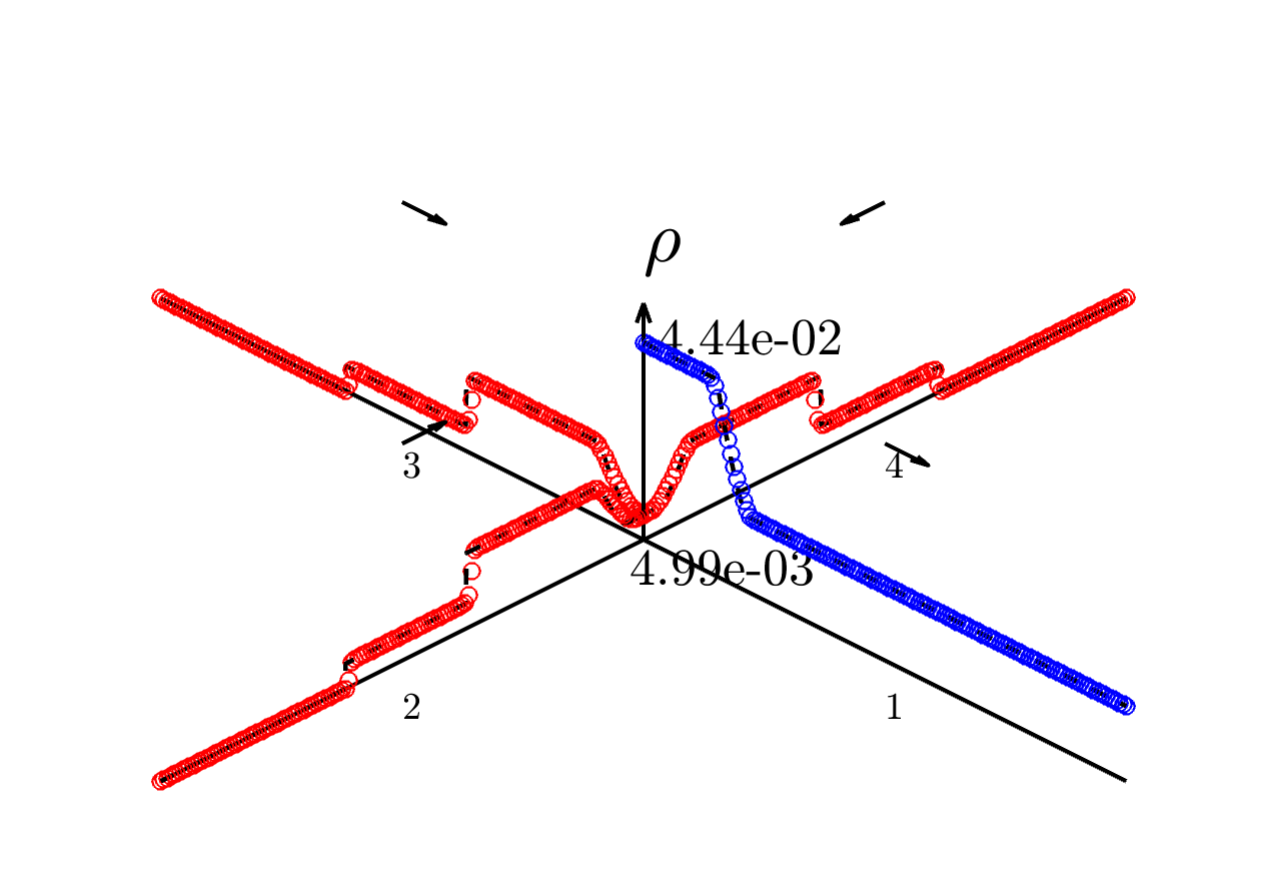}
			\includegraphics[trim=55 30 61 30,clip,width = 0.5\textwidth]{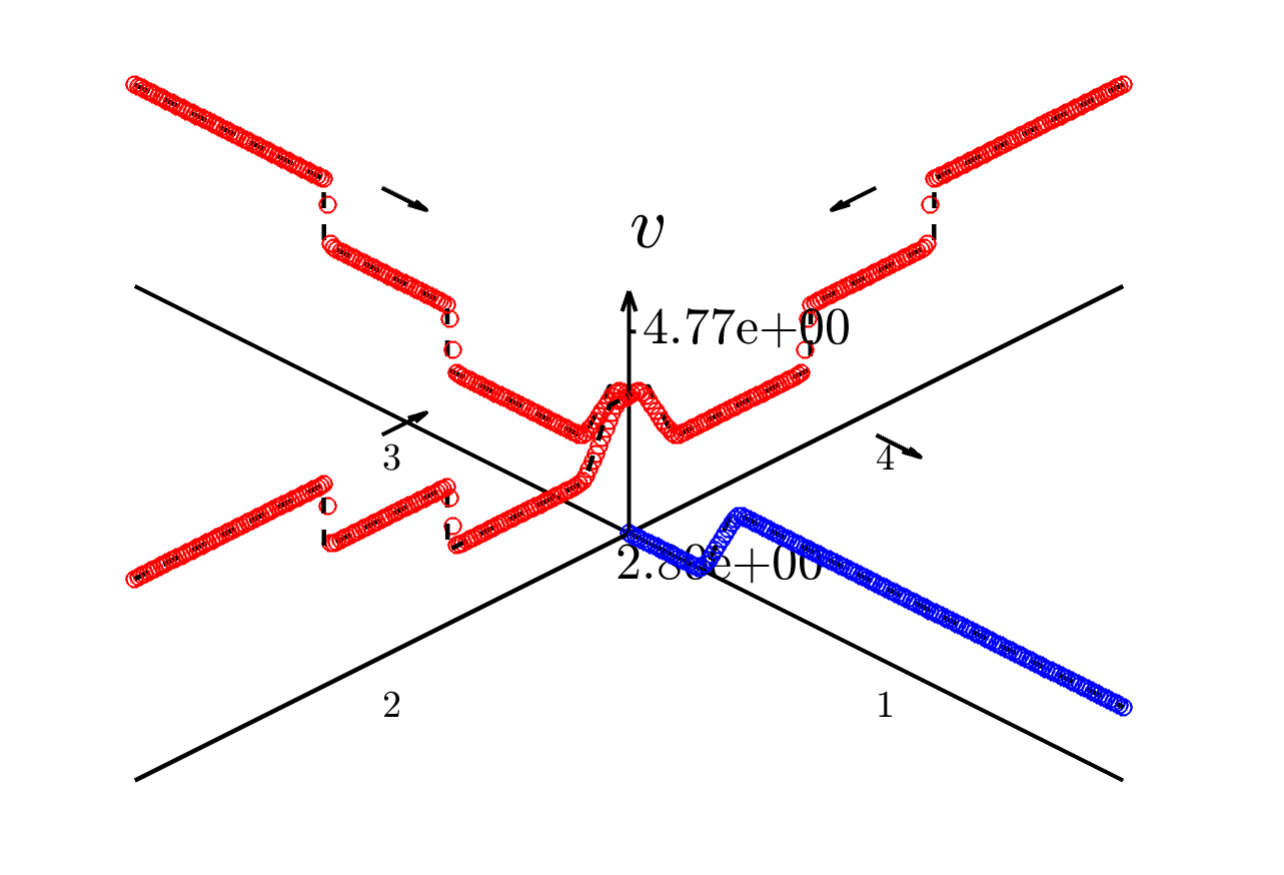}
		}
		\caption{\sf Example \ref{ex:T6}, $\gamma = 0$, BP-OEDG scheme, $t = 0.08$, $\Delta x = 1/150$. Reference solutions are depicted in black dashed lines.}
		\label{fig:T6}
	\end{figure}
\end{exa}

\begin{exa}	\label{ex:T8}
	
	In this example, we test the propose BP-OEDG scheme on a traffic simulation over a network involving seven roads and three junctions, as depicted in \Cref{fig:T8}. The initial conditions are given by 
	\begin{align*}
		& 
		\rho^{(1)}(x,0) = \rho^{(2)}(x,0) = \rho^{(6)}(x,0) = 0.1, 
		\quad 
		\rho^{(7)}(x,0) = 10^{-10}, 
		\\
		& \rho^{(3)}(x,0) = \rho^{(5)}(x,0) =
		\begin{dcases}
			0.18 + 10^{-10} 
			& 
			{\rm if} \; x \in [0, 0.1],
			\\
			0.1 \left[ 0.9 + (1-x) \cos\left(\frac{25 \pi (x - 0.1)}{0.9}\right)\right] + 10^{-10} 
			& 
			{\rm otherwise},
		\end{dcases} \\
		& 
		\rho^{(4)}(x,0) = 0.1 \big[1 + x \cos{(25 \pi x}) \big] + 10^{-10},
		\quad
		w_0^{(r)}(x) = 0.5, \; r = 1, \cdots, 7.
	\end{align*}
	The parameter $\gamma = 2$. 
	At Junction A, the traffic distribution matrix is given by
	\(
	\mathcal{A} =
	\begin{pmatrix}
		0.5 & 0.5 \\
		0.5 & 0.5 
	\end{pmatrix}
	\),
	and both Junctions A and B have the priority vector $\boldsymbol{\beta} = (0.5,0.5)^\top$. 
	At Junction C, incoming traffic flux is evenly distributed between outgoing Roads 4 and 7. 
	The proposed BP-OEDG scheme is used to conduct the simulation, and the results are presented in \Cref{fig:T8}. 
	The trigonometric traffic waves initially located on Roads 3, 4, and 5, are accurately presented by the numerical results on a relatively coarse mesh. The discontinuity initially located on Road 4 passes through Junction A, and then introduces discontinuities on Roads 1 and 3, both of which are sharply resolved without oscillations.
	On Road 7, a discontinuity with near-vacuum state on its downstream side is captured without noticeable overshoots or undershoots.
	
	\begin{figure}[th!]
		\centering
		\includegraphics[trim=0 55 0 110,clip,width = 0.7\textwidth]{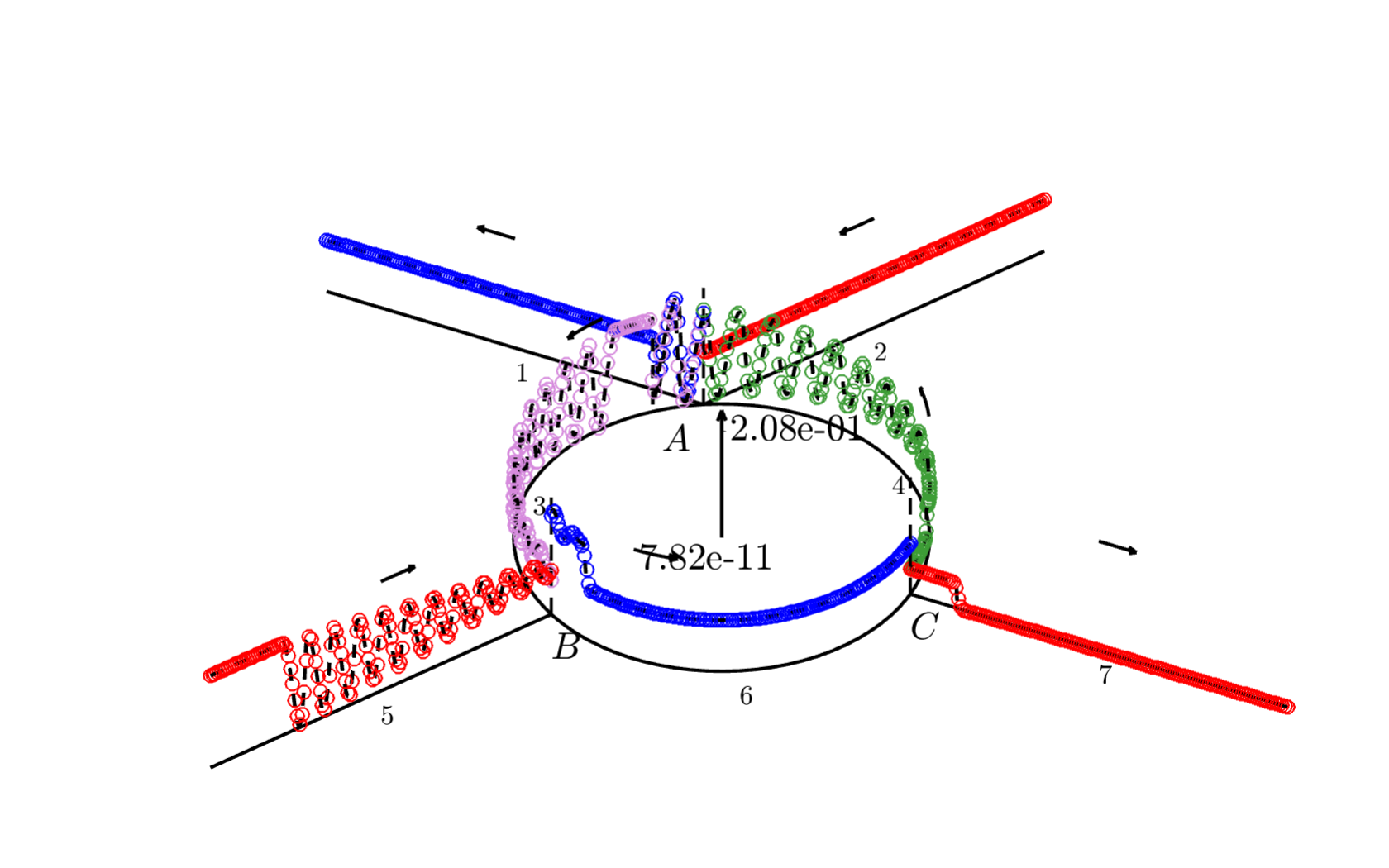}
		\caption{\sf Example \ref{ex:T8}, BP-OEDG scheme, $\Delta x = 1/150$, $t = 0.25$. Reference solutions are depicted in black dashed lines.}
		\label{fig:T8}
	\end{figure}
	
\end{exa}

\begin{exa}\label{ex:T9}
	In the last example, we consider a network with 48 roads and 25 junctions. In \Cref{fig:T9a}, a schematic diagram of the network is shown, and each road segment is assigned a color: black, blue, red or green. The initial conditions on these roads are given according to the color assigned:
	\begin{align*}
		& \rho^{\text{black}}(x,0) \equiv 10^{-10}, 
		\quad 
		\rho^{\text{blue}}(x,0) \equiv 0.1, 
		\quad 
		\rho^{\text{red}}(x,0) = 0.1 \left[ 1 + x \cos{(25 \pi x})\right] + 10^{-10}, 
		\\
		& \rho^{\text{green}}(x,0) =
		\begin{dcases}
			0.1 & {\rm if} \; x \in [0.1, 0.9], \\
			10^{-10} & {\rm otherwise}, 
		\end{dcases}
		\quad w^{(\cdot)}(x,0) \equiv 0.5 \quad \textrm{for all road segments}.
	\end{align*}
	
	The parameter $\gamma = 1$ and the HB junction condition is adopted.
	The proposed BP-OEDG scheme is adopted to simulate this test, and the resulting numerical solution on Roads 1, 26, and 44 are shown in \Cref{fig:T9b}. 
	
	
	\begin{figure}[h!]
		\centering
		\includegraphics[trim=185 25 162 18,clip,width = 0.5\textwidth]{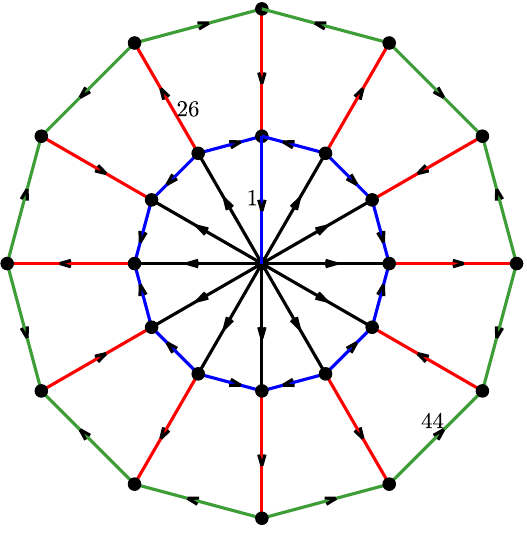}
		\caption{\sf Example \ref{ex:T9}, schematic diagram of the network.}
		\label{fig:T9a}
	\end{figure}
	
	\begin{figure}[h!]
		\centering
		\centerline{
			
			\begin{subfigure}[t][][t]{0.32\textwidth}
				\includegraphics[trim=25 15 35 15,clip,width = \textwidth]{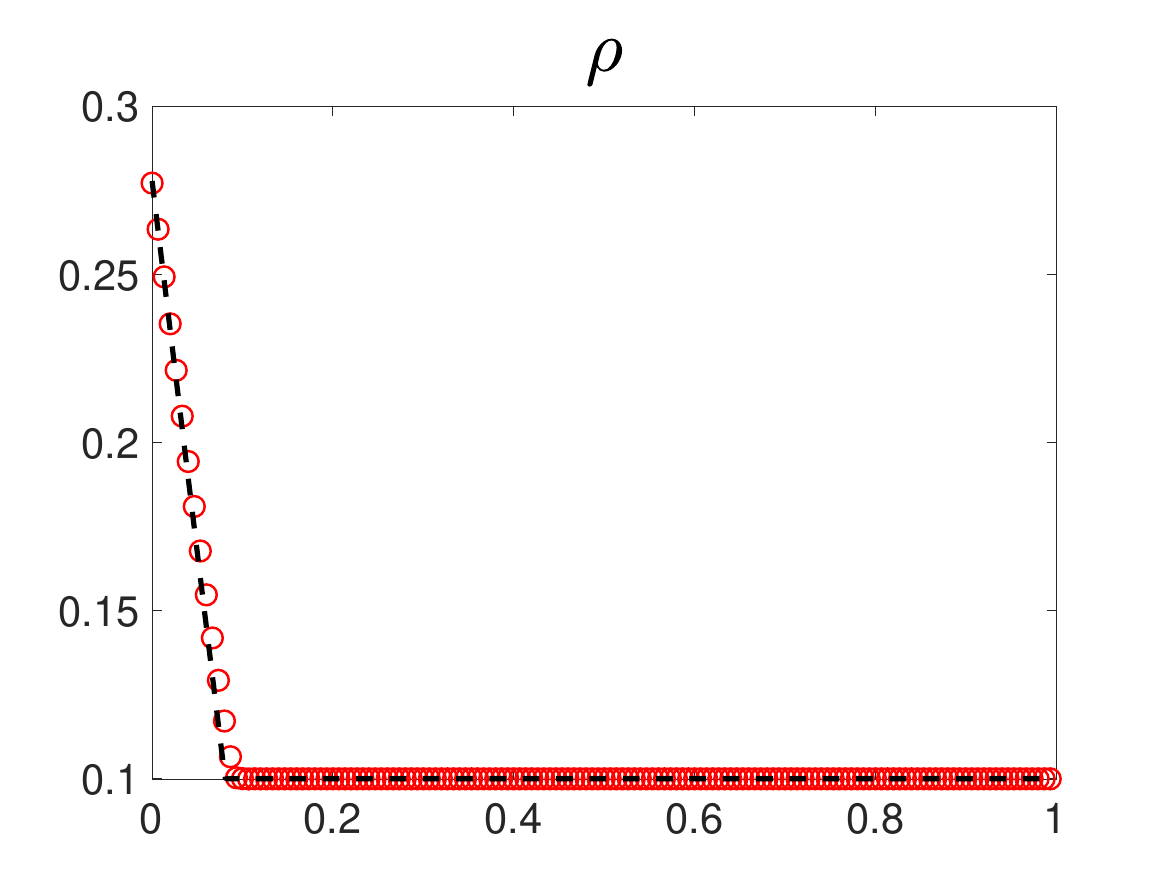}
				\caption{\sf Road 1.}
			\end{subfigure}
			
			\begin{subfigure}[t][][t]{0.32\textwidth}
				\includegraphics[trim=25 15 35 15,clip,width = \textwidth]{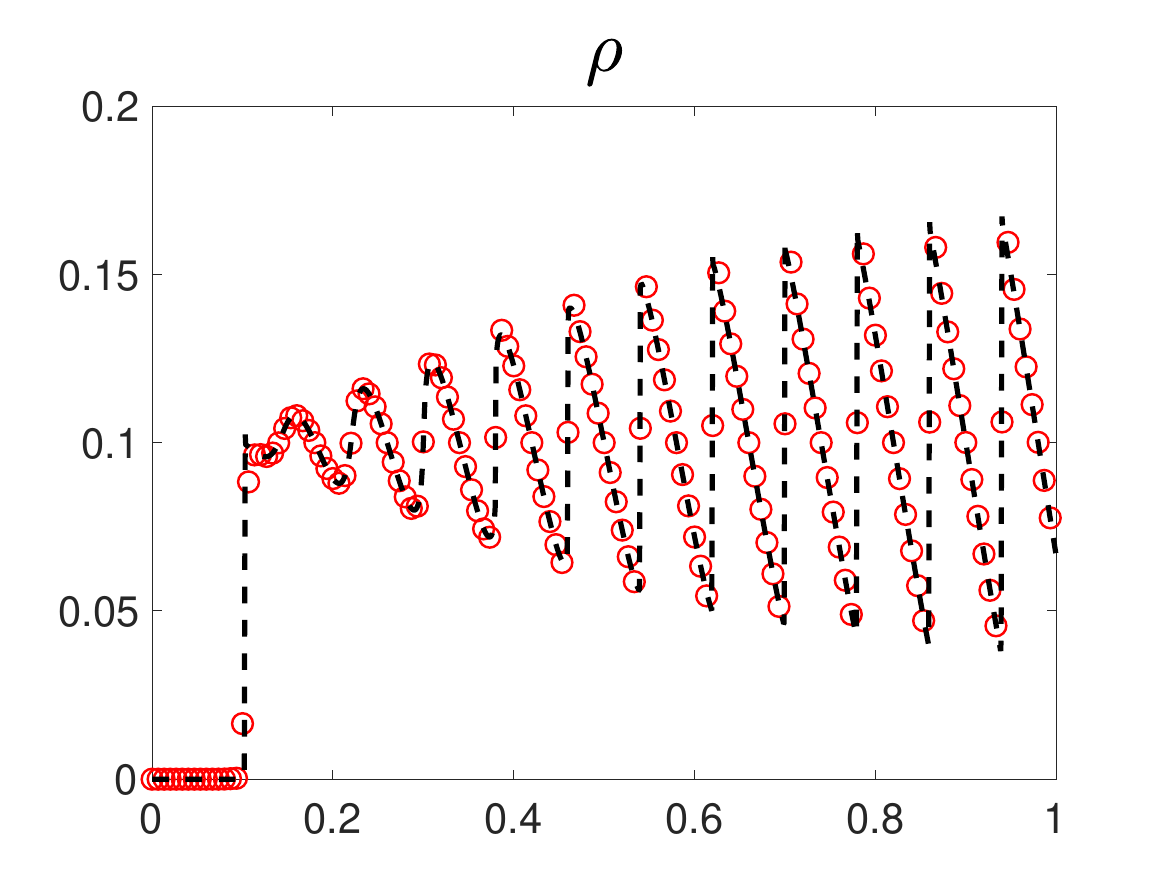}
				\caption{\sf Road 26.}
			\end{subfigure}
			
			\begin{subfigure}[t][][t]{0.32\textwidth}
				\includegraphics[trim=25 15 35 15,clip,width = \textwidth]{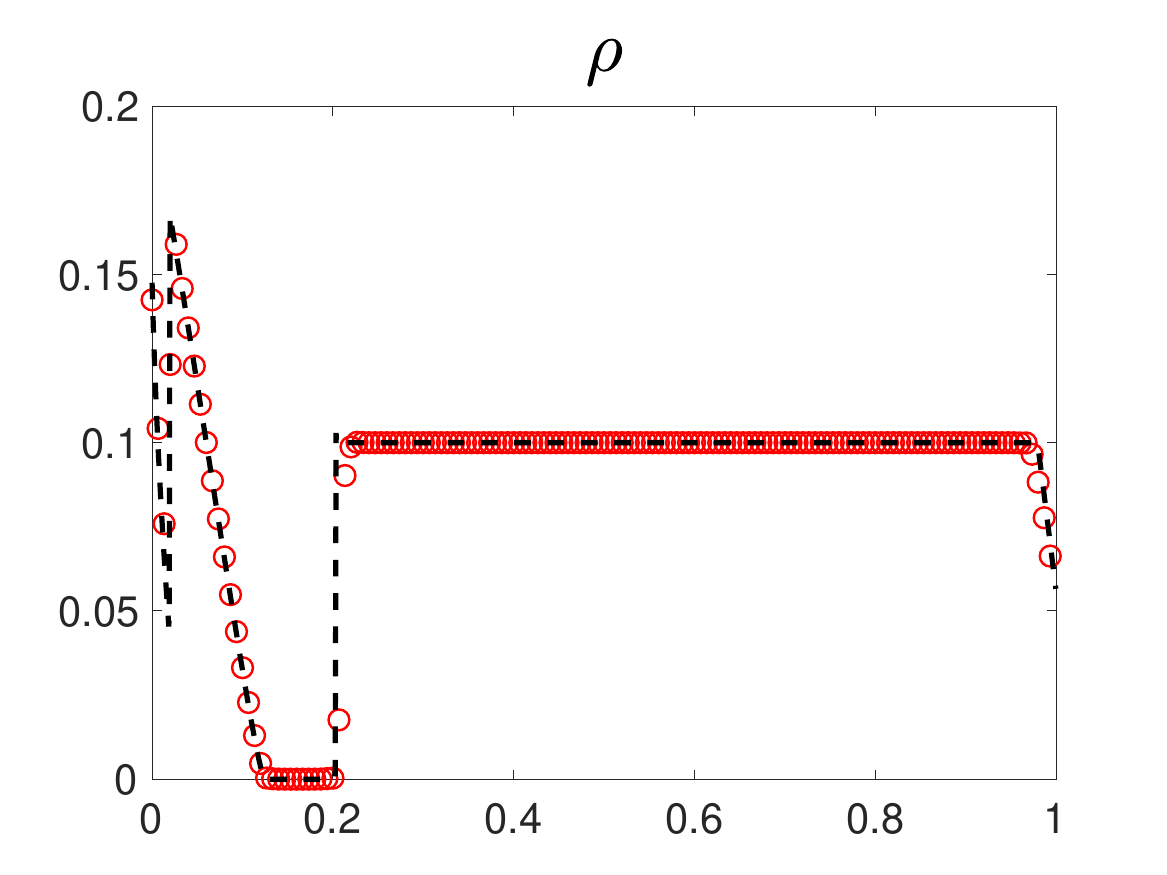}
				\caption{\sf Road 44.}
			\end{subfigure}
			
		}
		\caption{\sf \Cref{ex:T9}, BP-OEDG scheme, $t = 0.25$, $\Delta x = 1/150$. Reference solutions are depicted in black dashed lines.}
		\label{fig:T9b}
	\end{figure}
\end{exa}

\section{Conclusion}
In this paper, we have developed the bound-preserving oscillation-eliminating discontinuous Galerkin (BP-OEDG) schemes for the adapted pressure Aw--Rascle--Zhang (AP ARZ) model and its degenerate case, the original ARZ model. These schemes satisfy the maximum principles of \( w \) and \( c \), the minimum principle of \( v \), \(w\), and \(c\), as well as the positivity of \( \rho \). 
As the foundation of this work, we have systematically investigated the invariant domains induced by the aforementioned bounds and rigorously proven several key properties, including convexity and the (generalized) Lax--Friedrichs splitting properties. 
Notably, although our BP-OEDG schemes do not address the \( v \) maximum principle directly, the enforced maximum principle of \( w \) effectively maintains the \( v \) values under an alternative upper bound, which also effectively controls the velocity overshoots.
Furthermore, we have introduced the local and global approaches to estimate the upper and lower bounds of Riemann invariants $w$, $v$, and $c$. These two approaches can be conveniently employed to handle boundary conditions at the two ends of single road segment and coupling conditions at junctions in road networks.
We have presented several challenging numerical examples, including applications on a single road segment, on road networks, and with near-vacuum states, to demonstrate the BP property, robustness, and effectiveness of the proposed BP-OEDG schemes.

\bibliographystyle{model1-num-names}%
\bibliography{references}

\end{document}